\documentclass[11pt]{amsart} 

\usepackage[dvipsnames]{xcolor}
\usepackage{amsfonts,amsthm,amsmath,amssymb,etoolbox,comment,graphicx,array,tabu,commath,latexsym,environ,tikz,tikz-cd,stackengine,diagbox,caption}
\usepackage{mathrsfs}
\usepackage{soul}
\usepackage{eso-pic}
\usepackage[toc,page]{appendix}
\usepackage{enumitem}
\usepackage{mathtools}

\usepackage[colorlinks, linkcolor = black, anchorcolor = black, citecolor = black, urlcolor=black]{hyperref}

\newcommand{\fg}{\mathfrak{g}}

\newcommand{\cF}{\mathcal{F}}

\newcommand{\cH}{\mathcal{H}}

\newcommand{\cP}{\mathcal{P}}

\newcommand{\cS}{\mathcal{S}}

\newcommand{\cU}{\mathcal{U}}

\newcommand{\cW}{\mathcal{W}}

\newcommand{\cZ}{\mathcal{Z}}

\newcommand{\bD}{\mathbf{D}}

\newcommand{\bF}{\mathbf{F}}

\newcommand{\bI}{\mathbf{I}}

\newcommand{\bg}{\mathbf{g}}
\newcommand{\bx}{\mathbf{x}}

\newcommand{\fb}{\mathfrak{b}}

\newcommand{\fF}{\mathfrak{F}}

\newcommand{\fH}{\mathfrak{H}}
\newcommand{\fU}{\mathfrak{U}}

\newcommand{\fGr}{\mathfrak{Gr}}

\newcommand{\Gr}{\mathrm{Gr}}
\newcommand{\RP}{{\mathbb{RP}}}
\newcommand{\R}{\mathbb R}
\newcommand{\C}{\mathbb C}
\newcommand{\Z}{\mathbb Z}

\newcommand{\Zc}{\mathcal Z}

\newcommand{\Pc}{\mathcal P}

\newcommand{\Fb}{\mathbf F}

\newcommand{\id}{\mathrm {id}}
\newcommand{\genus}{\mathrm {genus}}

\newcommand{\inte}{\mathrm{int}}

\newcommand{\area}{\operatorname{area}}

\newcommand{\dmn}{\mathrm{dmn}}

\newcommand{\Sym}{\mathrm{Sym}}

\newcommand{\proj}{\operatorname{proj}}

\renewcommand{\Re}{\mathrm{Re}}

\renewcommand{\tilde}{\widetilde}

\newcommand{\rank}{\operatorname{rank}}

\newcommand{\spt}{\operatorname{spt}}

\newcommand{{\sing}}{\operatorname{sing}}
\newcommand{\x}{\times}

\newcommand{\ins}{\mathrm{in}}
\newcommand{\out}{\mathrm{out}}

\newcommand{\ord}{\mathrm{ord}}

\newcommand{\bz}{\mathbf z}

\newcommand{\GS}{\mathcal{S}}

\newcommand{\odd}{\mathrm{odd}}

\newcommand{\balpha}{\boldsymbol{\alpha}}

\usepackage[style=alphabetic, backend=biber, sorting=nyt, url=false ]{biblatex}
\title[Minimal surfaces with arbitrary genus in 3-spheres]{Minimal surfaces with arbitrary genus in 3-spheres of positive Ricci curvature}

\author{Adrian Chun-Pong Chu} 

\address{Cornell University, Ithaca, NY 14853, USA}
\email{cc2938@cornell.edu}

\date{\today}

\numberwithin{equation}{section}
\newtheorem{thm}{Theorem}[section]
\newtheorem{cor}[thm]{Corollary}
\newtheorem{prop}[thm]{Proposition}
\newtheorem{lem}[thm]{Lemma}

\theoremstyle{definition}
\newtheorem{defn}[thm]{Definition}

\newtheorem{exmp}[thm]{Example}
\newtheorem{rmk}[thm]{Remark}

\newtheorem{oqn}[thm]{Question}

\newtheorem*{question*}{Question}

\begin{document}

\begin{abstract}
We describe some topological structure in the set of all surfaces with finitely many singularities in the 3-sphere. 

    As an application, we prove that every Riemannian 3-sphere of positive Ricci curvature contains, for every $g$, a genus $g$ embedded minimal surface with area at most twice the first Simon-Smith width of the ambient 3-sphere.
\end{abstract}
\maketitle
\setcounter{tocdepth}{1}

\section{Introduction}

In past decades,  moduli spaces of    embedded surfaces   in   3-manifolds are much studied. A celebrated result of A. Hatcher \cite{Hat83} says the space of all smooth embedded spheres in $S^3$ is homotopy equivalent to $\RP^3$. Based on this, Johnson-McCullough showed that the space of all  unknotted tori in $S^3$ is homotopy equivalent to $\RP^2\x\RP^2$  \cite{JM13}, and Ketover-Liokumovich proved  analogous results for lens spaces \cite{ketoverLiokumovich2025smaleLensSpace}.  We also remark that, the  space of all embedded surfaces in a given 3-manifold $M$ is closely related to the Smale conjecture, which concerns the relationship between the space $\textrm{Diff}(M)$ of diffeomorphisms and the  subspace $\textrm{Isom}(M)$  of isometries. This, in turn, is a rich topic with fruitful progress in  recent decades \cite{ivanov1984SpaceDiffeo,mcculloughRubinstein1997Smale,gabai2001smale,hongKalliongisMcCulloughRubinstein2012diffeomorphisms,bamlerKleiner2019ricciContractibility,bamlerKleiner2023ricciDiffeo,ketoverLiokumovich2025smaleLensSpace}.

In this paper, we are going to  study the set of all embedded surfaces that {\it possibly have finitely many singularities} in the 3-sphere $S^3$. Besides the pure topological interest, there is actually an inevitable need to study surfaces with singularities, from a geometric analysis perspective. Heuristically, from a Morse theoretic point of view, the problem of constructing minimal surfaces in a 3-manifold $M$ is inherently related to   the topological structure of the set of all surfaces in $M$: This is because  minimal surfaces are by definition critical points of the area functional, defined on the set of all surfaces. However, under the gradient flow of the area functional, i.e.  mean curvature flow, surfaces may develop singularities. 

From a min-max point of view, which would be the main focus of this paper, consideration of singularities is also natural, because in a  min-max procedure the minimal surfaces obtained may have  lower genus than the members of the initial family one uses to apply min-max.

We now describe the setting. Let $M$ be    a closed Riemannian 3-manifold. We denote by $\cS(M)$   the set of all surfaces embedded in $M$ that possibly have finitely many singularities: For some technical reasons, it would be desirable to focus  on surfaces that  separate $M$ into two regions.  There is actually a natural way to define the notion of genus for surfaces with singularities. Now, let $\cS_{\leq g}(M)\subset \cS (M)$ be the subset of elements of  genus $\leq g$. Readers  may refer to \S \ref{sect:prelim} for the precise definitions.  

A main theorem of this paper is the following  topological result, which basically says, for every $g$,  there is some non-trivial relative structure for the pair $(\cS_{\leq g}(S^3),\cS_{\leq g-1}(S^3))$.

\begin{thm}\label{thm:mainTopo}
    In the $3$-sphere $S^3$, for every positive integer $g$, there exists a map ``continuous" in the sense of Definition \ref{def:Simon_Smith_family},  $$\Psi:X\to\cS_{\leq g}(S^3),$$
that cannot be deformed via ``pinch-off processes" to become a map into $\cS_{\leq g-1}(S^3)$.

\end{thm}
The notion of   ``pinch-off process" is stated  in Definition \ref{defn:deformViaPinchoff}. Essentially, it refers to a deformation process for elements in $\cS(M)$  that allows any combinations of the following three processes: (1) isotopy, (2) neck-pinch surgery, and (3)  shrinking some connected components into points: See Figure \ref{fig:pinchOff}. In particular,  pinch-off process  is genus-non-increasing. Regarding the notion of ``continuity" for the map $\Psi$, roughly speaking,  members of $\Psi$ are required to vary smoothly, except at finitely many  points.

In our proof, we explicitly construct such a map $\Psi$ whose  domain is $(2g+3)$-dimensional. This is reminiscent of  the fact that in the unit 3-sphere, the  Lawson minimal surface $\xi_{g,1}$  of genus $g$  has Morse index $2g+3$ \cite{Law70}, as proven by Kapouleas-Wiygul \cite{kapouleasWiygul2020LawsonIndex}.

As a consequence, using   Theorem 1.3 of  \cite{ChuLiWang2025GenusTwoI} by Chu-Li-Wang (see Theorem \ref{thm:TopoMinMax} below), which  builds on numerous  foundational results  in min-max theory    \cite{Smith82, CD03, DeLellisPellandini10,   Ket19, Zho20, MN21,WangZhou23FourMinimalSpheres}, we  prove the following result. First, given any Riemannian 3-sphere $S^3$, denote by $\cP$   the set of all 1-sweepouts by smooth embedded spheres that are continuous in the smooth topology. We define the  {\it first Simon-Smith width} (also called the {\it first spherical area width}) of $S^3$ by
$$\sigma_1(S^3):=\inf_{\Phi\in \cP}\max_{x\in \dmn(\Phi)}\area(\Phi(x)).$$ Note, if the ambient metric has positive Ricci curvature, then $\sigma_1(S^3)$ is equal to area of the least-area  embedded minimal sphere in $S^3$ (see \S \ref{sect:pSweepout}).

\begin{thm}\label{thm:mainMinimal}
    Every Riemannian $3$-sphere $S^3$ of positive Ricci curvature contains, for every $g$,  some orientable, embedded, genus $g$ minimal surface with area at most $2\sigma_1(S^3)$.
\end{thm}

   \begin{figure}
        \centering
        \makebox[\textwidth][c]{\includegraphics[width=2.5in]{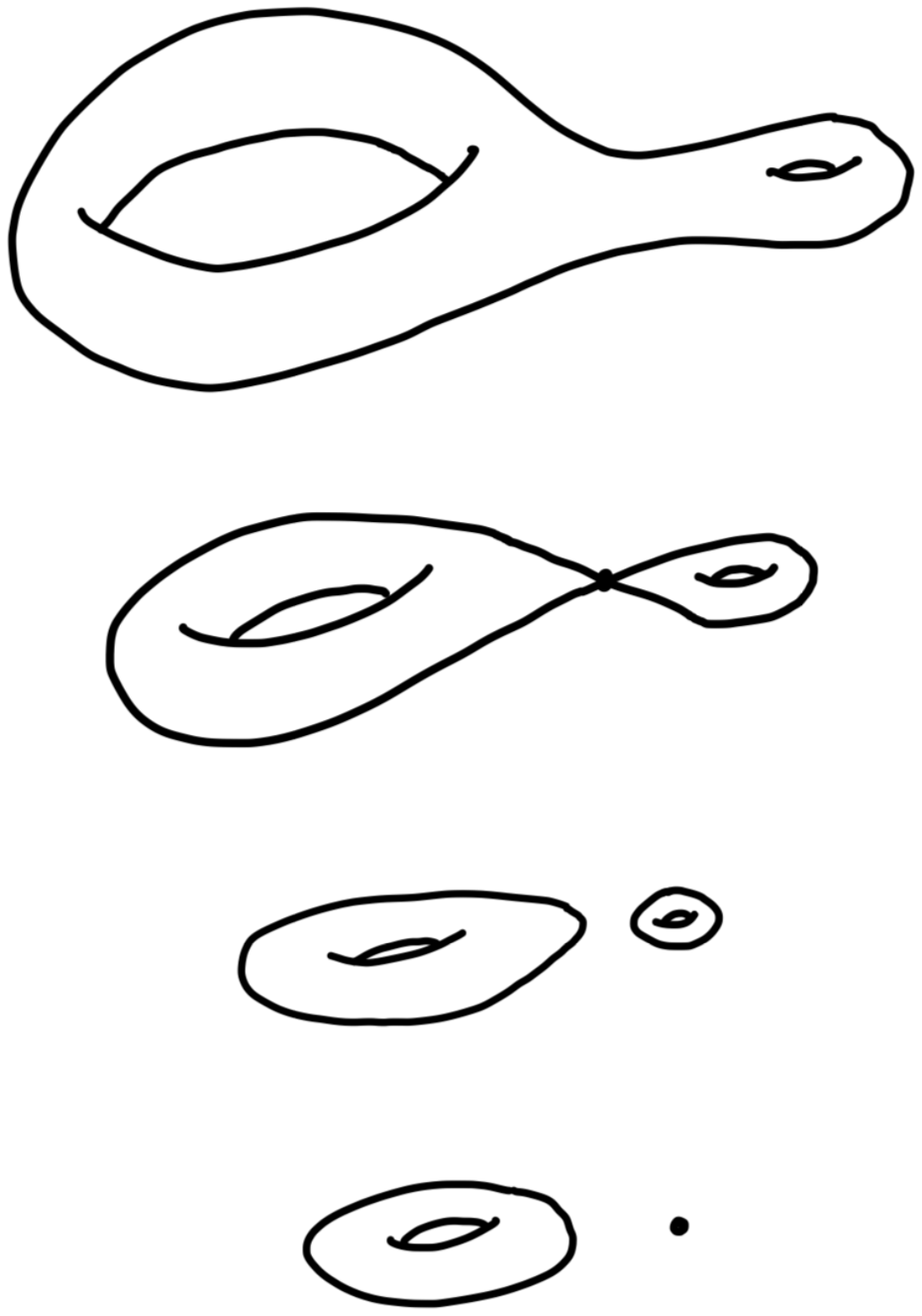}}
        \caption{This is an example of pinch-off process with two spacetime singularities. It has one neck-pinch surgery (the second picture shows a double cone), and one connected component shrunk to a point.}
        \label{fig:pinchOff}
    \end{figure}

We note that, in the case of the unit 3-sphere, the area bound in Theorem \ref{thm:mainMinimal} is $8\pi$. By Heller-Heller-Traizet's work \cite{HellerHellerTraizet2023AreaLawson}, the area of the Lawson minimal surfaces $\xi_{g,1}$ approaches $8\pi$ from below as $g\to\infty$. Thus,  assuming the well-known conjecture  that $\xi_{g,1}$ minimizes the area among embedded genus $g$ minimal surfaces (see  \cite{kusner1989comparison} by R. Kusner), the area bound in Theorem \ref{thm:mainMinimal} is sharp in the case of the unit 3-sphere.

Our motivation behind Theorem \ref{thm:mainMinimal} stems from S.-T. Yau's conjecture that every Riemannian 3-sphere has at least 4 embedded minimal spheres \cite{Yau82}, and also B. White's conjecture that every Riemannian 3-sphere has at least 5 embedded minimal tori \cite{White89}.  Yau's conjecture was solved in the case of bumpy metric and positive Ricci curvature by Wang-Zhou \cite{WangZhou23FourMinimalSpheres}. In \cite{chuLi2024fiveTori},   Y. Li and the author   proved  B. White's conjecture  in the case where the ambient metric has positive Ricci curvature. See also Xingzhe Li and Zhichao Wang's  work on the existence of  minimal tori \cite{LiWang2024NineTori}. Very recently,  Yangyang Li, Zhihan Wang, and the author also proved the existence of genus 2 minimal surfaces in 3-spheres with positive Ricci curvature \cite{ChuLiWang2025GenusTwoI}.

For more previous works on the construction of geodesics or minimal surfaces with controlled topological type, readers may refer to  \cite{LS47,Str84, GJ86,Grayson89, white1991space, Zho16, HK19, bettiolPiccione2023bifurcationsCliffordTorus, HK23, Ko23a, Ko23b,bettiolPiccione2024nonplanarMinSpheres}.  
During the preparation of this article, the author was  informed that D. Ketover also independently obtained results on the existence of minimal surfaces of arbitrary genus in 3-spheres of positive Ricci curvature.

\subsection{Further directions}
In view of Theorem \ref{thm:mainMinimal}, the following question naturally arises.
\begin{oqn}\label{oqn:ng}
    What is the optimal number $\mathfrak n_g(S^3)$ for which every 3-sphere of positive Ricci curvature must have at least $\mathfrak n_g(S^3)$ minimal surfaces of genus $g$?
\end{oqn} Note, $\mathfrak n_0(S^3)=4$ by   \cite{WangZhou23FourMinimalSpheres} and  $\frak n_1(S^3)\geq 5$ by \cite{chuLi2024fiveTori}. Obtaining an asymptotic formula for $\frak n_g(S^3)$ as $g\to\infty$ would already be very interesting. One can also ask an analogous  question with the positive Ricci curvature assumption removed, or replaced with the bumpy metric assumption. Of course, we may also consider such questions on arbitrary $3$-manifolds.

To obtain a lower bound for   $\frak n_g(S^3)$, it is expected that we should study, in a suitable sense, the relative cohomology ring structure of the pair $(\cS_{\leq g}(S^3),  \cS_{\leq g-1}(S^3))$ (see  \cite{LiWang2024NineTori} for the case $g=1$). In particular, the  following problem would be a crucial intermediate step:
\begin{oqn}
    Construct homotopically non-trivial smooth families of smooth genus $g$ surfaces in $S^3$.
\end{oqn}
 On the topological side, we can also ask, if a map $\Phi$ into $\cS_{\leq g}(S^3)$ cannot be deformed via pinch-off processes to become a map into $\cS_{\leq g-1}(S^3)$, what is the minimum possible dimension of the domain of $\Phi$? The author conjectures that the number $2g-3$, as provided by Theorem \ref{thm:mainTopo}, is optimal.

As for upper bounds for $\frak n_g(S^3)$, the situation is even more obscure. Is it possible to obtain the {\it exact} value of $\frak n_g(S^3)$ in terms of topological information on the pair $(\cS_{\leq g}(S^3),\cS_{\leq g-1}(S^3))$? Here is another way to view this problem. In finite dimensional Morse theory,  any Morse function $f:M\to\R$ must have at least $b_k(M)$ critical points of Morse index $k$, where $b_k(M)$ is the $k$-th Betti number of $M$. If a function $f$ has {\it exactly} $b_k(M)$ critical points of Morse index $k$, then $f$ is called a {\it perfect} Morse function. So, under this analogy,  here we are really just asking: 
\begin{oqn}\label{oqn:perfect}
    Assuming $S^3$ has positive Ricci curvature (or a   bumpy metric with positive Ricci curvature), is the area functional on the pair $(\cS_{\leq g}(S^3),\cS_{\leq g-1}(S^3))$ in some sense ``perfect"?
\end{oqn}
 Ultimately, this  could be related  to classifying  minimal surfaces in the {\it  standard $3$-sphere}. Namely, one possible strategy for classification  is to rely on the topological structure of $\cS_{\leq g}(S^3)$.  Then Question \ref{oqn:ng} is related to the existence side, while Question \ref{oqn:perfect}   the uniqueness side.

\subsection{Sketch of proof for Theorem \ref{thm:mainTopo}} 
Fix $g\geq 1$. Let $S^3\subset\R^4$ be the unit sphere, and $B^n$ the closed unit $n$-ball. We are going to construct a map $$\Psi:\RP^5\x B^{2g-2}\to\cS_{\leq g}(S^3)$$ that cannot be deformed via pinch-off processes to become a map into $\cS_{\leq g-1}(S^3)$ 

\subsubsection{A $(2g+3)$-parameter family $\Psi$}\label{sect:introPsi}
  We first define in $S^3$ a 5-parameter family $\Phi_5$ of surfaces. For each $a=[a_0:a_1:...:a_5]\in\RP^5$, we let  $\Phi_5(a)$ be the zero set (in $S^3$) given by
\begin{equation}\label{eq:introPhi5}
    a_0+a_1x_1+a_2x_2+a_3x_3+a_4x_4+a_5x_1x_2=0.
\end{equation}
Crucially, this is a {\it $5$-sweepout} in the sense of Almgren-Pitts min-max theory.
However, this   family is  not  allowed in our setting,  as it contains members with infinitely many singularities. Namely, for any 
$$a\in A_{\sing}:=\{[a_1a_2:a_1:a_2:0:0:1]:a_1^2+a_2^2<1\},$$
$\Phi_5(a)$ is given by $(x_1+a_2)(x_2+a_1)=0$, which is a union of two (not necessarily equatorial)  spheres  intersecting perpendicularly. So we need to desingularize each of them, by replacing a neighborhood of the circle of intersection $C(a_1,a_2)$ with $\leq g+1$ handles. In fact, for each $a\in A_{\sing}$ we will find a $B^{2g+1}$-family of ways to desingularize $\Phi_5(a)$ (though a 3-parameter subfamily of those desingularizations already appeared in $\Phi_5$: e.g. $\{x_1x_2+a_0+a_3x_3+a_4x_4=0\}$ desingularizes $\{x_1x_2=0\}$ whenever $a_0,a_3,a_4$ are small). This gives us a map $$\Psi:\RP^5\x B^{2g-2}\to\cS_{\leq g}(S^3).$$
Geometrically, the $2g+1$ parameters come from varying the sizes and locations of the handles inserted, which may pinch or merge with each other: See Figure \ref{fig:oneSurf}. 

Let us more precisely define    $\Psi$. Let $O_1$ denote the point $[0:0:...:0:1]\in\RP^5$.  We first explain how to desingularize $\Phi_5(O_1)$, which is defined by $x_1x_2=0$. Since we only need to desingularize the  circle of intersection $\{x_1=x_2=0\}$, let us focus on a small $\delta$-tabular neighborhood $N_1$ of it. Naturally $N_1\cong \bD_\delta\x S^1$, where $\bD_\delta\subset \R^2$ denotes a tiny disc of radius $\delta$, and $S^1:=\R/2\pi\Z$ represents the direction given by the circle of intersection. Thus, on $N_1$, we may use the coordinates  $(x_1,x_2)$ for $\bD_\delta$ and $\alpha$ for $S^1$.  Note the coordinates $x_1,x_2$ here are just the same as the original $x_1,x_2$ for $\R^4$.

For any $b=(b_2,b'_2,...,b_g,b'_g)\in B^{2g-2}$, we define the trigonometric polynomial$$F_b(\alpha):=b_2\cos 2\alpha+b'_2\sin 2\alpha+...+b_g\cos g\alpha+b'_g\sin g\alpha+(1-\|b\|)\cos(g+1)\alpha,$$for $\alpha\in S^1:=\R/2\pi\Z$. Let us first define the member $\Psi(O_1,b)$ within the thin tube $N_1$: We define it to be the zero set given by 
\begin{equation}\label{eq:introX1X2Fb}
    x_1x_2+\epsilon F_b(\alpha)=0
\end{equation}under the coordinate $(x_1,x_2,\alpha)$, where $\epsilon>0$ is some small constant. Geometrically, it looks like a tower of handles as shown in  Figure \ref{fig:oneSurf}. Now, to obtain the whole surface $\Psi(O_1,b)$, we just remove the part $\Phi_5(O_1)\cap N_1$ from $\Phi_5(O_1)$, and glue in this tower of handles. 
 
Note, in Figure \ref{fig:oneSurf}, the intersection points of the surface with the vertical axis (representing $S^1$) are given by the roots of $F_b(\alpha)=0$. Since the trigonometric polynomial $F_b$ is of degree $\leq g+1$, there are at most $2g+2$ intersection points, which easily implies that $\Psi(O_1,b)$ has genus $\leq g$ for every $b$. In fact, $\Psi(O_1,0)$ has genus exactly $g$, while $\Psi(O_1,b)$ has genus  $\leq g-1$ for any $b\in \partial B^{2g-2}$.

  \begin{figure}
        \centering
        \makebox[\textwidth][c]{\includegraphics[width=4in]{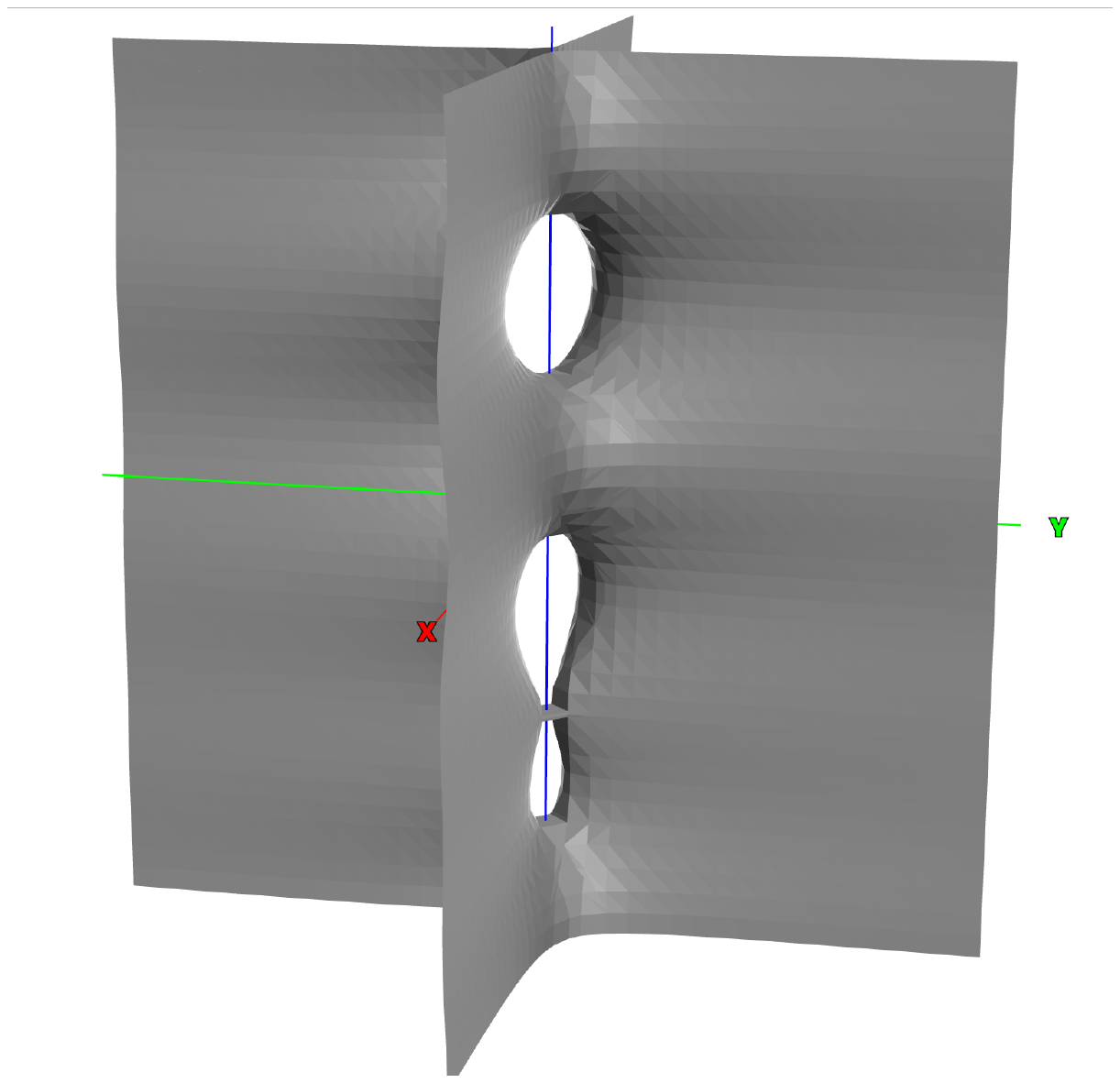}}
        \caption{This shows one  member of $\Psi$ which  desingularizes the zero set $(x_1+a_2)(x_2+a_1)=0$ using handles of different sizes and locations. The vertical axis represents the circle of intersection $C(a_1,a_2):=\{x_1+a_2=x_2+a_1=0\}$}
        \label{fig:oneSurf}
    \end{figure}

Now, we modify $\Phi_5(a)$ for any $a\in \RP^5$ close to $A_{\sing}$. Since changing $a_1,a_2$ essentially corresponds to translating the zero set in $x_1,x_2$-directions, let us just focus on those $a$ close to $O_1$ with  $a_1=a_2=0$: This set of $a$ lies in a small 3-ball centered at $O_1$. Namely, as $a$ moves from $O_1$ towards the boundary of this 3-ball, we just interpolate from the function (\ref{eq:introX1X2Fb}) to the one in (\ref{eq:introPhi5}), and take the zero sets.

\subsubsection{Three levels of interpolation arguments}\label{sect:threeLevel} The proof of Theorem \ref{thm:mainTopo} consists of three levels of interpolation arguments:
\begin{enumerate}
    \item To show that $\Psi$ cannot be deformed such that the area of all its members go to zero.
    \item To show that $\Psi$ cannot be deformed into $\cS_0{(S^3)}$.
    \item To show that $\Psi$ cannot be deformed into $\cS_{\leq g-1}{(S^3)}$.
\end{enumerate}

To achieve (1), we just need to note that the family $\Psi$ contains some $1$-parameter  subfamily $\Psi|_\gamma$  that is a {\it $1$-sweepout}, i.e. the surfaces in $\Psi|_\gamma$ sweep out the whole $S^3$ (an odd number of times). In min-max theory, this feature guarantees the family can produce at least one minimal surface \cite{Alm65,Pit81, Smith82,CD03}.

To achieve (2), we note that for every $b\in B^{2g-2}$ the subfamily  $\Psi|_{\RP^5\x\{b\}}$ is close to $\Phi_5$, and thus must be a 5-sweepout. But since   there does not exist 5-sweepout  consisted of solely genus 0 elements, $\Psi|_{\RP^5\x \{b\}}$ cannot be deformed into $\cS_0(S^3)$. On the min-max side, this feature allows one to produce some minimal surface of {\it non-zero genus} \cite{MN14,chuLi2024fiveTori}.

To achieve (3), we use the observation in the last paragraph, together with  some elementary topology, to show the following:   For any family $\Psi'$ obtained from $\Psi $ via pinch-off processes, there must  be some $(2g-2)$-dimensional subset $D $ in $ Y:=\RP^5\x B^{2g-2}$, with $\partial D\subset \partial Y$, such that:
\begin{itemize}
\item $[D]=[\{O_1\}\x B^{2g-2}]$  in $H_{2g-2}(Y,\partial Y;\Z_2)$. In particular,   $D$ and $\RP^5\x\{0\}$ have intersection number 1 (mod 2) in $Y$.
    \item All members of $\Psi'|_D$ have genus $\geq 1$
\end{itemize}
As a consequence, we can guarantee:
\begin{itemize}
    \item Since pinch-off process is genus-non-increasing, all members in the original family $\Psi|_D$ also have genus $\geq 1$
    \item 
Recall every member of $\Psi|_{\partial Y}$ has genus $\leq g-1$. So, every member of $\Psi|_{\partial D}$ must have genus in the range $[1,g-1]$.
\end{itemize}

On the other hand, we will examine the homology classes of members of $\Psi|_D$ and conclude, in fact, it is impossible to deform $\Psi|_D$ into $\Psi'|_D$ such  that all members of $\Psi'|_D$ have genus {\it in the range $[1,g-1]$}, for that would mean too many homology classes have pinched during the pinch-off processes. This shows $\Psi'|_D$ must still have some genus $g$ member, and so $\Psi'$ cannot be a map into $\cS_{\leq g-1}(S^3)$, as desired. 

Before we further explain this method,   let us describe the  heuristics by comparing with mean curvature flow,  which is topologically like a pinch-off process, as both processes are genus-non-increasing.

\subsubsection{A comparison with mean curvature flow}
In \cite{chuSun2023genus}, A. Sun and the author showed that for certain type of one-parameter family of initial conditions, as illustrated by Figure \ref{fig:mcf}, if we run mean curvature flow to all its members,  some member must develop a {\it genus one singularity}. The moral is, by an interpolation argument on the homology classes, we can guarantee that {\it some} member in the initial family must  behave in a certain desirable way under the flow. For example, in  \cite{chuSun2023genus}, we showed for some member, both the inward neck and outward neck  would pinch at the same time. As for the current paper, we run a high parameter interpolation argument to show that, for some member in $\Psi|_D$ (and thus $\Psi$), {\it no homology class can pinch}. Thus, $\Psi'|_D$ must still contain some member of genus $g$.

   \begin{figure}[h]
        \centering
        \makebox[\textwidth][c]{\includegraphics[width=\textwidth]{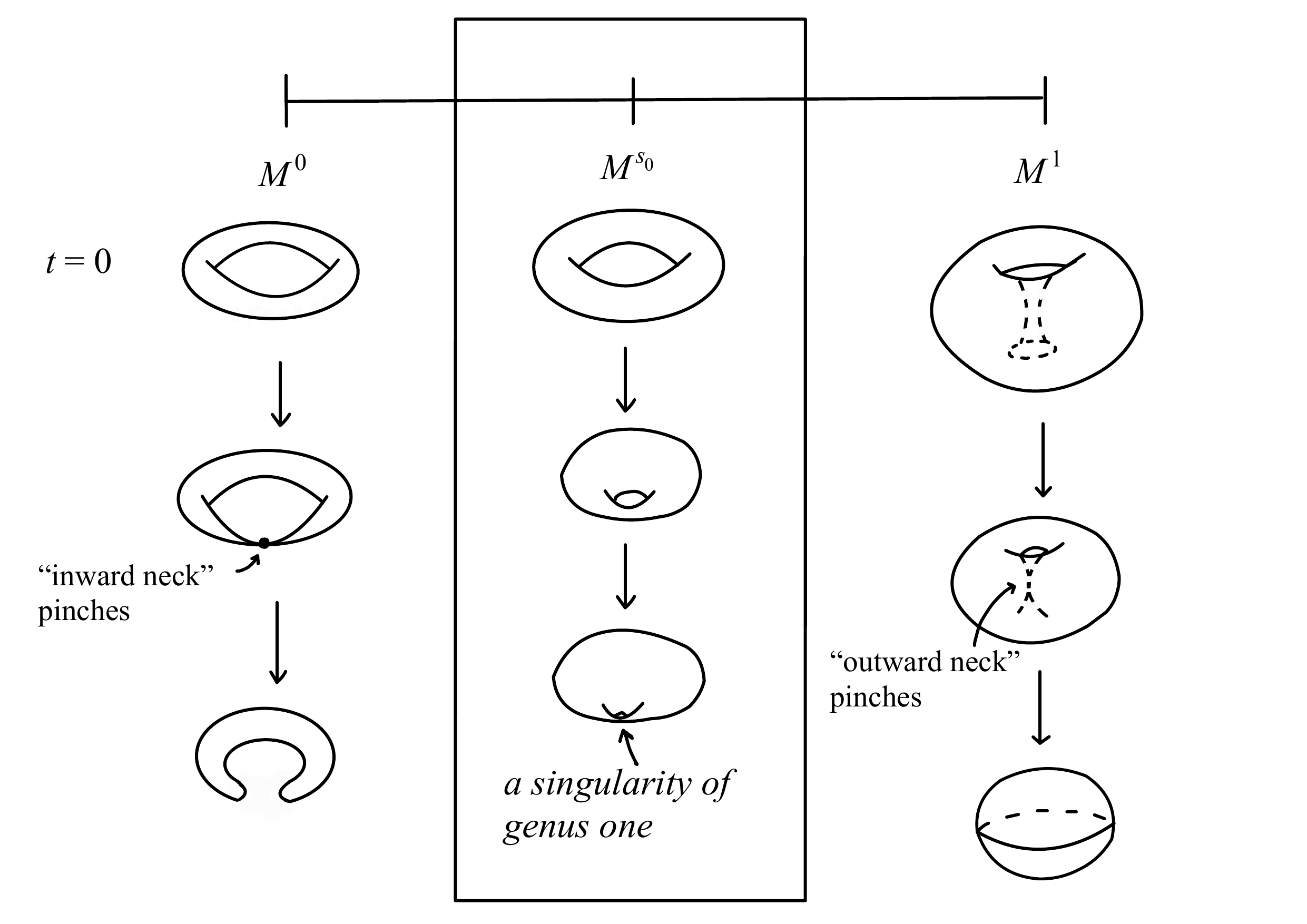}}
        \caption{The top row shows a one-parameter family of initial conditions $\{M^s\}_{s\in [0,1]}$ such that,  under mean curvature flow, $M^0$ will develop an inward neck-pinch, while $M^1$ an outward neck-pinch. Then there exists some $s_0\in [0,1]$ such that $M^{s_0}$ will develop a genus one singularity. }
        \label{fig:mcf}
    \end{figure}

\subsubsection{``Homological non-triviality" of $\Psi|_{\partial D}$}\label{sect:homoNonTrivial}
We will relate the homology groups $H_1(S^3\backslash \Psi(y))$ for different $y\in D$. Note, all homology groups will be in $\Z_2$-coefficients. First, let $D_g\subset D$ be the (open) subset of all $y\in D$ such that $\Psi(y)$ has genus exactly $g$ (which means, instead of being like  Figure \ref{fig:oneSurf}, the surface should intersect the vertical axis at $2g+2$ distinct points). In this case, $H_1(S^3\backslash \Psi(y))=\Z^g_2\oplus\Z^g_2$. On the other hand, for every $y$ on the boundary $ \partial D_g$, some handles  have pinched or merged with each other, so that   $\Psi(y)$ has genus in the range $[1,g-1]$ and, in particular, the size of $H_1(S^3\backslash \Psi(y))$ is strictly less than $2^{2g}$. Nonetheless, we will find a way relate  the groups  $H_1(S^3\backslash \Psi(y))$ for all $y\in\overline{D_g}=D_g\cup\partial D_g$, by  embedding each of them into some fixed $\Z^g_2\oplus\Z^g_2$ continuously in $y$. In other words, we will define a ``continuous" map $\frak f$ from $\overline{D_g}$ into the set $\Gr(\Z^g_2\oplus\Z^g_2)$ of all subgroups of $\Z^g_2\oplus\Z^g_2$, such that  for each $y\in \overline{D_g}$ the subgroup $\frak f(y)$ is isomorphic to $H_1(S^3\backslash \Psi(y))$.

In fact, if we restrict to the $(2g-3)$-cycle $\partial D_g\subset \overline{D_g}$, then by the fact that every member of $\Psi|_{\partial D_g}$ has genus in the range  $[1,g-1]$, we can ensure that  the image  $\frak f(\partial D_g)$ lies within the following  subset of $\Gr(\Z^g_2\oplus\Z^g_2)$:
$$\Gr^g[1,g-1]:= \{(A_1,A_2):A_1,A_2\subset \Z^g_2 \textrm{ are subgroups}, 1\leq \rank(I_\id|_{A_1\oplus A_2})\leq g-1\},$$
where $I_\id:\Z^g_2\oplus \Z^g_2\to\Z_2$ denotes the standard bilinear form given by the identity matrix. And for any $y\in \overline{D_g}$, the {\it linking number bilinear form} for homology classes in the inner region and the outer region of $\Psi(y)$ would, under the map $\frak f$, correspond exactly to  $I_\id$.

Here comes a crucial topological step: {\ul{\it In fact, the homology class $[\frak f|_{\partial D_g}]$ is non-zero in $H_{2g-3}(\Gr^g[1,g-1])$}} (Theorem \ref{thm:partialFSigma2}). Essentially, this boils down to the fact that $D$ and $\RP^5\x\{0\}$ have intersection number 1 in $Y$ (as mentioned  in \S \ref{sect:threeLevel}), and the way we constructed  $\Psi$. The proof of this fact is tricky but elementary. We will  comment  further on this  in \S \ref{sect:toyModel}.

Let us assume this fact. Then, recall that $\Psi$ can be deformed via pinch-off processes to  become the family $\Psi'$. Since pinch-off process is genus-non-increasing, the groups $H_1(S^3\backslash\Psi'(y))$ may be smaller than $H_1(S^3\backslash\Psi(y))$, meaning only certain {\it subgroup}
 of elements in  $H_1(S^3\backslash\Psi(y))$  survive the pinch-off process. Some classes might cease to exist due to some neck-pinches. Over on the $\Gr(\Z^g_2\oplus\Z^g_2)$ side, the pinch-off process has the effect of deforming the map $\frak f:\overline{D_g}\to \Gr(\Z^g_2\oplus\Z^g_2)$ to become some other map $ \frak f'$, such that each subgroup $\frak f(y)\subset \Z^g_2\oplus\Z^g_2$ gets collapsed into some small subgroup $ {\frak f}'(y)\subset \frak f(y)$. 
 
 Finally, we can prove that   some member of $ \Psi'|_{D_g}$ must still have genus $g$. If not, then every member of $ \Psi'|_{\overline{D_g}}$ has genus in the range $[1,g-1]$, which means ${\frak f'}$ maps into $\Gr^g[1,g-1]$. This contradicts the topological fact that  $[{\frak f'}|_{\partial D_g}]=[\frak f|_{\partial D_g}]\ne 0$ in $H_{2g-3}(\Gr^g[1,g-1])$  we mentioned   two paragraphs ago. In conclusion, $\Psi'$ cannot be a map into $\cS_{\leq g-1}$.

\subsubsection{Non-trivial topology in $\Gr^g[1,g-1]$}\label{sect:toyModel}
Let us give some evidence for the crucial topological step, that $[\frak f|_{\partial D_g}]$ is non-zero in $H_{2g-3}(\Gr^g[1,g-1])$, from the last section. Namely, we will focus on the case $g=2$, and show that   the space
$$\Gr^2[1]:=\{(A_1,A_2):A_1,A_2\subset \Z^2_2 \textrm{ are subgroups}, \rank(I_\id|_{A_1\oplus A_2})=1\}$$
indeed has some non-trivial topology.

As mentioned in \S \ref{sect:introPsi}, the genus 2 members in the family $\Psi:\RP^5\x B^2\to\cS_{\leq 2}$ should look like desingularizaions of two intersecting spheres, having handles with various sizes and locations. These handles  can  pinch or merge with each other, resulting in singular surfaces of genus 0 or 1. 

More precisely, let us look at the left side of Figure \ref{fig:gr}. By pinching some of the loops $a_i,b_i$, we can obtain singular genus 1 surfaces. In fact, they arise in pinching the loops in any one of the following 12 ways: (1) Pinch $a_1$, (2) pinch $a_1$ and $b_1$, (3) pinch $b_1$, (4) pinch  $b_1$ and $a_2$, (5) pinch $a_2$, ..., (12) pinch $b_3$ and $a_1$. 
As a remark, when we pinch a neck, we will not break it completely, but will instead keep the pinched neck as a singular point.

Let us now consider the   first homology groups (in $\Z_2$-coefficients) of the {\it complement regions} of these genus one surfaces with (exactly one) singularity. Observe that, pinching  necks  would destroy  homology classes. For example, if $\Sigma_1$ is obtained from pinching $a_1$ while $\Sigma_2$ is  obtained from pinching both $a_1$ and $b_1$, then $H_1(S^3\backslash \Sigma_2)$ would  be a {\it subgroup}  of $H_1(S^3\backslash \Sigma_1)$. In fact,  on the right side of Figure \ref{fig:gr}, we are exactly illustrating such inclusion of subgroups: Each of the 12 points represents the first homology group of the complement region of some singular genus 1 surface. The labels  indicates the loops pinched, and the arrows point to the smaller subgroups. 

At the same time, each of these 12 points   can be viewed as an element of $\Gr^2[1]$, with the standard bilinear form $I_\id:\Z^2_2\x\Z^2_2\to\Z_2$ corresponding to the linking number bilinear form for homology classes in the inner region and the outer region of a singular genus 1 surface. Now, the key observation is: Equipped with an appropriate topology,  $\Gr^2[1]$  is actually like\footnote{To be precise, it is a weakly homotopy equivalence.}   a circle, and by going through all 12 ways of  pinching the loops $\alpha_i,\beta_i$, we would travel through all 12 points in this circle, thereby detecting some non-trivial topology. 

  \begin{figure}
\centering
\makebox[\textwidth][c]{\includegraphics[width=0.8\textwidth]{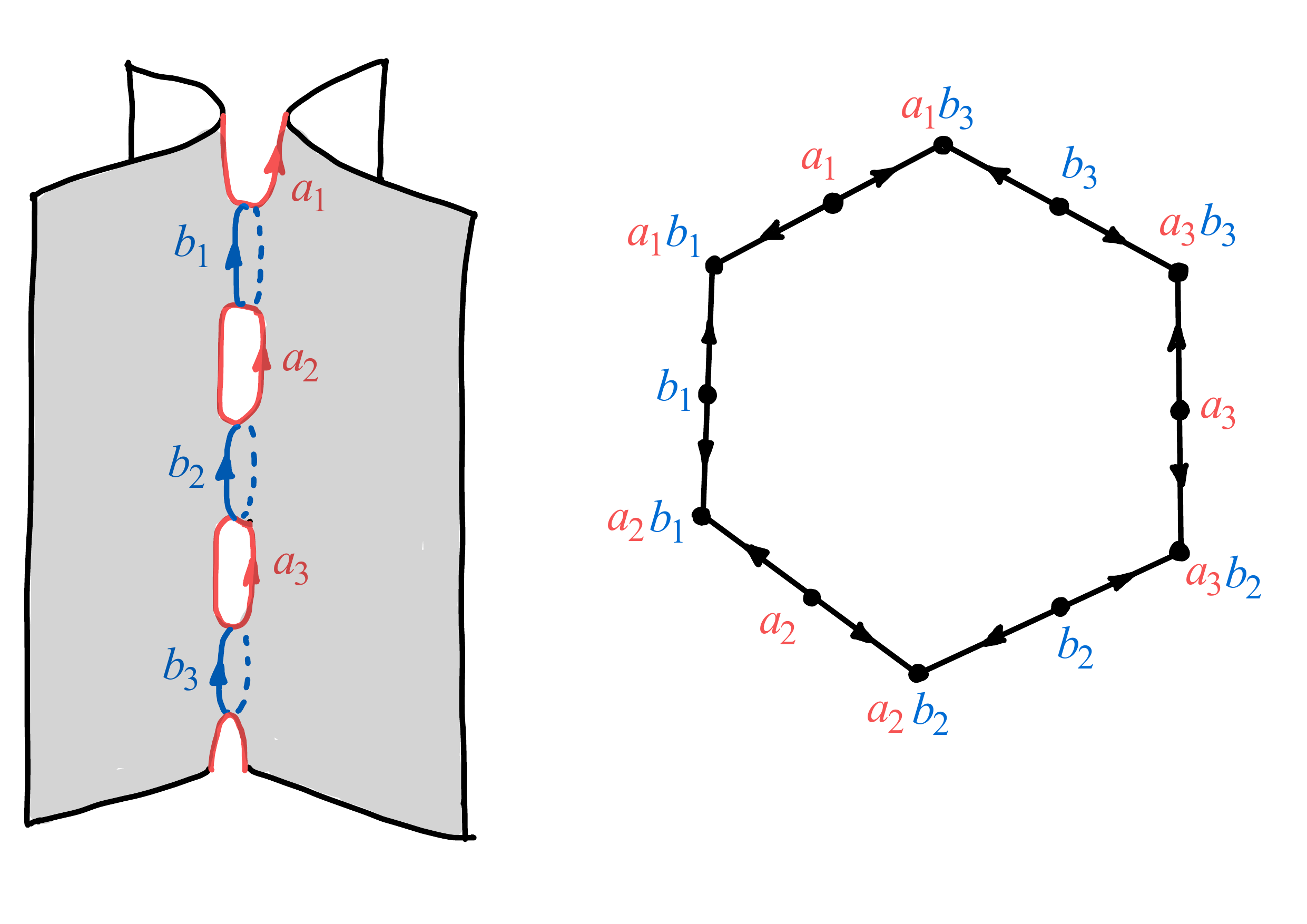}}
\caption{A genus 2 surface with loops for pinching, and an illustration of $\Gr^2[1]$.}
    \label{fig:gr}
\end{figure}

\subsection{Sketch of proof for Theorem \ref{thm:mainMinimal}} Let $(S^3,\bg_0)$ denote the given 3-sphere with positive Ricci curvature. By Theorem 1.3 of \cite{ChuLiWang2025GenusTwoI}, the fact that $\Psi:Y\to\cS_{\leq g}$ cannot be deformed into $\cS_{\leq g-1}$ immediately implies the existence of some genus $g$ minimal surface. Thus, to prove  Theorem \ref{thm:mainMinimal}, we just need to establish the area bound $2\sigma_1(S^3,\bg_0)$.

We will prove the following  slight improvement of \cite[Theorem 1.3]{ChuLiWang2025GenusTwoI}: The area of the genus $g$ minimal surface obtained is bounded from above by
\begin{equation}\label{eq:genusGAreaBound}
    \sup\{\area(\Psi(y)):y\in Y,\;\genus(\Psi(y))=g\}.
\end{equation} Then it suffices to modify $\Psi$ such that this upper bound is less than $ 2\sigma_2(S^3,\bg_0)+\epsilon$, for every $\epsilon>0$.

Recall that the  members in $\Psi$ with non-zero genus all are close to some union of two spheres intersecting at $90^\circ$, if viewed in the unit 3-sphere. Let us now deform the whole family $\Psi$ such that this angle is no longer $90^\circ$, but some small number $\theta>0$. Letting $\theta\to0$, we can ensure all members of the modified family $\tilde\Psi$ that have non-zero genus must be varifold-close to some {\it multiplicity two spheres orthogonal to the $x_1$-direction} (i.e. of the form $(x_1-c)^2=0$). Thus, under the unit metric on $S^3$, the value of (\ref{eq:genusGAreaBound}) is at most twice the area of the middle slice $x_1^2=0$, namely $8\pi$, plus $\epsilon$.

Now, we need to transport all these onto the sphere with positive Ricci curvature $(S^3,\bg_0)$. Let $\Sigma$ be a minimal sphere that achieves the first width $\sigma_1(S^3,\bg_0)$. We recall  the existence of an ``optimal 1-sweepout" associated to $\Sigma$, which is a foliation by mean convex spheres and mean concave spheres of maximum area $\area(\Sigma)$: This was constructed in Haslhofer-Ketover's work \cite{HK19}, using mean curvature flow with surgery. Now, let $\varphi$ be a diffeomorphism on $S^3$ that maps the foliation  $\{x_1=c\}_{-1<c<1}$ to this optimal foliation, and use $\varphi$ to transport the family $\tilde\Psi$ on the unit 3-sphere to $(S^3,\bg_0)$. Then, it can easily be ensured that  the value (\ref{eq:genusGAreaBound}) is less than $2\sigma_1(S^3,\bg_0)+\epsilon$, as desired.

\subsection{Organization}
In \S \ref{sect:prelim}, we state some preliminary results. In \S \ref{sect:homoDescent}, we introduce the notion of homology descent for pinch-off process. In \S \ref{sect:ProofTopo} and \ref{sect:mainProofMinimal}, we prove the two main theorems respectively, while leaving the proofs of some key steps to later sections.

In \S \ref{sect:ProofPsi1}, we prove Theorem \ref{prop:Psi1}. In \S \ref{sect:ProofCapProd}, we prove Lemma \ref{lem:capProd}. In \S \ref{sect:ProofDeformRetractF2}, we prove Proposition \ref{prop:DeformRetractF2}. In \S \ref{sect:ProofEmbedHomology}, we prove Proposition \ref{prop:embedHomology},  Finally, in \S \ref{sect:ProofPartialFSigma2}, we prove Theorem \ref{thm:partialFSigma2}, which is a crucial topological step for our whole argument.

\subsection*{Acknowledgment} The author would like to thank Yangyang Li and  Zhihan Wang for the fruitful collaborations on  related projects and their comments on this paper. He would also like to thank
Andy Putman, Daniel Stern, and Ao Sun for the insightful discussions.

This material is based upon work supported by the National Science Foundation under Grant No. DMS-1928930, while the  author was in residence at the Simons Laufer Mathematical Sciences Institute (formerly MSRI) in Berkeley, California, during the Fall semester of 2024. The   author is also partially supported by the AMS-Simons travel grant.

\section{Preliminaries}\label{sect:prelim}
\subsection{Surface with singularities}
 Let $M$ be a closed orientable 3-manifold equipped with a Riemannian metric. In this paper we will consider surfaces with finite area and {\it finitely many singularities}, which possibly include some isolated points too.

    \begin{defn} \label{def:punctate_surf} Let $M$ be a closed orientable 3-manifold.
        We let $\cS(M)$ be the set of all closed sets $S\subset M$ that satisfy the following.
        \begin{enumerate}
            \item The 2-dimensional Hausdorff measure $\cH^2(S)$ is finite.
            \item\label{item:punctateOrientable} There exists a finite set $P \subset S$ such that $S \setminus P$ is a smooth, orientable, embedded surface.
            \item\label{item:punctateSeparate} (Separating) Let $S_{\mathrm{iso}}$ be the set of isolated points of $S$. Then the complement of $S\backslash S_{\mathrm{iso}}$ is a disjoint union of two open regions of $M$, each having $S\backslash S_{\mathrm{iso}}$ as its topological boundary.
        \end{enumerate} 
        We may call any such closed set $S$ a {\it  punctate surface} in $M$.
    \end{defn}
Consideration of surfaces with singularities is essential for Theorem \ref{thm:TopoMinMax}: In certain steps of its proof, one needs to deform some {\it family} of surfaces into ones of lower genus, and sometimes this is impossible (even if we allow the deformation to be non-smooth at some points), unless singularities are allowed  (see  \cite[Example 2.6]{chuLi2024fiveTori}).

As proven in \cite[\S 2]{chuLi2024fiveTori}, every element  $S\in \cS(M)$ can be associated uniquely to an integral 2-current $[S]\in \bI_2(M;\Z_2)$ with $\partial [S]=0$. As a result, we can borrow the $\bF$-metric  on $\cZ_2(M;\Z_2)$ and impose it on $\cS(M)$. Furthermore, we can also define the notion of genus. See \cite[\S 2]{ChuLiWang2025GenusTwoI} regarding the well-definedness of this definition. Below, $B_r(P)$ denotes the union of open $r$-balls in $M$ with centers lying in $P$, while $\fg(\cdot)$ denotes the genus.
 
\begin{defn} \label{def:genus} Let $S\in\cS(M)$, $S_{\sing}$ be the set of its non-smooth points, and  $E\subset [0,\infty)$ be the set of $r>0$ such that $S \setminus B_r(S_{\sing} )$ is a smooth surface with  boundary. Then we define the {\it genus} of $S$ by
        \[
            \fg(S) := \lim_{r\in E,r\to 0} \fg(S \setminus B_{r}(S_{\sing}  ))\,.
        \]
    Moreover, we introduce the following notations:
    \begin{itemize}
        \item The set of $S\in\cS(M)$ with genus $g$ is denoted $\cS_{g}(M)$, or simply $\cS_{g}$.
        \item The set of $S\in\cS(M)$ with genus $\leq g$ is denoted $\cS_{\leq g}(M)$, or simply $\cS_{\leq g}$.
    \end{itemize} 
    \end{defn}

We now consider families of members in $\cS(M)$.  

   \begin{defn} \label{def:Simon_Smith_family}
        Let $X$ be a finite cubical or simplicial complex. 
        A map $\Phi: X \to \GS(M)$ is called a {\em Simon-Smith  family} with parameter space $X$ if the following hold.
        \begin{enumerate}[label=\normalfont(\arabic*)]
            \item \label{item:Hausdorff_cts} The  map $x\mapsto\cH^2(\Phi(x))$ is continuous. 
            \item \label{item:closedFamily} (Closedness) The family $\Phi$ is ``closed" in the sense that 
            $$\{(x,p)\in X\x M: p\in \Phi(x)\}$$
            is a closed subset of $X\x M$. 
            \item \label{item:SingPointsUpperBound} ($C^\infty$-continuity away from singularities) For each $x\in X$, we can choose a finite set $P_\Phi(x)\subset \Phi(x)$ such that:
            \begin{itemize}
                \item $\Phi(x)\backslash P_\Phi(x)$ is a smooth embedded orientable surface.  
                \item For any $x_0\in X$  and within any open set  $U\subset\subset M \backslash P_\Phi(x_0)$, $\Phi(x)\to \Phi(x_0)$ smoothly whenever $x\to x_0$.
                \item $\displaystyle  \sup_{x\in X} |P_\Phi(x)|<\infty.$
            \end{itemize}
        \end{enumerate}
    \end{defn}
  We  say that a Simon-Smith family $\Phi$ is of  {\it genus $\leq g$} if each member of $\Phi$ has genus $\leq g$.  We note that, every finite cubical complex is homeomorphic to a finite simplicial complex and vice versa (see \cite[\S 4]{BP02}).

  We now describe a special type of one-parameter family, called {\it pinch-off process}. See Figure \ref{fig:pinchOff} for an example.
      \begin{defn}\label{defn:pinch_off}
        In a closed orientable Riemannian 3-manifold $M$, a   collection  $\{\Gamma(t)\}_{t\in [a,b]}$ of elements in $\cS(M)$ is called a {\em    pinch-off process} if the following holds.  Let us view the parameter space $[a, b]$   as a time interval. Then
         there exist finitely many spacetime points $(t_1,p_1),\cdots,(t_n,p_n)$ in $(a,b]\x M$, with each  $p_i\in \Gamma(t_i)$, such that  each $(t,p)\in[a,b]\x M$ is of one of the following types.

        If $(t,p)$ is not equal to any $(t_i,p_i)$, then:
        \begin{enumerate}
            \item  There exists a closed neighborhood $J\subset [a,b]$ around $t$ and a ball $U\subset M$ around $p$ such that  $\{\Gamma(t')\}_{t'\in J}$ deforms by isotopy within $U$.
        \end{enumerate}
If $(t,p)$ is equal to some $(t_i,p_i)$, then there exists a closed neighborhood $J\subset [a,b]$ around $t$ and a ball $U\subset M$ around  $p$ such that the family $\{\Gamma(t')\cap U\}_{t'\in J}$ of surfaces is of one of the following types: (Note it is possible that $t=b$.)
        \begin{enumerate}
        \setcounter{enumi}{1}
                \item {\it Surgery }:
                The initial surface deforms by isotopy for $t'<t$, but then at time $t$  becomes, up to diffeomorphism, a double cone with $p$ as the cone point: See the second picture in Figure \ref{fig:pinchOff}. Afterwards, it either remains a double cone or splits into two smooth discs: In either case it is allowed to deform by isotopy.
                \item {\it Shrinking}: 
                For each $t'\in J$, {\it $\Gamma(t')$ does not intersect $\partial U$.} Moreover, for the family
                $\{ \Gamma(t')\cap U\}_{t' \in J }$, the initial surface deforms by  isotopy in $U$ when $t'<t$, but then becomes just the point $p$ at time $t$, and this point moves smoothly after time $t.$
        \end{enumerate}
\end{defn}
    
\subsection{$p$-sweepout}\label{sect:pSweepout}
Let us also recall some notions from Almgren-Pitts min-max theory. Let $M$ be any closed Riemannian manifold eqipped with a Riemannian metric. Let $\Phi:X\to\cZ_2(M;\Z_2)$ be an $\bF$-continuous map, with $X$ being a  finite simplicial complex or cubical complex.  

By the Almgren isomorphism theorem \cite{Alm62} (see also \cite[\S 2.5]{LMN18}),  when equipped with the flat topology, $\Zc_{n}(M;\Z_2)$ is weakly homotopic equivalent to $\R\mathbb P^\infty$. Thus we can denote its cohomology ring by $\Z_2[\bar\lambda]$, where $\bar \lambda \in H^1(\Zc_{n}( M;\Z_2), \Z_2)$ is the generator.
        
    \begin{defn} 
        Let $\Pc_p$ be the set of all $\Fb$-continuous maps $\Phi:X\to \Zc_2(M;\Z_2)$, where $X$ is a finite simplicial complex or cubical complex, such that the pullback $\Phi^*(\bar\lambda^p)\ne 0$. Elements of $\Pc_p$ are called {\it $p$-sweepouts.}
    \end{defn} 
Let us also recall a fact from    \cite[\S 3]{chuLi2024fiveTori} (see also \cite[\S 2]{ChuLiWang2025GenusTwoI}), which is based on the fact that the  Almgren-Pitts 5-width of the standard 3-sphere is $2\pi^2$ \cite{MN14,Nur16}.
\begin{prop}\label{prop:no5sweepoutGenus0}
No Simon-Smith family of genus $0$ in $S^3$ can be a $5$-sweepout.
\end{prop}

Finally, we introduce the first Simon-Smith width.
\begin{defn}
Given any Riemannian 3-sphere $(S^3,\bg)$, denote by $\cP$   the set of all 1-sweepouts by smooth embedded spheres that are continuous in the smooth topology. Then we define the  {\it first Simon-Smith width} (also called the {\it first spherical area width}) of $(S^3,\bg)$ by
$$\sigma_1(S^3,\bg):=\inf_{\Phi\in \cP}\max_{x\in \dmn(\Phi)}\area(\Phi(x)).$$
\end{defn}

\begin{rmk}\label{rmk:firstWidth} The set of embedded minimal spheres in any 3-manifold of positive Ricci curvature  is  compact under the smooth topology, by Choi-Schoen compactness \cite{ChoiSchoen1985compactness}.
Hence, if $(S^3,\bg)$ has positive Ricci curvature, then by Simon-Smith min-max theory \cite{Smith82,CD03}, together with Theorem 1.8 of Haslhofer-Ketover's work \cite{HK19} (which gives the existence of an optimal 1-sweepout by smooth spheres), it can be easily shown that  there is some embedded minimal sphere $\Sigma$ of the least area,  such that $\area(\Sigma)=\sigma_1(S^3,\bg)$.
\end{rmk}

\subsection{From topology in $\cS(M)$ to minimal surfaces}
First we introduce a definition.
\begin{defn}\label{defn:deformViaPinchoff}
    Let $\Phi,\Phi':X\to\cS(M)$   be Simon-Smith families. Suppose we have a Simon-Smith family $H:[0,1]\x X\to \cS (M)$ such that:
    \begin{itemize}
        \item  $H(0, \cdot)=\Phi $
        and $H(1, \cdot)=\Phi'$.
        \item For each $x\in X$, $t\mapsto H(t,x)$ is a pinch-off process.
    \end{itemize}
    Then we say  $H$ is a {\it deformation via pinch-off processes} from $\Phi$ to $\Phi'$.
\end{defn}

We now prove a slighlt variant of  \cite[Theorem 1.3]{ChuLiWang2025GenusTwoI} by Chu-Li-Wang, by improving the area upper bound therein. This result  builds on numerous previous important and foundational results in min-max theory \cite{Smith82, CD03, DeLellisPellandini10, Ket19, MN21, Zho20, WangZhou23FourMinimalSpheres}. Crucially, we need the multiplicity one theorem in the Simon-Smith setting by Wang-Zhou \cite{WangZhou23FourMinimalSpheres}.

\begin{thm}\label{thm:TopoMinMax}
    Let $(M,\bg)$ be a closed orientable Riemannian $3$-manifold with positive Ricci curvature, and $g$ be a positive integer. Suppose there exists some Simon-Smith family
    $$\Phi:X\to \cS_{\leq g}(M)$$ that cannot be ``deformed via pinch-off processes"   to become a map into $\cS_{\leq g-1}(M)$. 
    Then $M$ contains some orientable, embedded  minimal surface of genus $g$ with  area at most \begin{equation}\label{eq:genusGArea}
        \sup\{\area(\Psi(x)):x\in X,\fg(\Psi(x))=g\}.
    \end{equation}
\end{thm}
\begin{proof}
Define $L$ to be the value (\ref{eq:genusGArea}).
The statement of \cite[Theorem 1.3]{ChuLiWang2025GenusTwoI} is the same as this theorem, except that  the area bound there is weaker: It was $\max_{x\in X}\area(\Psi(x))$, instead of $L$. So our task here is to merely improve the area bound.

It suffices to prove this claim: For every $\epsilon>0$, there  exists an orientable embedded minimal surface of genus $g$ and area $\leq L+\epsilon$. After that,  Theorem \ref{thm:TopoMinMax} follows by taking $\epsilon\to 0$ and the Choi-Schoen compactness \cite{ChoiSchoen1985compactness}.  

Suppose by contradiction there exists some $\epsilon>0$ that violates this claim. Define the set $$X_g:=\{x\in X:\fg(\Psi(x))=g\}.$$ 
Since the function $x \mapsto\area(\Phi(x))$ is continuous, we can find some subcomplex $Z'$ of $X$ (after refinement) such that:
\begin{itemize}
    \item  $X_g$ is contained in the interior $\inte(Z')$ of $Z'$ (so that $X_g\cap\overline{X\backslash Z'}=\emptyset$).
    \item For every $x\in Z'$, $ \area(\Psi(x))<L+\epsilon/3$.
\end{itemize}
Now there are two cases: (1) $Z'$ is a union of connected components of $X$ (so that the intersection $Z'\cap\overline{X\backslash Z'}$ is empty), or (2) otherwise.

For Case 1, it is easy to derive a contradiction. Namely,
since by assumption there is no minimal surface of genus $g$ and area $\leq L+\epsilon$,  \cite[Theorem 1.3]{ChuLiWang2025GenusTwoI} immediate says there is a deformation via pinch-off processes from $\Phi|_{Z'}$ to some map into $\cS_{\leq g-1}.$ Then by the definition of Case 1, we can directly use this deformation to deform $\Phi|_{Z'}$, while fixing $\Phi|_{\overline{X\backslash Z'}}$, to show that the whole family $\Phi$ can be deformed via pinch-off processes to become a map into $\cS_{\leq g-1}$, contradicting our assumption on $\Psi$.

For Case 2, the argument to derive a contradiction is entirely analogous: We just need to perform some cut-off near the boundary of $Z'$. Namely, we choose some subcomplex $Z\subset X$ (after refinement) containing $Z'$ such that:
\begin{itemize}
    \item For every $x\in Z$, $ \area(\Psi(x))<L+2\epsilon/3$.
    \item  There exists a smooth function $\eta:X\to [0,1]$, such that $\eta=0$ on $\overline{X\backslash Z}$ and $\eta=1$ on $Z'$.
\end{itemize}
Now, since by assumption there is no minimal surface of genus $g$ and area $\leq L+\epsilon$, \cite[Theorem 1.3]{ChuLiWang2025GenusTwoI} implies that $\Phi|_Z$ can be deformed via pinch-off processes to become a map into $\cS_{\leq g-1}$. This means there exists some Simon-Smith family $$H:[0,1]\x Z\to\cS_{\leq g}$$ such that $H(0,\cdot)=\Phi|_Z$,  $H(1,\cdot)$ maps into $\cS_{\leq g-1}$, and $H(\cdot,x)$ is a pinch-off process for each $x$.

Now, let us construct a deformation via pinch-off processes from $\Phi$ to a map into $\cS_{\leq g-1}$. Namely, we define $$G:[0,1]\x X\to\cS_{\leq g}$$ by 
$G(t,x):=\Phi(x)$ for all $t$
if $x\in \overline{X\backslash Z}$, 
and  $G(t,x):= H(\eta(t),x)$ if $x\in Z$. We claim that $G(1,\cdot)$  maps into $\cS_{\leq g-1}$. Indeed, on $Z'$, this is true because $H(1,\cdot)|_{Z'}$ maps into $\cS_{\leq g-1}$. And on $\overline{X\backslash Z'}$,  since $\Phi|_{\overline{X\backslash Z'}}$ maps into $\cS_{\leq g-1}$ (by the first bullet point in the definition of $Z'$ above) and pinch-off process is genus-non-increasing, we also know $G(1,\cdot)|_{\overline{X\backslash Z'}}$ maps into $\cS_{\leq g-1}$.

In summary $G$ is a deformation via pinch-off processes from $\Phi$ to a map into $\cS_{\leq g-1}$. This contradicts the assumption on $\Phi$ in Theorem \ref{thm:TopoMinMax}.\end{proof}

We will also need the following  proposition on interpolations from \cite[\S 2]{ChuLiWang2025GenusTwoI}.

\begin{prop}\label{prop:pinchOffg0g1}
    Given a closed orientable Riemannian $3$-manifold $(M, {\mathbf g})$,  and integers $0\leq g_0<g_1$, suppose that $\Phi:X\to\GS(M)$ is a Simon-Smith family of genus $\leq g_1$ with $X$ a cubical complex. Let $K$ be an arbitrary compact subset in $X$, and assume $\Phi|_K$ maps into $\cS_{\leq g_0}$.
    
    Then there exists a subcomplex $Z\subset X$, whose interior contains $K$, of the following property: $\Phi|_Z$ can be deformed via pinch-off processes to become some map into $\cS_{\leq g_0}$.
\end{prop}

\subsection{Subspaces of $\Z^n_2 $}\label{sect:subspaceOfGr} Let $\Z_2=\Z/2\Z$. Then $\Z^n_2$ is a vector space over $\Z_2$. In this section, we are going to consider the set of subspaces (or, equivalently,  subgroups) in $\Z^n_2$, i.e. the {\it Grassmannian}, and define a topology on it.
  
   For any $n\geq 1$, and any vector space $V$, we define the Grassmannian
$$\Gr(V):= \{A:A \textrm{ is subspace of }V\}.$$
We will primarily be interested in the case where $V$ is the vector space  $\Z^n_2$ over $\Z_2$.
Later we will also be interested in the  following subset  of $\Gr(\Z^{2n}_2)=\Gr(\Z^{n}_2\oplus\Z^n_2)$.
Let $I_\id:\Z_2^n \oplus \Z_2^n\to\Z_2$ be the standard bilinear form given by the $n\x n$ identity matrix. 
For any integers $0\leq n_1\leq n_2\leq n$, we  define  $$\Gr^n[n_1,n_2]:=\{(A_1, A_2):A_1,A_2\textrm{ are subspaces of }\Z_2^n,n_1\leq \rank(I_\id|_{A_1\oplus A_2})\leq n_2\}.$$
Note, since the set $\Gr(\Z^n_2)\x \Gr(\Z^n_2)$ can naturally be viewed as a subset of $\Gr(\Z^{2n}_2)$ via the mapping $(A_1,A_2)\mapsto A_1\oplus A_2$, so can $\Gr^n[n_1,n_2]$.

We are going to define a topology on $\Gr(\Z^{2n}_2)$, and thereby also $\Gr^n[n_1,n_2]$. But first, let us review  a general way to define topology on finite sets.
\subsubsection{Topology on finite sets}
Early works on finite topological spaces were given by M. McCord \cite{mccord1966FiniteSpace} and R. Stong \cite{stong1966finiteSpace}.
We also refer readers to a more modern treatment \cite{mayPishevarFinite} by J. P. May and E. Pishevar.

Given any set $X$ with a partial order $\leq$, we can define a topology on it. Namely, for any $x\in X$, we define $$U_x:=\{y\in X:x\leq y\}.$$ Then the collection of all $U_x$ is actually a basis, and we thus can form a topology with this basis for $X$. It is easy to see that $U_x$ would be the smallest open set containing $x$.
\begin{rmk}
    We note that our definition of $U_x$ differs from the usual convention in literature, in which $U_x$ is defined as $\{y\in X:y\leq x\}$ instead. Our definition of $U_x$ here is such that   the maps $\frak i_{\Gr,\ins},\frak i_{\Gr,\out}$ defined in \S \ref{sect:identificationGr} are continuous, under the topology induced by the   partial order on $\Gr^n$ in  Definition \ref{defn:partialOrderGr} (and flipping the direction of the partial order in Definition \ref{defn:partialOrderGr} would be really unnatural, as one shall see).
\end{rmk}

Since $\leq$ is a partial order, this topology is in fact $T_0$, meaning  for each pair of distinct points, there exists an open set
containing one but not the other. If $X$ is finite, then the space $X$ would also be {\it Alexandroff} (also called an {\it A-space}), which means arbitrary intersection of open sets are open. 

In the case where $(X,\leq )$ is finite, we can naturally give it an {\it abstract simplicial complex structure}, which refers of a family of subsets of $X$ that is closed under taking subsets,  i.e. every subset of a set in the family is also in the family.  Members of this family are called  {\it simplexes}. 

Namely, the simplexes  are given by the {\it  chains} in $X$: Recall that  any sequence of the form $x_1\leq...\leq x_n$, with $x_i$'s distinct,   is called a chain. The members of any given simplex are also called the {\it vertices} in it. The {\it dimension} of a simplex is defined as the number of its vertices minus 1. As a side note, let us mention that, as proven by M. McCord \cite{mccord1966FiniteSpace}, for any finite partially ordered set $X$, the geometric realization of the above abstract simplicial complex structure would be weakly homotopy equivalent to the space $X$ itself.

Finally, let us record the following facts for future use. 
\begin{lem}[Directly from Lemma 6 of   \cite{mccord1966FiniteSpace}]\label{lem:contrac} 
    If  $(X,\leq)$ is a finite partially ordered set, and $x\in X$, then $U_x$ is contractible.
\end{lem}

\begin{thm}[Theorem 6 of \cite{mccord1966FiniteSpace}]\label{lem:homoEquiv}
Suppose $p$ is a map from a space $E$ into a space $B$ for which there
exists a basic open cover $\cU$ of $B$ satisfying the following condition:
For each $U \in \mathcal{U}$,  the restriction
$p|_{p^{-1}(U)}: p^{-1}(U) \to U$
is a weak homotopy equivalence.
Then $p$  is a weak homotopy equivalence.
\end{thm}

\subsubsection{A partial order on $\Gr(\Z^n_2)$}
\begin{defn}\label{defn:partialOrderGr}
    We  define a partial order on $\Gr(\Z^n_2)$ by declaring $A\leq B$ if and only if $A\subset B$.
\end{defn}
By the previous section, we can view $\Gr(\Z^n_2)$ as a topological space, with an abstract simplicial complex structure. It is easy to see that the simplex of maximal dimension is $n$, because the maximum possible  length of a chain in  $\Gr(\Z^n_2)$ is $n+1$. Moreover, since $\Gr(\Z^n_2)=U_{(\{0\},\{0\})}$, this space is  contractible  by Lemma \ref{lem:contrac}.

Now consider $\Z^{2n}_2=\Z^n_2\oplus\Z^n_2$, and let $e_1,...,e_n$ be the standard basis in $\Z^n_2$. 
In later sections, we will be particularly interested in the subspace $\Gr^n[1,n-1]\subset \Gr(\Z^{2n}_2)$.  An example of a simplex of maximal dimension in it is \begin{align*}
   & (\langle e_1\rangle,\langle e_1\rangle)\leq(\langle e_1\rangle,\langle e_1,e_2\rangle)\leq...\leq (\langle e_1\rangle,\langle e_1,...,e_{n}\rangle)\\
   &\leq (\langle e_1,e_2\rangle,\langle e_1,...,e_{n}\rangle)\leq ...\leq (\langle e_1,...,e_{n-1}\rangle,\langle e_1,...,e_{n}\rangle),
\end{align*}
which has dimension $2n-3$.

\subsection{Symmetric product of $S^1$}\label{sect:symProd} Given a topological space $X$,
we let $\Sym^{n}(X)$ be the $n$-fold symmetric product of  $X$: 
$$\Sym^n(X):=X^n/S_n,$$
where $S_n$ is the symmetric group acting on $X^n$ by permutation.  Consider the circle $S^1:=\R/2\pi \Z$. In this section, we review the topology of $\Sym^n(S^1)$ and a natural simplicial complex structure on it.  We refer readers to \cite{morton1967symmetricProduct} or \cite{nlab:symmetric_product_of_circles} for references.

It is well-known that $\Sym^n(S^1)$ is  homeomorphic to an $\Delta^{n-1}$-bundle over $S^1$, where $\Delta^{n-1}$ denotes a closed $(n-1)$-simplex. Let us describe this homeomorphism below.

We first describe the  projection map of this bundle, $\Sym^n(S^1)\to S^1$: 
It sends any $\balpha\in \Sym^n(S^1)$ to the sum of all its $n$ members (which is well-defined modulo $2\pi$).  Let $\Sym^n_0(S^1)\subset \Sym^n(S^1)$ be the fiber over the point $[0]\in \R/2\pi\Z$. We will describe the homeomorphism from $\Delta^{n-1}$ to 
$\Sym^n_0(S^1)$ below. 

We start with a closed $(n-1)$-simplex 
\begin{equation}\label{eq:simplex}
    \{(r_1,...,r_{n-1})\in[0,2\pi]^{n-1}:\sum r_i\leq2\pi\}.
\end{equation}
For any point $(r_1,...,r_{n-1})$ in it, we define 
$$s_1:=0,s_2:=r_1,s_3:=r_1+r_2,\;...\;,s_n:=r_1+...+r_{n-1}.$$
Note $(s_1,...,s_n)$ naturally can be viewed as a point in $\Sym^n(S^1)$. Finally, we move each point by the mean: We consider
$$t_k:= s_k-\frac 1n(s_1+...+s_n),$$
for $k=1,...,n$. Then $(t_1,...,t_n)$ gives a point in $\Sym^n_0(S^1)$. The desired homeomorphism $\Delta^{n-1}\to\Sym^n_0(S^1)$ is given the map $(r_1,...,r_{n-1})\mapsto (t_1,...,t_n)$: See \cite{morton1967symmetricProduct} or \cite{nlab:symmetric_product_of_circles} for the detailed proof.

Using the simplicial complex structure on the $(n-1)$-simplex (\ref{eq:simplex}), we can impose on $\Sym^n_0(S^1)$  a natural simplicial  structure. This simplicial structure has the following geometric feature. Fix any $\balpha\in\Sym^n_0(S^1)$, and suppose that it has exactly $k$ distinct entries ($1\leq k\leq n$). Then we consider the set of all possible configurations in $\Sym^n_0(S^1)$ that can be obtained from applying some diffeomorphism on $S^1$ to $\balpha$. The collection of all such configurations forms an open $(k-1)$-dimensional subset of $\Sym^n_0(S^1)$ and would be an open $(k-1)$-cell, which contains in particular the point $\balpha$.

\section{Homology descent for pinch-off process}\label{sect:homoDescent}

Let $M$ be a smooth orientable 3-manifold.
In this section, we study the  homology classes of the complement region of members of Simon-Smith families. Namely, given a Simon-Smith family $\Phi:X\to\cS(M)$, we would like to  relate the homology groups $H_1(M\backslash \Phi(x))$ for $x\in X$. Note all chains and homology groups will be in $\Z_2$-coefficients.

\subsection{Homology classes in a Simon-Smith family}

\begin{defn}
Let $M$ be a smooth orientable 3-manifold, and $\Phi:X\to\cS(M)$ be a Simon-Smith family.
    We define the set
$$\fH(\Phi):=\{(x,c):x\in X,c\in H_1(M\backslash \Phi(x))\}.$$
\end{defn}
Note that, for each $x\in X$, we can canonically identify $H_1(M\backslash\Phi(x))$ with the subset  $\{x\}\x H_1(M\backslash\Phi(x))\subset \fH(\Phi)$.

We  equip on $\fH(\Phi)$ a topology as follows. For any $x_0\in X$, let $\gamma\subset M\backslash\Phi(x_0)$ be any simplicial 1-cycle. Let $U\subset X$ be any open neighborhood of $x_0$ that satisfies the following property: For each $x\in U$, the 1-cycle $\gamma$ lies inside $M\backslash\Phi(x)$ too. Note that any sufficiently small open neighborhood  $U$ of $x_0$ is guaranteed to satisfy this property, by the closedness property of Simon-Smith family (Definition \ref{def:Simon_Smith_family}). Then, we define the set
\begin{equation}\label{eq:openSetDef}
    \fU(x_0,\gamma,U):=\{(x,c)\in \fH(\Phi):x\in U,c\in H_1(M\backslash\Phi(x)),c=[\gamma]\}.
\end{equation}
Note, here $c=[\gamma]$ means $c$ is  represented by $\gamma$ when {\it $\gamma$ is viewed as a $1$-cycle in $M\backslash\Phi(x)$.} 
\begin{lem}
    The collection of all such sets $\fU(x_0,\gamma,U)$ forms a base for some topology on $\fH(\Phi)$.
\end{lem}
\begin{proof}
    Suppose we have two sets $\fU_1=\fU(x_1,\gamma_1,U_1)$ and $\fU_2=\fU(x_2,\gamma_2,U_2)$, and some $(x_0,c_0)$ that lies in their intersection. We need to find some 
    set of the form (\ref{eq:openSetDef}) that contains $(x_0,c_0)$ and is a subset of  $\fU_1\cap \fU_2$.
    
    We observe that, by construction:
    \begin{itemize}
        \item $x_0\in U_1\cap U_2$.
        \item $\gamma_1\subset M\backslash\Phi(x_0)$, $\gamma_2\subset M\backslash\Phi(x_0)$.
        \item $[\gamma_1]=[\gamma_2]=c_0$ as elements in $H_1(M\backslash\Phi(x_0))$.
    \end{itemize}
    Let $\sigma\subset M\backslash\Phi(x_0)$ be some simplicial 2-chain bounded by $\gamma_1-\gamma_2$. Then again by the closedness property of Simon-Smith families, there exists some open neighborhood $U$ of $x_0$ within $U_1\cap U_2$ such that for any $x\in U$, $\sigma\subset M\backslash\Phi(x)$ too. Clearly, by construction,
    $$(x_0,c_0)\in \fU(x_0,\gamma_1,U).$$
        To finish the proof, it suffices to show that 
    $$\fU(x_0,\gamma_1,U)\subset \fU_1\cap\fU_2.$$

    Pick any $(x,c)\in \fU(x_0,\gamma_1,U).$ First, note that:
    \begin{itemize}
        \item $x\in U\subset U_1$.
        \item $\gamma_1\subset M\backslash\Phi(x)$ as $\sigma\subset M\backslash\Phi(x)$ by construction.
        \item $c=[\gamma_1]$ as elements in $H_1(M\backslash\Phi(x))$.
    \end{itemize}
    Hence, $(x,c)\in \fU_1$.
    Moreover, we have
    \begin{itemize}
        \item $x\in U\subset U_2$.
        \item $\gamma_2\subset M\backslash\Phi(x)$ as $\sigma\subset M\backslash\Phi(x)$.
        \item Since $\gamma_2=\gamma_1+\partial\sigma$, 
        we also have $c=[\gamma_1]=[\gamma_2]$ as elements in $H_1(M\backslash\Phi(x))$.
    \end{itemize} Thus, $(x,c)\in\fU_2$ too. So $x\in\fU_1\cap\fU_2$, as desired.
\end{proof}

\begin{exmp}
    Suppose $M=S^3$, and $\Phi_1:[0,1]\to \cS(M)$ is a smooth family of tori. Then each $H_1(M\backslash \Phi(x))$ is isomorphic to $\Z_2\oplus\Z_2$, and the topological space $\frak H(\Phi_1)$ has four connected components.

    Now, consider another smooth family $\Phi_2:S^1\to \cS(S^3)$ of  tori such that when a point travels along $S^1$ once, the inside and outside regions of the associated torus  are exchanged. Then $\frak H(\Phi_2)$  has only three connected components.
\end{exmp}

We will also often consider the following subspaces:
\begin{defn}
    Suppose $\Phi:X\to\cS(M)$ is a  Simon-Smith that admits a continuous choice  of inside and outside regions (continuous in the sense of measure) for $\Phi(x)$, for every $x$. Then we define
    $$\fH_{\ins}(\Phi):=\{(x,c)\in\fH(\Phi):c\in H_1(\ins(\Phi(x)))\}$$
    and 
        $$\fH_{\out}(\Phi):=\{(x,c)\in\fH(\Phi):c\in H_1(\out(\Phi(x)))\}$$
\end{defn}
Note, $\ins(\cdot)$ denotes the inside region, while $\out(\cdot)$ the outside. We put on these two sets  the subspace topology from $\fH(\Phi)$.

Now, let us consider the Grassmannians of $H_1(M\backslash \Phi(x))$, for different $x$. Note that in $H_1(M\backslash \Phi(x))$,  subspace and subgroup are the same.

\begin{defn}
Let $M$ be a smooth orientable 3-manifold, and $\Phi:X\to\cS(M)$ be a Simon-Smith family.
    We define the set
$$\fGr(\Phi):=\{(x,A):x\in X,A  \textrm{ is a subspace of } H_1(M\backslash\Phi(x)) \}.$$
\end{defn} 
Again, for any $x\in X$, we can canonically view the Grassmannian $\Gr(H_1(M\backslash \Phi(x)))$ as a subset of $\fGr(\Phi)$.

We can similarly equip $\fGr(\Phi)$ with a topology as follows. For any $x_0\in X$, let $\gamma_1,...,\gamma_n\subset M\backslash\Phi(x_0)$ be any finite collection of simplicial 1-cycle. Let $U\subset X$ be any open neighborhood of $x_0$ that satisfies the following property: For each $x\in U$, each 1-cycle $\gamma_i$ lies inside $M\backslash\Phi(x)$ too. Then, we define the set
\begin{align} 
    \fU(x_0,\{\gamma_i\}_{i=1}^n,U):=&\;\{(x,A)\in \fH(\Phi):x\in U,A \textrm{ is a subspace of } H_1(M\backslash\Phi(x))\\
   \nonumber  &\textrm{ with } [\gamma_1],...,[\gamma_n]\in A\}.
\end{align}
Note, here $[\gamma_i]$ can denote elements of $H_1(M\backslash \Phi(x))$ for any $x\in U$. Then similar as before, we can show that these sets define a basis for some topology on $\fGr(\Phi)$. 

\begin{rmk}
Trivially, for every $x\in X$, the subspace topology on $\Gr(H_1(M\backslash \Phi(x)))$, as viewed as a subspace of $\fGr(\Phi)$, coincides with the topology induced by the partial order on $\Gr(H_1(M\backslash \Phi(x)))$ given by inclusion, as introduced in \S \ref{sect:subspaceOfGr}. 
\end{rmk}

Lastly, let us also define the following two subsets of $\fGr(\Phi)$, which will be equipped with the subspace topology.

\begin{defn}
    Suppose $\Phi:X\to\cS(M)$ is a  Simon-Smith that admits a continuous choice of inside and outside regions for $\Phi(x)$, for every $x$. Then we define
    $$\fGr_{\ins}(\Phi):=\{(x,A)\in\fGr(\Phi):A \textrm{ is a subspace of }  H_1(\ins(\Phi(x)))\}$$
    and 
        $$\fGr_{\out}(\Phi):=\{(x,A)\in\fGr(\Phi):A \textrm{ is a subspace of }  H_1(\out(\Phi(x)))\}$$
\end{defn}

\subsection{Homology descent for pinch-off process}\label{sect:descent}
Specially for pinch-off process, we have a more refined way to relate the homology classes: We introduce notions called {\it homology descent} and {\it homology termination}. These concepts were first developed by A. Sun and the author for mean curvature flow in  \cite{chuSun2023genus}. Below, we are going to obtain analogous results for pinch-off process. Note, we will be using $\Z_2$-coefficients as usual.

We let $\{\Gamma(t)\}_t$ be a pinch-off process: For notational simplicity, we shall always assume it has $[0,T]$ as its parameter space in this section.  We will  consider the complements of the time-slices. We denote, for any $t\in [0,T]$,
    \[
        W[t]:=\{t\}\x (M\backslash \Gamma(t))\subset [0,T]\x M,
    \]
    and for any $t_0,t_1\in[0,T]$, 
    \[
        W[t_0,t_1]:=\bigcup_{t\in[t_0,t_1]}W[t]\subset [T_0,T_1]\x M.
    \]
    
Furthermore,   by Definition \ref{def:punctate_surf} and Definition \ref{def:Simon_Smith_family}, we may fix a continuous choice  of inside region $\ins(\Gamma(t))$ and outside region $\out(\Gamma(t))$ for $\Gamma(t)$, for any $t\in [0,T]$. These all are open regions in $M$. Now, denote $W_\ins[t]:=\{t\}\x \ins(\Gamma(t))$ and $W_\out[t]:=\{t\}\x \out(\Gamma(t))$.

\begin{defn}[Homology descent]\label{defn_order}
We define a  relation $\succ$ on the {\it disjoint} union $$\bigsqcup_{t\in [0,T]}H_1(W[t])$$ as follows. Given two times $ t_0\leq t_1$ in $[0,T]$, and two  homology  classes $c_0\in H_1(W[t_0] )$ and $c_1\in H_{1}(W[t_1] )$, we say that {\it $c_1$ descends from $c_0$}, and denote $c_0\succ c_1,$ if every representative $\gamma_0\in c_0$ and $\gamma_1\in c_1$ must together bound some $2$-chain $\sigma\subset W[T_0,T_1]$, i.e. $\gamma_0-\gamma_1=\partial\sigma$ (see Figure \ref{fig:descent}).
\end{defn}

  \begin{figure}
        \centering
        \makebox[\textwidth][c]{\includegraphics[width=0.7\textwidth]{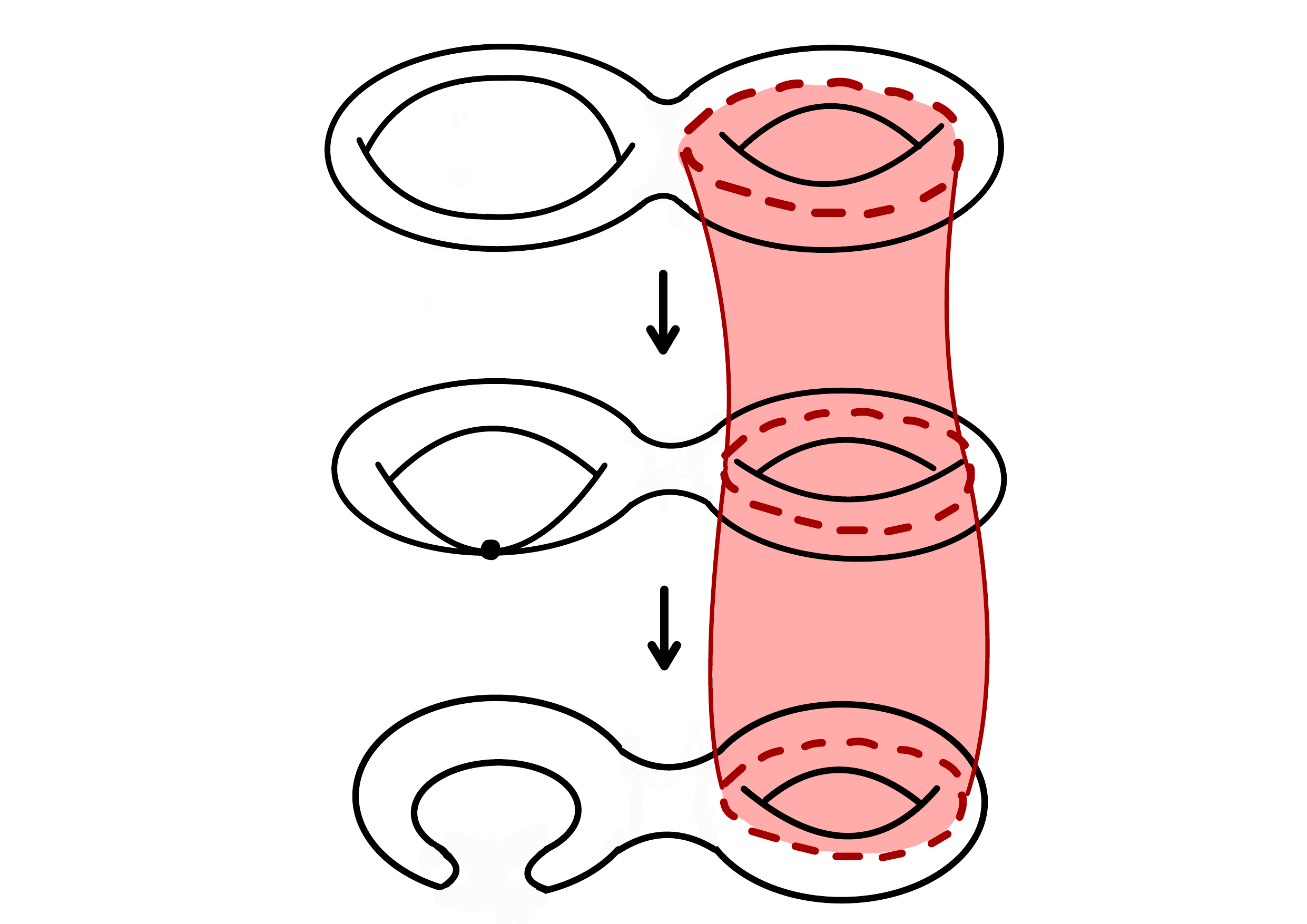}}
        \caption{ The  loop in the top picture is $\gamma_0$, while the  loop in the bottom picture is $\gamma_1$. The shaded surface is $\sigma$.}
        \label{fig:descent}
    \end{figure}

It is straightforward to check that $\succ$ is actually a partial order. Below, let us state some  properties for this partial order.

\begin{prop}\label{prop_at_least_one}
   Let $c_1\in H_1(W[T_1] )$ and $T_0\leq T_1$. Then there exists at least one   $c_0\in H_1(W[T_0] )$ such that $c_0\succ c_1$.
\end{prop} Note such $c_0$ may not be unique. Take Figure \ref{fig:descent} as an example. At the starting time, if $c_0$ denotes the class in the outside region that is linked with the left (solid) handle, then we see that at the end, $c_0$ descends to the trivial class. But clearly, the trivial class at the beginning also descends to the trivial class at the end.
\begin{proof}
In \cite[\S 9]{chuLi2024fiveTori}, it was proven that for any loop in $W[T_1]$, we can homotope it within $W[T_0,T_1]$ to a loop in $W[T_0]$. Now,   letting $c_1$ be represented by a sum of loops $\gamma_1+...+\gamma_n$,  we apply this  fact to each $\gamma_i$. Then the proposition follows immediately.
\end{proof}

\begin{prop}
    Let $T_0<T_1$, $c_0\in H_1(W[T_0])$, and  $c_1\in H_1(W[T_1])$ such that $c_0\succ c_1$. Then for every $t\in [T_0,T_1]$ there exists at least one $c\in H_1(W[t])$ such that $c_0\succ c\succ c_1$.
\end{prop}
Roughly speaking, this says a homology class cannot disappear and then reappear. Later we will see that actually such $c$ is unique.
\begin{proof}
Under our assumption, we can find $\gamma_0\in c_0$ and $\gamma_1\in c_1$ such that they together bound some simplicial $2$-chain $\sigma$ in $W[T_0,T_1]$.  By perturbing $\sigma$ and tilting its faces, we can ensure that each face does not lie entirely within any time slice $  \{t\}\x M$. This enables us to define the 1-cycle $\beta_t := \sigma\cap(\{t\}\x M)$ in $W[t]$ for each $t \in [T_0, T_1]$. Consequently, we have  $c_0 \succ [\beta_t] \succ c_1$.
\end{proof}

\begin{prop}\label{prop_homology_unique}
    Let $c_0\in H_1(W[T_0] )$ and $T_0\leq T_1$. Then there exists at most one  $c_1\in H_1(W[T_1] )$ such that $c_0\succ c_1$.
\end{prop}
It is possible that such $c_1$ does not exist. Take Figure \ref{fig:descent} as an example. At the starting time, let $c_0$ be the homology class in the inner region that represents the solid handle on the left. Then it does {\it not} descend to any class at the final time.  Also note, it is based on this proposition that we picked the symbol $\succ$ (instead of $\prec$): It   pictographically reflects that more than one homology class may descend into one, but not the other way around.
\begin{proof}
 We fix the times $T_0$ and $T_1$, and the class $c_0$.
Let $A$ be the set of all $t\in [T_0,T_1]$ for which there are at least two classes, $c,c'\in H_1(W[t])$, such that $c_0\succ c$ and $c_0\succ c'$. Suppose by contradiction that $A$ is non-empty.

Let $T:=\inf A$. Take a sequence $t_1,t_2,...$ in $A$ with $t_i\downarrow T$. Fix an $i$ large enough such that the time interval $(T,t_i]$ contains no singular time for the pinch-off process.

By definition,   we can find 1-cycles $\gamma_0\subset W[T_0]$ and $\gamma_1,\gamma'_1\in W[t_i]$ such that
$\gamma_0$ and $\gamma_1$ bound some simplicial 2-chain $\sigma$ in $W[T_0,T_1]$, and $\gamma_0$ and $\gamma'_1$ also bound some simplicial 2-chain $\sigma'$ in $W[T_0,T_1]$.  For each $t\in [T_0,T_1]$, we define $\beta_t$ as the time slice $\sigma\cap (\{t\}\x M)$  of $\sigma$, and $\beta'_t$ as the time slice $\sigma'\cap (\{t\}\x M)$  of $\sigma'$. By perturbing $\sigma$ and $\sigma'$, we can assume that each $\beta_t,\beta'_t$ is a simplicial 1-cycle. By modifying the 2-chains $\sigma,\sigma'$, we can also assume that the two families of cycles $\{\beta_t\}_t$ and $\{\beta'_t\}_t$ vary smoothly for $t$ in some small neighborhood of $T$. 

We claim that $T\in A$. Indeed, if $T\notin A$, then by the definition of $A$, the cycles $\beta_T$ and $\beta'_T$  must be homologous in $W[T]$, as both $[\beta_T]$ and $[\beta'_T]$ descend from $c_0$. Let $B_T$ be the 2-cycle they bound in $W[T]$. By the closedness property of Simon-Smith family, we can form a smoothly varying family $\{B_t\}_t$ with $B_t\subset W[t]$ and $\partial B_t=\beta_t-\beta'_t$ for all $t\geq T$ sufficiently close to $T$. But since the pinch-off process $\Gamma(t)$ varies smoothly for $t\in (T,t_i]$, we can extend the  family $\{B_t\}_t$ onto the time interval $  [T,t_i]$. This contradicts that $[\gamma'_1]\ne [\gamma_1]$ in $H_1(W[t_i])$. Hence, we know indeed $T\in A$. In particular, $T>T_0$.

Now, fix a time $s<T$  close to $T$ such that within the time period $[s,T)$, the pinch-off process has no singular time. Then since $s\notin A$, we know $\beta_s$ and $\beta'_s$ must be homologous in $W[s]$. Let $C_s$ be a 2-chain they bound in $W[s]$. To derive a contraction, we are going to find a  family $\{C_t\}_{t\in [s,T]}$ (which need not be continuously varying) with $C_t\subset W[t]$ and $\partial C_t=\beta_t-\beta'_t$: This would imply in particular $\beta_T$ and $\beta'_T$ are homologous in $W[T]$, contradicting $T\in A$, which we proved in the last paragraph.

If $T$ is not a singular time for the pinch-off process, such family  $\{C_t\}_{t\in [s,T]}$ clearly exists. So we shall assume $T$ is a singular time. Now, without loss of generality, we can assume that the loop $\gamma_0$  either lies within $W_\ins[T_0]$ or $W_\out[T_0]$ (for the general situation we can consider the components of $\gamma_0$ in the inside and outside regions separately). Let us just do the first case, as the other is analogous. Hence, all $\beta_t,\beta'_t$ must lie in $W_\ins[t]$.

In fact, to simplify presentation, we will assume the pinch-off process has only one singular point $(T,p)$ at time $T$, for it would be straightforward to adapt the argument below to the case of finitely many singular points at time $T$. Now, let us say that the singular point $(T,p)$ is of {\it inward type} if either case holds:
\begin{enumerate}[label=(\alph*)]
    \item This singular point is a neck-pinch point, such the solid cylinder region lies in the inside region.
    \item This singular point corresponds to shrinking some component of the inside region into the point $p$, at time $T$.
\end{enumerate}
Analogously, we can define what ``$(T,p)$ is of {\it outward type}" means. We discuss the inward and outward types separately.

\subsection*{$(T,p)$ is of inward type.} There are two cases, (a) and (b), as stated above. Let us do case (b)  first. Let $[s_1,T]\x U$ be a spacetime neighborhood  within which the pinch-off process takes the form of shrinking some connected component into the point $p$ (as described precisely in Definition \ref{defn:pinch_off}). By choosing this spacetime neighborhood small enough, we can assume that $[s_1,T]\x  U$ avoids $\beta_t,\beta'_t$  for each $t\in [s_1,T]$. Hence, we can first construct the desired family $\{C_t\}$ for $t\in [s,s_1]$ such that,   by the definition of case (b), {\it $C_{s_1}\cap U$ is a union of  $2$-cycles without boundary}. Hence, we can remove these 2-cycles from $C_{s_1}$, and easily extend the family over to the whole time interval $[s,T]$, as the  pinch-off process varies smoothly outside $U$.

Then, we do case (a). Let $U$ be a ball in $M$ around $p$, and $s_1<T$, for   which the neck-pinch takes place in $[s_1,T]\x  U$. Then for $t\in [s_1,T)$, $\Gamma(t)$ is a smooth cylinder. And by choosing this spacetime neighborhood small enough, we can assume that $[s_1,T]\x  U$ avoids $\beta_t,\beta'_t$  for each $t\in [s_1,T]$.   First, we can smoothly vary the 2-chain $C_s$ to obtain a smooth family $\{C_t\}_{t\in [s,s_1]}$   with $C_t\subset W[t]$ and $\partial C_t=\beta_t-\beta'_t$. Now, at time $s_1$, the 2-chain $C_{s_1}$ may intersect $U$. If so, the intersection must be within the solid cylinder region $\{s_1\}\x(\ins(\Gamma(s_1))\cap U)$. 

Now,  $\Gamma(s_1)\cap \partial U$ consists of two loops, and each of them  bounds a disc on the sphere $\partial U$. Let $D\subset\partial U$ be the union of these two discs. We can assume that the cross section $C_{s_1}\cap D$ consists of finitely many loops $\delta_1,\cdots,\delta_m$, with each $\delta_j$  bounding a disc $d_j$ within $D$. By performing a surgery to $C_{s_1}$ along each loop $\delta_j$ (i.e. removing a cylinder near $\delta_j$ and gluing back two approximate copies of $d_j$), we can obtain another 2-chain $\tilde C_{s_1}$ that  avoids the sphere $\partial U$. By removing the components of $\tilde  C_{s_1}$ in $U$, we can assume $\tilde  C_{s_1}$ avoids $U$. As a result, now there is no problem in smoothly deforming the 2-chain $\tilde  C_{s_1}$ in the time interval $[s_1,T]$, and obtain a family $\{\tilde  C_t\}_{t\in [s_1,T]}$ with desired properties. 

\subsection*{$(T,p)$ is of outward type}. If $(T,p)$ is an outward neck-pinch, it is clear from the definition of pinch-off process that we can find a smoothly varying family $\{C_t\}_{t\in [s,T]}$  with $C_t\subset W[t]$ and $\partial C_t=\beta_t-\beta'_t$, as desired. Intuitively speaking, this is because the inside region is  gaining ground. 

If instead $(T,p)$ corresponds to an outward shrinking process, then pick a  spacetime ball $[s_1,T]\x  U$, with $U$ centered at $p$, which avoids $\sigma$ and $\sigma'$. Clearly we can define $C_t$ for $t$ up to $s_1$, but $C_{s_1}$ may intersect $U$. Now, we perform surgery for $C_{s_1}$ along every loop component of $C_{s_1}\cap \partial U$ to make it avoid $\partial U$. Then we discard everything inside $U$, and call the new 2-chain $\tilde C_{s_1}$. Now there is no problem in smoothly deforming the 2-chain $\tilde  C_{s_1}$ in the time interval $[s_1,T]$, and obtain a family $\{\tilde  C_t\}_{t\in [s_1,T]}$ with desired properties.  \end{proof}

Hence, based on the  previous two propositions, we have:
\begin{cor}\label{prop_unique_clss_inbetween}
    Let $T_0<T_1$, $c_0\in H_1(W[T_0])$, and  $c_1\in H_1(W[T_1])$ such that $c_0\succ c_1$. Then for every $t\in [T_0,T_1]$ there exists a \ul{unique} $c\in H_1(W[t])$ such that $c_0\succ c\succ c_1$.
\end{cor}

Based on the above corollary, we introduce the following definition.
\begin{defn}[Homology termination]\label{defn_termination}
    Let $c_0\in H_{1}(W[T_0])$. Suppose the set
    $$\{t\in (T_0,T]:\textrm{there is no } c\in H_1(W[t]) \textrm{ such that } c_0\succ c\}$$
    is non-empty 
    (in which case this set must be an interval by the above corollary). Then we denote by $\frak t(c_0)$ the infimum of this set, and say that $c_0$ {\it terminates at time $\frak t(c_0)$}.    If instead the above set is empty, then we say {\it $c_0$ never terminates.} 
\end{defn}

\begin{rmk}\label{rmk:maxExistence}
    By the openness of the region $W[0,T]$, we know that if $c_0$ does terminate (i.e. the set above is non-empty). then  there can be no $c\in H_1(W[\frak t(c_0)])$ such that $c_0\succ c$. Therefore, one can interpret the time interval $[T_0,\mathfrak{t}(c_0))$ as the ``maximal interval of existence'' for $c_0$.
\end{rmk}

Finally, we introduce the following definition.
\begin{defn}
    Let $\{\Gamma(t)\}_{t\in [0,T]}$ be a pinch-off process. For each $t\in [0,T]$, we define\footnote{$\fb$ stands for backtrack.} $\fb^\Gamma(t)$ to be the subset of $H_1(W[0])$ that consists of all elements $c\in H_1(W[0])$ that descends to some element of $H_1(W[t])$ (i.e. $c$ has not terminated by time $t$).
\end{defn}

Note that $\fb^\Gamma(t)$ must be a subgroup of $H_1(W[0])$. Indeed, suppose we know $c_0,c'_0\in H_1(W[0])$  descend to respectively $c_t,c'_t\in H_1(W[t])$,  then we know that $c_0+c'_0$ must descend to  $c_t+c'_t\in H_1(W[t])$. In addition, we denote
$$\fb^\Gamma_\ins(t):=\fb^\Gamma(t)\cap H_1(W_\ins[0]),\quad \fb^\Gamma_\out(t):=\fb^\Gamma(t)\cap H_1(W_\out[0]).$$
By identifying each $M\backslash(\Gamma(t))$ with $W[t]=\{t\}\x\Gamma(t)$, we can as well assume $\fb^\Gamma_\ins(t)$ is a subspace of $H_1(\ins(\Gamma(0)))$, and similarly for $\fb^\Gamma_\out(t).$ 
Note,  both maps
$$\fb^\Gamma_\ins:[0,1]\to \Gr(H_1(\ins(\Gamma(0))))\quad\textrm{ and }\quad\fb^\Gamma_\out:[0,1]\to \Gr(H_1(\out(\Gamma(0))))$$ are  continuous under the topology we equip on Grassmannians in \S \ref{sect:subspaceOfGr} by Remark \ref{rmk:maxExistence}. 

\subsection{Homology descent and Grassmannians on family}
Now, let us unify the previous two subsections. Suppose we are given a Simon-Smith family $\Phi:X\to \cS(M)$, and a deformation via pinch-off processes for it, $H:[0,1]\x X\to \cS(M)$. This means:
\begin{itemize}
    \item $H(0,\cdot)=\Phi$.
    \item For each $x$, $H(\cdot, x)$ is a pinch-off process.
\end{itemize}
Hence, we can define a map$$\fb^H_\ins:[0,1]\x X\to \fGr_\ins(\Phi)$$
by sending each $(t,x)$ to $(x, \fb^{H(\cdot,x)}_\ins(t))$. Note $\fb^{H(\cdot,x)}_\ins(t)$ is an element in the Grassmannian $\Gr(H_1(\ins(\Phi(x))))$. Analogously, we can define a map
$$\fb^H_\out:[0,1]\x X\to \fGr_\out(\Phi).$$
Again, it is straightforward to check that both maps are continuous. Namely, take any open set $\frak U(x_0,\{\gamma_i\}_i,U)$ in $\fGr_\ins(\Phi)$, and suppose $\frak b^H_\ins(t_1,x_1)$ is a point it in. Then $x_1\in U$,   $\gamma_i\subset M\backslash\Phi(x_1)$ for every $i$, and $[\gamma_i]$ has not terminated yet by time $t_1$ regarding the pinch-off process $H(\cdot,x_1)$. It is then easy to see that the same property must also hold for any $(t,x)$ close enough to $(t_1,x_1)$, using the closedness property of  Simon-Smith family for $H$.

\subsection{A fact on linking number}
Finally, before we  end this section, let us record the following fact about homology groups of complements of individual elements in $\cS(S^3)$.

Let $S\in\cS(S^3)$, and we fix an inside direction for it. We can define a bilinear form $$L:H_1(\ins(S))\x H_1(\out(S))\to\Z_2$$ using  the linking number  for $1$-cycles. We first state a standard fact for smooth embedded surfaces, whose proof is included in  Appendix \ref{appendix:Linking}.
\begin{lem}\label{lem:alexanderSmooth}
Let $S$ be a smooth, orientable genus $g$ surface embedded in $S^3$. Then $\rank(L)=\fg(S)$.
\end{lem}

Now, we adapt this fact for singular surfaces:

\begin{lem}\label{lem:alexander}
  For any $S\in\cS(S^3)$,    $\rank L=\fg(S).$
\end{lem}
\begin{proof}
    We first assume $n:=\rank L$ and $\fg(S)$ are both finite. Then we can find  1-cycles $\alpha_1,...,\alpha_n\subset \ins(S)$ and $\beta_1,...,\beta_n\subset \out(S)$  such that $L(\alpha_i,\beta_j)=\delta_{ij}$. Now, choose finitely many balls $B_1,...,B_k$ covering the singularities of $S$ that are not isolated points, with each $\partial B_i$ intersecting $S$ transversely,  such that $S\backslash \cup_{i}B_i$ is a smooth surface with boundary of genus $\fg(S)$. We can perform surgeries on $S$ along the finitely many  loop components of $(\cup_i\partial B_i)\cap S$ (i.e. for each such loop, we remove a thin cylinder in $S$, and glue back two discs to $S$, each looks like the disc on $\partial B_i$ bounded by that loop), such that the resulting surface $S'$ avoids every $\partial B_i$.  Now, remove all components inside the balls $B_i$, and call the resulting surface  $S''$. With some smoothening, we can assume $S''$ a smooth surface of  genus $\fg(S)$. By choosing the balls $B_i$ small, we can also assume that each $B_i$ avoids all loops $\alpha_i,\beta_j$.  Then $\alpha_i\subset\ins(S'')$ and $\beta_j\subset\out(S'')$. Hence, it follows Lemma \ref{lem:alexanderSmooth}that $\fg(S'')\geq n$, and so $\fg(S)\geq n$.

    To prove $g:=\fg(S)\leq \rank(S)$, we pick 1-cycles $\alpha''_1,...,\alpha''_g\subset\ins(S'')$ and $\beta''_1,...,\beta''_g\subset\out(S'')$ such that $L(\alpha''_i,\beta''_j)=\delta_{ij}$: This can be done because of Lemma \ref{lem:alexanderSmooth}. By deforming the loops, we can assume they all avoid the balls $B_i$. Hence, if we reverse the process described above, and obtain $S$ back from $S''$, we can also assume each $\alpha''_i\subset\ins(S)$ and $\beta''_j\subset \out(S)$. This shows $\rank (S)\geq g$. So we have proven $\rank(S)=\fg (S)$.
    
   Finally, using a similar argument, we can show that if one of $\rank(L)$ and $\fg(S)$ is infinite, then so is the other.
\end{proof}

\section{Proof of Theorem \ref{thm:mainTopo}}\label{sect:ProofTopo}

In this section, we prove Theorem \ref{thm:mainTopo}, with the proof of several important ingredients  postponed to later sections. Denote by $S^3$ the unit sphere in $\R^4$, and fix a positive integer $g$. We need to  construct a  Simon-Smith family $\Psi:Y\to \cS_{\leq g}(S^3)$ that cannot be deformed via pinch-off process to become a map into $\cS_{\leq g-1} (S^3)$ 

\subsection{A $(2g+3)$-parameter family $\Psi$}\label{sect:2g+3} The set $S^3$ is defined by the equation $x_1^2+x_2^2+x_3^2+x_4^2=1$ in $\R^4$. Let us first introduce a new   coordinate system on $S^3$. Namely,  let $\bD$ denote the closed unit disc in $\R^2$, and    let $C\subset S^3$ be the great circle given by $x_3=x_4=0$. For each $x=(x_1,x_2,x_3,x_4)\in S^3\backslash C$, we can write $$(x_1,x_2,x_3,x_4)=(x_1,x_2,\sqrt{1-x_1^2-x_2^2}\cos\alpha,\sqrt{1-x_1^2-x_2^2}\sin\alpha)$$
with $(x_1,x_2)$ in the open unit disc $\inte(\bD)$ and $\alpha\in S^1:=\R/2\pi\Z$. Below, we will often use the coordinate $(x_1,x_2,\alpha)$ for $S^3\backslash C$. Let $C^\perp$ be the great circle in $S^3$ given by $x_1=x_2=0$. Moreover, we let $B^n\subset\R^n$ be the closed unit $n$-ball.

We first define a  map $\Phi_5:\RP^5\to\cZ_2(S^3;\Z_2)$.
For each $a=[a_0:a_1:...:a_5]\in\RP^5$, we define $ \Phi_5(a)\subset S^3$ to be the zero set  given by 
\begin{align*} a_0+ a_1x_1+ a_2x_2+ a_3x_3+ a_4x_4+ a_5x_1x_2=0.
\end{align*}
While this map is $\bF$-continuous, it is however not a Simon-Smith family.  Namely, for every $a$ in the set
$$A_{\sing}:= \{[a_1a_2 :a_1 :a_2 :0 :0 :1] \in\RP^5: a_1^2 + a_2^2 < 1\}\,,$$  $\Phi_5(a)$ is given by the zero set $(x_1+a_2)(x_2+a_1)=0$, which is the union of two round (but not necessarily equatorial) spheres that intersect transversely along the circle $\{x\in S^3: x_1=-a_2, x_2=-a_1\}$. Such singularities are not allowed in a Simon-Smith family. Thus, we need to desingularize the intersecting curves. In fact,   we will introduce a $B^{2g-2}$-family of ways to desingularize this family $\Phi_5$.

 Recall $\bD$ is the closed unit disc in $\R^2$. Note the map $ \textrm{int}(\bD)\to A_{\sing}$ defined by  $(a_1,a_2)\mapsto [a_1a_2:a_1:a_2:0:0:1]$ naturally gives a coordinate system for $A_{\sing}$.
 
For simplicity, let $C(a_1,a_2)\subset S^3$  denote the circle $\{x_1=-a_2,x_2=-a_1\}$. Note $C(0,0)=C^\perp$. Let $\eta:[0,1]\to[0,1]$ be a smooth non-increasing function that is  zero   at $1$, and positive but small  in $[0,1)$, to be specified. For any $(a_1,a_2)\in \mathrm{int}(\bD)$, we define the sets
\begin{equation}\label{eq:N1}
    N_1(a_1,a_2):=\{(x_1,x_2)\in\R^2:\|(x_1,x_2)+(a_2,a_1)\|\leq \eta(a_1^2+a_2^2)\}\x (\R/2\pi\Z)
\end{equation}
and
$$N_2(a_1,a_2):=\{(x_1,x_2)\in\R^2:\eta(a_1^2+a_2^2)\leq \|(x_1,x_2)+(a_2,a_1)\|\leq 2\eta(a_1^2+a_2^2)\}\x  (\R/2\pi\Z).$$
Here $\|\cdot\|$ denotes the Euclidean norm.
In fact, using the coordinate system $(x_1,x_2,\alpha)$ for $S^3\backslash C$, we will usually  view these two sets as subsets of $S^3$. Then $N_1(a_1,a_2)$ is a  solid torus   neighborhood of the circle $C(a_1,a_2)$, assuming $\eta$ is small enough.

    Let $\epsilon_1:[0,1]\to[0,1]$ be a smooth non-increasing function that is  zero  at $1$, and positive but small (depending on $\eta$) in $[0,1)$, to be specified.
Define two (neither open nor closed) subsets $A_1,A_2$ in $\RP^5$ as follows (Figure \ref{fig:a2}):
    \begin{align*}
        A_1:=\{[a_0:a_1:a_2:a_3:a_4:1]:&\;\|(a_0-a_1a_2,a_3,a_4)\|\leq \epsilon_1(a_1^2+a_2^2),(a_1,a_2)\in\textrm{int}(\bD)\},\\
A_2:=\{[a_0:a_1:a_2:a_3:a_4:1]:&\;\epsilon_1(a_1^2+a_2^2)\leq\|(a_0-a_1a_2,a_3,a_4)\|\leq 2\epsilon_1(a_1^2+a_2^2),\\&\;(a_1,a_2)\in\inte(\bD)\}.
    \end{align*}
Again $\|\cdot\|$ just denotes the Euclidean norm.  Note  $ A_{\sing}\subset A_1\subset A_2$, but   $\partial A_{\sing} $  is not a subset of $A_1$.
Geometrically, $A_1$ is like a tabular neighborhood of $A_{\sing},$ but the width of the neighborhood goes to 0 as $a\in A_{\sing}$ approaches $\partial A_{\sing}$.
 
 \begin{figure}[h]
        \centering
        \makebox[\textwidth][c]{\includegraphics[width=2in]{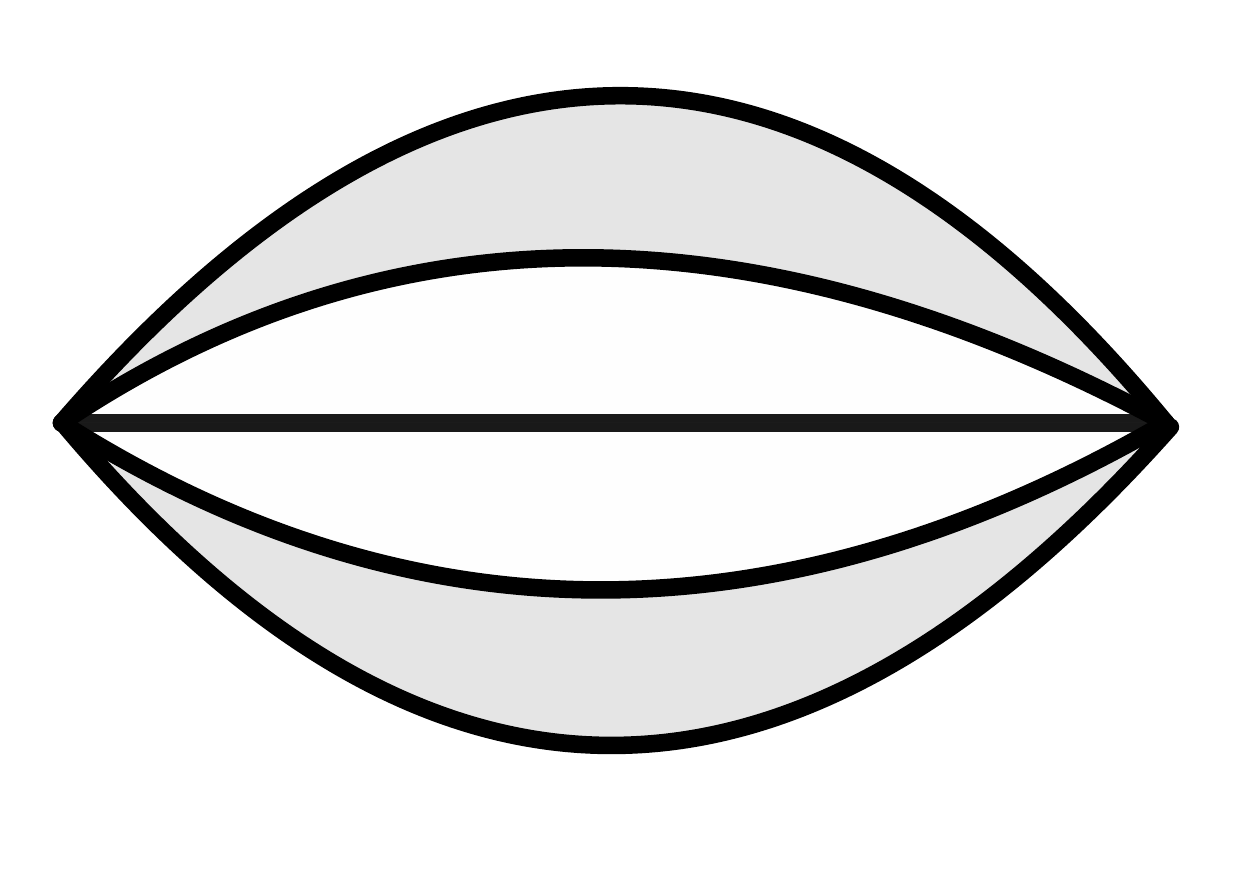}}
        \caption{The horizontal line, with the endpoints removed, is $A_{\sing}$. The white region is $A_1$ and the shaded region is $A_2$.}
        \label{fig:a2}
    \end{figure} 
 
Let $\psi:(A_1\cup A_2) \x S^3\to[0,1]$ be a smooth function such that (see Figure \ref{fig:psi}):
\begin{itemize}
    \item For each $a\in A_1  $ and $x\in  N_1(a_1,a_2)$, $\psi(a,x)=1$.
    \item For each $a\in\partial A_2\backslash\partial A_1$ (which does not include $A_{\sing}$) and $x\in \overline{S^3\backslash N_2(a_1,a_2)}$, $\psi(a,x)=0$.
\end{itemize}

  \begin{figure}[h]
        \centering
        \makebox[\textwidth][c]{\includegraphics[width=2in]{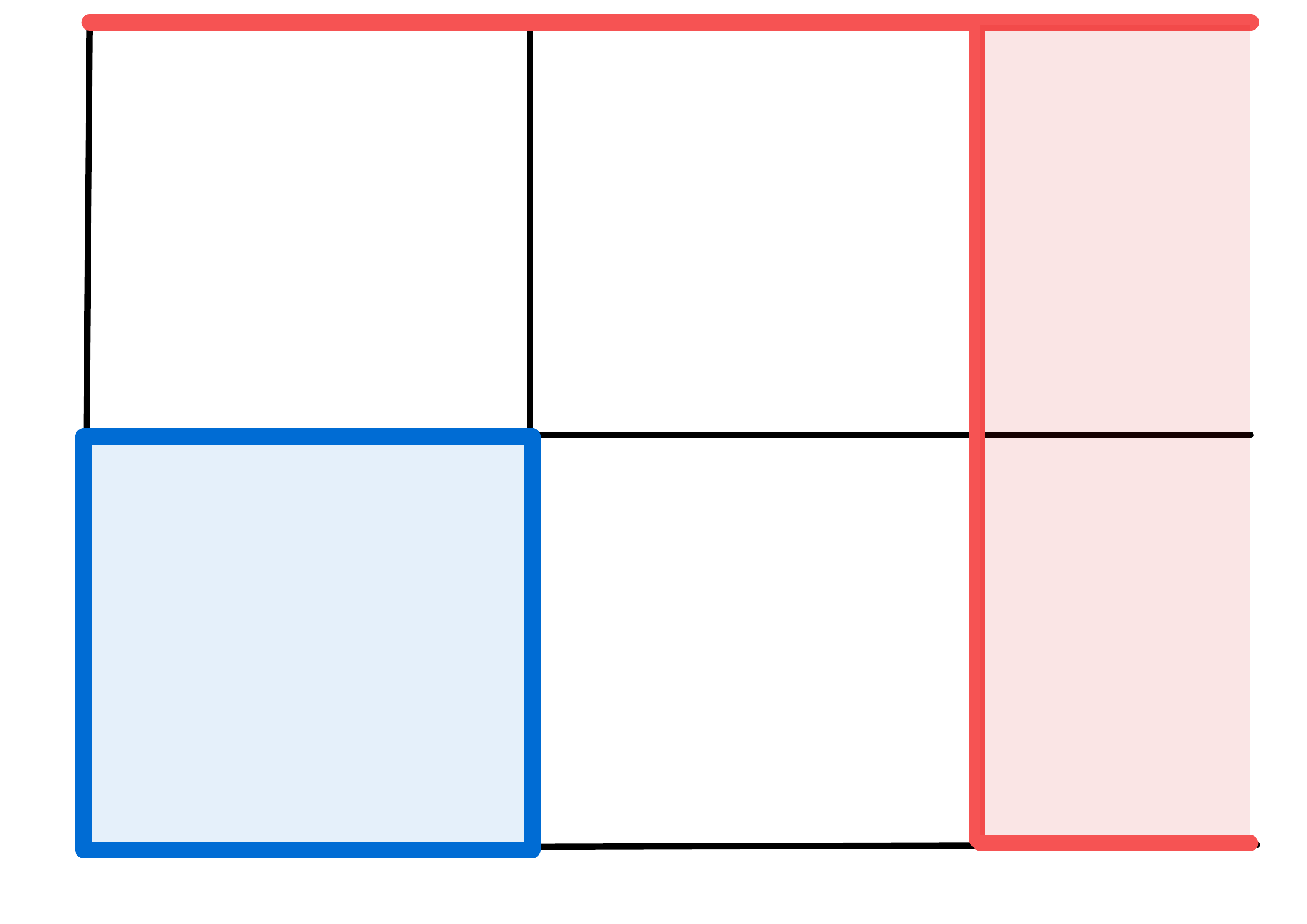}}
        \caption{Withing the left edge, the top segment is $A_2$ while the bottom segment is $A_1$. Within the bottom edge, the left segment is $N_1$, the middle is $N_2$, while the right segment is $\overline{S^3\backslash (N_1\cup N_2)}$. Then the function $\psi$ is 1 on the blue part, and 0 on the red part.}
        \label{fig:psi}
    \end{figure}
  
We first define a map $\hat\Phi_5: \RP^5 \to \cS_0(S^3)$ by modifying $\Phi_5$. 
\begin{itemize}
    \item For $a\in \overline{\RP^5\backslash A_2}$, we just define $\hat\Phi_5(a):=\Phi_5(a)$.
    \item For $a\in A_1\cup A_2 $, we  define $\hat\Phi_5(a)$ to be the zero set of 
\begin{align*}
    &x_1x_2+a_0+a_1x_1+a_2x_2+(1-\psi(a,x))(a_3x_3+a_4x_4)\\
    &+\psi(a,x)\sqrt{1-a_1^2-a_2^2}(a_3\cos\alpha+a_4\sin\alpha).
\end{align*}
\end{itemize}
  
At first glance, it would appear that this family is not continuous near $\partial A_{\sing}$ (which is outside $A_1,A_2$). But as we shall see this is not the case because, as $a\in A_{\sing}$ approaches  $\partial A_{\sing}$, the singular circle $C(a_1,a_2)$ would be very small so that, roughly speaking, the apparent discontinuity would all get pushed into a single point in $S^3$, as allowed in a Simon-Smith family.
\begin{rmk}
Since the functions $\epsilon_1,\eta$ are small, we will see in \S \ref{sect:descrption} that actually $\hat \Phi_5$ is still $\bF$-close to $\Phi_5$ and are ``topologically the same". We defined $\hat\Phi_5$ such that for    $a\in A_1$, and within the solid torus $N_1(a_1,a_2)$,  the set $\hat\Phi_5(a)$ is the zero set of
\begin{equation}\label{eq:boundaryA1}
    x_1x_2+a_0+a_1x_1+a_2x_2+ \sqrt{1-a_1^2-a_2^2}(a_3\cos\alpha+a_4\sin\alpha)=0.
\end{equation} This feature will be convenient for us below to further modify the family, and desingularize the circle of intersection $C(a_1,a_2)$.
\end{rmk}

We will describe a $B^{2g-2}$-family of ways to desingularize the intersection circles $C(a_1,a_2)$. Namely, we are going to define a map $$\Psi:\RP^5\x  B^{2g-2}\to\cS_{\leq g}(S^3).$$
We first let
$${a}=[a_0:a_1:...:a_5]\in\RP^5\quad\textrm{ and }\quad {b}=(b_{2},b'_{2},b_3,b'_3,...,b_{g},b'_g)\in  B^{2g-2}.$$
Then we define $\Psi$ as follows.
\begin{itemize}
    \item For $(a,b)\in \overline{\RP^5\backslash A_1}\x B^{2g-2}$, we just define $\Psi(a,b):=\hat\Phi_5(a).$
    \item For $(a,b)\in  A_1 \x B^{2g-2}$, we define $\Psi(a,b)$ as the union two parts: For the part $\Psi(a,b)\cap \overline{S^3\backslash N_1(a_1,a_2)}$, we define it to be the same as $\hat\Phi_5(a)\cap \overline{S^3\backslash N_1(a_1,a_2)}$.
    \item As for the part $\Psi(a,b)\cap  N_1(a_1,a_2)$, we define it to be the zero set of
    \begin{align}\label{eq:bigExpression}       &\;\;x_1x_2+a_0+a_1x_1+a_2x_2+ \sqrt{1-a_1^2-a_2^2}(a_3\cos\alpha+a_4\sin\alpha)\\
        \nonumber+&\;\epsilon_2(a_1^2+a_2^2)\cdot (\epsilon_1(a_1^2+a_2^2)-\|(a_0-a_1a_2,a_3,a_4)\|)\;\cdot\\
\nonumber&\left[b_2\cos2\alpha+b'_2\sin2\alpha+...+b_g\cos g\alpha+b'_g\sin g\alpha+(1-\|b\|)\cos(g+1)\alpha\right],
    \end{align}
    where $\epsilon_2:[0,1]\to(0,1]$ is some small (depending $\epsilon_1,\eta$) but positive smooth function to be specified. Note, for $a\in \partial A_1\backslash A_{\sing}$, the term $\epsilon_1(a_1^2+a_2^2)-\|(a_0-a_1a_2,a_3,a_4)\|$ is zero, so the above zero set would agree with the one given by (\ref{eq:boundaryA1}).
\end{itemize}
\begin{rmk}
    Here is how we   may  think about the dependence of the expression (\ref{eq:bigExpression}) on $(a,b)$. For simplicity, let us focus on the 3-dimensional ball 
    $$A^0_1:=\{a\in A_1:a_1=a_2=0\}\subset A_1,$$ and consider only $(a,b)\in A^0_1\x B^{2g-2}.$

    As $b$ moves from the center $0\in B^{2g-2}$ towards the boundary $\partial B^{2g-2}$, we see, inside the square bracket of (\ref{eq:bigExpression}), an interpolation from $\cos(g+1)\alpha$ to a trigonometric polynomial of degree $\leq g$.

    Now, let us fix $b$. Then as $a$ moves from the center of $A^0_1$ towards the boundary $\partial A^0_1$, we are interpolating from a trigonomtric polynomial of degree $\leq g+1$, to one of degree $\leq 1$ (so that (\ref{eq:bigExpression}) matches with  (\ref{eq:boundaryA1})).  
\end{rmk}

\begin{thm}\label{prop:Psi1} Fix $g\geq 1$. With the functions $\eta,\epsilon_1,\epsilon_2$   chosen suitably small, the above gives a well-defined Simon-Smith family
$$\Psi:\RP^5\x B^{2g-2}\to\cS_{\leq g} (S^3)$$ with the following properties.
    \begin{enumerate}
    \item\label{Item:Family_Genus0}   $\Psi|_{\overline{\RP^5\backslash A_1}\x B^{2g-2}}$ is of genus $0$.
    \item\label{Item:Family_5-sweepout}   $\Psi|_{\RP^5\x \{0\}}$ is a $5$-sweepout in the Almgren-Pitts sense.
    \item\label{Item:Family_Genus1}  $\Psi|_{\RP^5\x \partial B^{2g-2}}$ is of genus $\leq g-1$.
    \item\label{item:topoSame} For any point $c=( c_1,c_2)\in \inte(\bD)$, we consider the closed $3$-dimensional ball
    $$A^{c}_1:=\{a\in A_1:(a_1,a_2)=(c_1,c_2)\},$$
    and denote by $\inte(A^c_1)$ its interior. Then 
    the $(2g+1)$-parameter families
    $\Psi|_{\inte(A^{c}_1)\x B^{2g-2}}$
    are ``topologically the same" (regardless of $c$) in the following sense. There exists a continuous  map, $$\{P:\inte(A_1)\x B^{2g-2}\to \inte(A^{(0,0)}_1)\x B^{2g-2}\},$$ and a continuous  family of smooth diffeomorphisms, $$\{\varphi:\inte(A_1)\x B^{2g-2}\to\mathrm{Diff}(S^3) \},$$ such that:
    \begin{itemize}
    \item For any $c\in\inte(\bD)$, 
    $P$ maps $\inte(A^c_1)\x B^{2g-2}$ homeomporhically onto $\inte(A^{(0,0)}_1)\x B^{2g-2}$.
        \item For any $(a,b)\in A_1\x B^{2g-2}$,
    $$\Psi(P (a,b))=\varphi(a,b)(\Psi(a,b)).$$
    \end{itemize}
    \item \label{item:areaBound}
    Under the standard unit metric on $S^3$, we can require $\Psi(a,b)$ and $\Phi_5$ to be arbitrarily $\bF$-close for every $(a,b)$, uniformly in $(a,b)$,  by choosing $\epsilon_1,\epsilon_2$ to be sufficiently small. 
    \end{enumerate}
\end{thm}

We will prove Theorem \ref{prop:Psi1} in   \S \ref{sect:ProofPsi1}.

\begin{rmk}\label{rmk:epsilon12}
As will be clear in the proof, we can assume  the functions $\eta,\epsilon_1,\epsilon_2$ to be arbitrarily small, though $\epsilon_2$ would depend on $\epsilon_1$, and $\epsilon_1$ on $\eta$. Of course, they all depend on $g$.
\end{rmk}

\subsection{Applying Theorem \ref{thm:TopoMinMax}} Let us now prove Theorem \ref{thm:mainTopo} by contradiction: Letting $Y$ be the parameter space $\RP^5\x B^{2g-2}$, we  suppose there exists some   Simon-Smith family $H:[0,1]\x Y\to \cS_{\leq g}$ such that: 
    \begin{itemize}
        \item $H(0,\cdot)=\Psi$.
        \item $\Psi':=H(1,\cdot)$ maps into $\cS_{\leq g-1}$.
        \item For each $y\in Y$, $t\mapsto H(t,y)$ is a pinch-off process.
    \end{itemize}
Let $Z_0$ (resp. $Z_{\geq 1})$ be the set of $y\in Y$ such that $\Psi'(y)$ has genus $0$ (resp. in the range $[1,g-1]$). 
    
Below, all chains, homology groups, and cohomology groups are assumed to have $\Z_2$-coefficients.
Now, since $\cS_{\leq g}(S^3)\subset \cZ_2(S^3;\Z_2)$, we can view $\Psi$ as a map into  $\cZ_2(S^3;\Z_2)$ as well. Define the pullback $\lambda:=\Psi^*(\bar\lambda)$ in $H^1(Y)$.  Now, it is easy to check that, for some (and thus all) homotopically non-trivial loop $\gamma\subset Y$, the subfamily $\Psi|_\gamma$ is a 1-sweepout, e.g. let $\gamma$ be the loop given by 
$$\{[a_0:a_1:0:...:0]\}\x\{0\}\subset Y.$$  Hence, the class $\lambda\in H^1(Y)=\Z_2$ is actually the unique non-trivial element. This immediately implies the cup product element $\lambda^5\in H^5(Y)=\Z_2$  is non-zero, and is such that the pairing $\lambda^5([\RP^5\x\{0\}])\ne 0$. 
\begin{rmk}
    By considering intersection number, the Poincar\'e dual $PD(\lambda^5)\in H_{2g-2}(Y,\partial Y)$ of $\lambda^5$ is given by the element $[\{O_1\}\x B^{2g-2}]$, where $O_1$ denotes the point $[0:...:0:1]\in \RP^5$. Equivalently, we have the following relation regarding  relative cap product:
\begin{equation}\label{eq:IntroCap}
    [Y]\frown\lambda^5=[\{O_1\}\x B^{2g-2}]\in H_{2g-2}(Y,\partial Y).
\end{equation}
Note $[Y]$ denotes the fundamental class in $H_{2g+3}(Y,\partial Y)$. Readers may refer to \cite[p.240]{hatcher2002book} for the definition of relative cup product.
\end{rmk}

On the other hand, recalling that no Simon-Smith family of genus 0 in $S^3$ could be a 5-sweepout, we know that the restriction  $\lambda^5|_{Z_0}$ given by pullback onto $Z_0$  is $0$. Then, we claim that by a purely elementary algebraic topological argument, we can show that:
\begin{lem}\label{lem:capProd}
    There exists some $(2g-2)$-subcomplex $D\subset Z_{\geq 1}$ with $\partial D\subset Z_{\geq 1}\cap \partial Y$ such that
$$[D]=[Y]\frown \lambda^5\in H_{2g-2}(Y,\partial Y).$$
\end{lem} 
We postpone the proof to \S \ref{sect:ProofCapProd}. 



In summary, we have:
\begin{enumerate}
\item $D$ is a $(2g-2)$-subcomplex   with $\partial D\subset  \partial Y$.
\item $[Y]\frown \lambda^5=[D]\in H_{2g-2}(Y,\partial Y).$
    \item Since each member of $\Psi'|_{D}$ has genus in the range $[1,g-1]$, and pinch-off process is genus-non-increasing, we also know that each member of $H|_{[0,1]\x D}$ has genus in the range $[1,g]$.
\end{enumerate}
We are going to derive a contradiction using these three items, and thereby finish the proof of  Theorem \ref{thm:mainTopo}.  

\subsection{A certain non-trivial cycle $\partial\sigma_3$} In this section, we will extract from $Y$ a $(2g-2)$-chain $\sigma_3$ that has some special feature.

Let us denote by $\sigma_1:D\to Y$ the $(2g-2)$-chain represented by $D$.
Since $\Psi(a,b)$ has genus 0 whenever $a\in \overline{\RP^5\backslash A_1}$ by Theorem \ref{prop:Psi1}, we know that   $\sigma_1$ maps into $ \inte(A_1)\x B^{2g-2}$ with  $\partial\sigma_1$ in $ \inte(A_1)\x\partial B^{2g-2}$ (recall $\inte(\cdot)$ denotes the interior of a set). 
From $[Y]\frown \lambda^5=[\sigma_1]$, it is easy to check that within  $\overline{A_1}\x B^{2g-2}$ (topologically a closed $(2g+3)$-ball),  $\sigma_1$ and $\overline{A_1}\x   0$ have intersection number 1 (mod 2) (possibly after perturbation): Note the boundaries of these two sets must be disjoint, because $\Psi|_{\partial A_1\x 0}$ has genus 0 but members of  $\Psi|_{\partial \sigma_1}$ have genus in the range $[1,g-1]$. See Figure \ref{fig:A2xB}. Denote  $$A^0_1=A^{(0,0)}_1:=\{a\in A_1:(a_1,a_2)=(0,0)\},$$  which is a 3-dimensional closed ball in $A_1$. By Theorem \ref{prop:Psi1} (\ref{item:topoSame}), there is a {\it strong} deformation retraction  of $\inte(A_1)\x B^{2g-2}$ onto $ \inte(A^0_1)\x B^{2g-2}$ (the word strong means $ \inte(A^0_1)\x B^{2g-2}$ is fixed throughout the retraction), $$F_1:[0,1]\x (\inte(A_1) \x B^{2g-2})\to \inte(A_1)  \x B^{2g-2},$$  such that:
\begin{itemize}
    \item For each $t\in [0,1]$, $F_1(t,\cdot)$ maps $\inte(A_1)\x \partial B^{2g-2}$  into itself.
    \item For each fixed $y$, the genus of $\Psi(F_1(t,y))$ is the same for every $t\in [0,1]$. 
\end{itemize}
\begin{figure}
        \centering
        \makebox[\textwidth][c]{\includegraphics[width=3in]{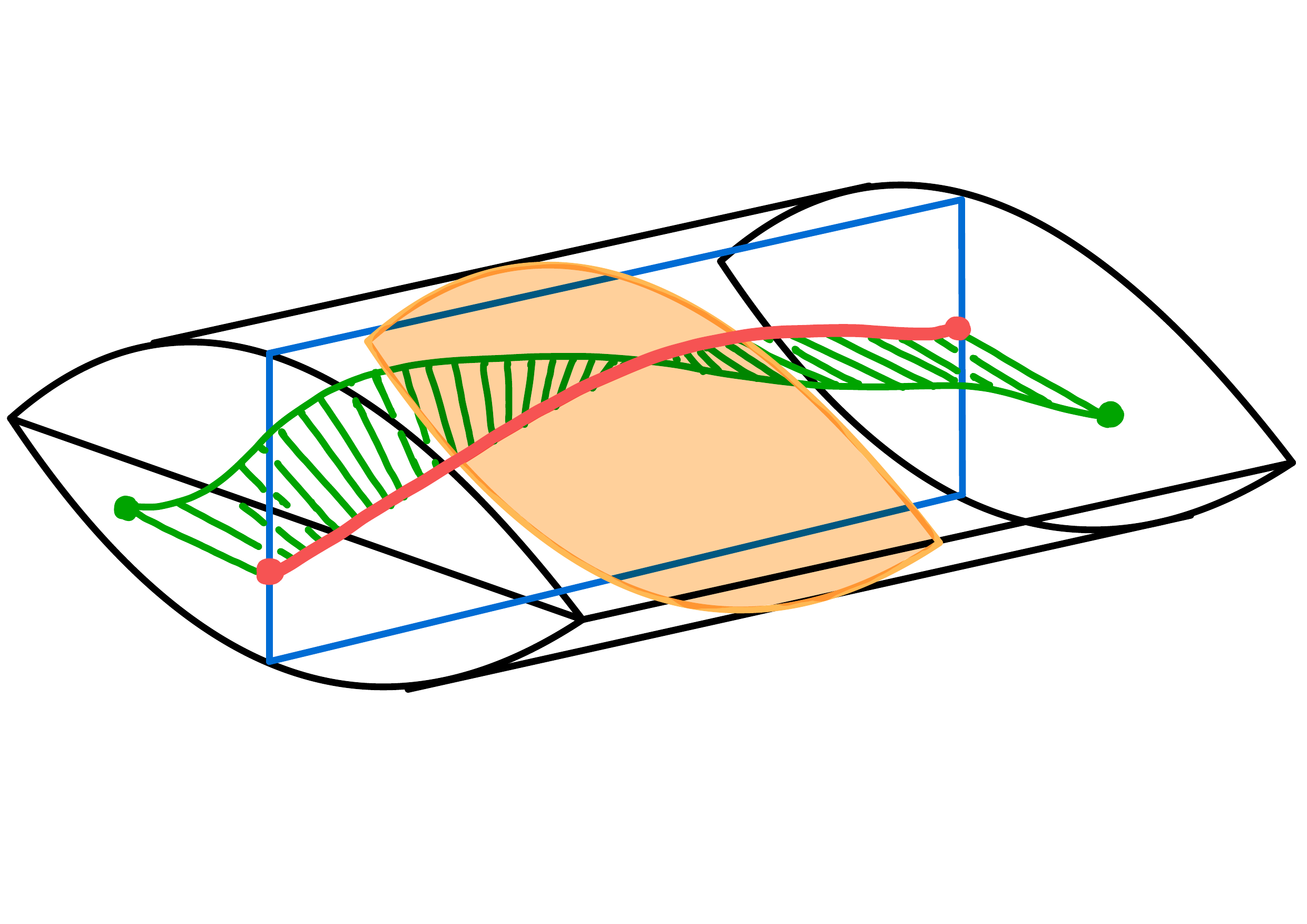}}
        \caption{This shows $A_1\x B^{2g-2}$. The orange slice is $A_1\x 0$. The plane with blue boundary is $A^0_1\x B^{2g-2}$. The green curve is $\sigma_1$, which can be deformed (following the green straight lines) to the    red curve, denoting  $\sigma_2:= F_1(1,\cdot)\circ \sigma_1.$ }
        \label{fig:A2xB}
    \end{figure}

Denote $X:= A^0_1\x B^{2g-2}$, which is homeomporhic to a closed ball, and
  consider the singular  $(2g-2)$-chain $\sigma_2:=F_1(1,\cdot)\circ \sigma_1$ in $X$. Note $\partial\sigma_2$ is in $ \partial X$ based on our definition of $F_1$. Moreover, from the fact that $\sigma_1$ and $\overline{A_1}\x   0$ have intersection number 1  in $\overline{A_1}\x B^{2g-2}$, we can immediately deduce:
\begin{lem}\label{lem:sigma2A01IntersectionNumber}
In $X$, $\sigma_2$ and $ A^0_{1}\x 0$ still have intersection number $1$ (mod $2$) (possibly after  perturbing $\sigma_2$), and their boundaries  are disjoint.
  \end{lem}

Now, for each $k=0,1,...,g$, define the  sets $$X_k:=\{y\in X:\fg(\Psi(y))=k\}, \quad X_{\geq k}:=\{y\in X:\fg(\Psi(y))\geq k\}.$$
Note $X_g$ is an open set, and $\overline{X_g}\subset \inte(X)$, as $\Psi|_{\partial X}$ has genus $\leq g-1$. In the following, we would want to focus on the subset $\overline{X_g}$.

\begin{prop}\label{prop:DeformRetractF2} The subset $X_g\subset X$ is an open   $(2g+1)$-dimensional region with piecewise smooth, algebraic boundary, and  $\overline{X_g}$ is topologically a closed $(2g+1)$-ball. 

Moreover,
    there exists a deformation retraction of $\mathrm{int}(X)$ onto $\overline{X_g}$,
    $$F_2:[0,1]\x \mathrm{int}(X)\to \mathrm{int}(X),$$
    such that:
    \begin{itemize}
        \item For each fixed $y$, the genus of $\Psi(F_2(t,y))$ is the same for all $t$. 
        \item The map $F_2(t,\cdot)$ fixes subset $\overline{X_g}$ for each $t$.
        \item For every $t$ and $y\in \mathrm{int}(X)\backslash{X_g}$, we have $F_2(t,y)\in \mathrm{int}(X)\backslash{X_g}$.
    \end{itemize}
\end{prop}
We will prove this proposition in \S \ref{sect:ProofDeformRetractF2}.

Now, using $\sigma_2$ and $A^1_0\x 0$, and the deformation retraction $F_2$, let us define two new chains $\sigma_3$ and $\tau''$ in $\overline{X_g}$:
\begin{itemize}
\item First, we let $\sigma_3:= \sigma_2\cap \overline{X_g}$, which, via some perturbation, can  still be assumed to be a singular chain because $X_g$ has piecewise smooth, algebraic boundary (Proposition \ref{prop:DeformRetractF2}).
\item Let $\tau$ denote the 3-dimensional subset $A^1_0\x 0$. We may view $\tau$ as a simplicial 3-chain. Note $\partial \tau\subset X_0$, as the genus of members $\Psi|_{\partial\tau}$ can be directly checked via inspecting the expression  (\ref{eq:bigExpression}). Now, let us push $\tau$ inward slightly, near its boundary, to obtain a new 3-chain $\tau'$ that is disjoint from $\partial X$:  We can  assume throughout this  deformation process, the {\it boundary} of the 3-chains all still lie in $X_0$, {\it and in particular avoid $\sigma_2$}.
\item Then, we deform $\tau'$ by  the deformation retraction $F_2$, and consider the  3-chain $\tau'':=F_2(1,\cdot)\circ \tau$, which maps into $\overline{X_g}$. Note,  by the first bullet point in Proposition \ref{prop:DeformRetractF2}, we again know that throughout the deformation process, the boundary of the 3-chains all still lie in $X_0$, {\it and in particular avoid $\sigma_2$}. See Figure \ref{fig:X} for an illustration.
\end{itemize} 
As direct consequences, we have:
\begin{itemize}
    \item $\partial \tau''$ maps into $\partial X_g\cap X_0$ while $\partial \sigma_3$ maps into $\partial X_g\cap X_{\geq 1}$. In particular, they are disjoint. 
    \item $\sigma_3$ and $\tau''$  have intersection number 1 in the $(2g+1)$-ball $\overline{X_g}$, using Lemma \ref{lem:sigma2A01IntersectionNumber}.
\end{itemize}
From these two points,  we immediately know
\begin{equation}\label{eqLsigma2NonTri}
    [\partial\sigma_3]\ne 0\quad\textrm{ in }\quad H_{2g-3}(  \partial X_{g}\cap X_{\geq 1}).
\end{equation} 

\begin{figure}
        \centering
        \makebox[\textwidth][c]{\includegraphics[width=2.5in]{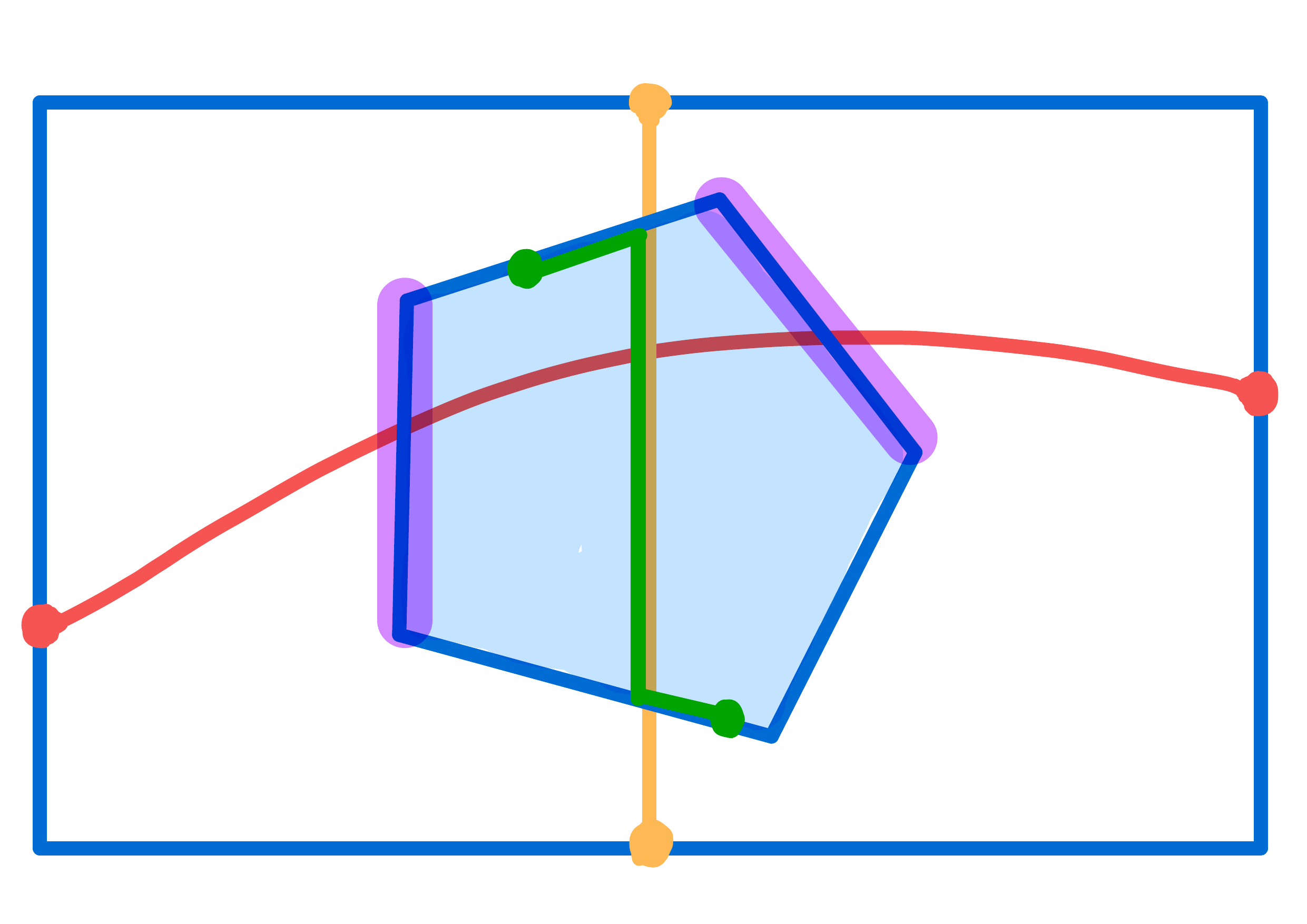}}
        \caption{The whole blue box is $X$, and the blue shaded region is $\overline{X_g}$. The red curve is $\sigma_2$, and the yellow line is $\tau=A^0_1\x 0$, which is deformed into $\tau''$, denoted by the green segments. The purple part is $\partial X_g\cap X_{\geq 1}.$}
        \label{fig:X}
    \end{figure}

\subsection{Identification with $\Gr(\Z^g_2)\x\Gr(\Z^g_2)$}\label{sect:identificationGr}
In this section, we begin to relate the homology classes of the complement regions of members of $\Psi|_{\overline{X_g}}$.

Recall that for $\fH (\Psi|_{\overline{X_g}})$, there is a natural projection map $\proj_1:\fH (\Psi|_{\overline{X_g}})\to \overline{X_g}$ such that for each $y\in \overline{X_g}$, its preimage set can be canonically identified with the group $H_1(S^3\backslash \Psi(y))$. Moreover, from the definition of Simon-Smith family, and that $\overline{X_g}$ is a subset of the contractible set $$(\RP^5\backslash \{a_5=0\})\x B^{2g-2},$$ we know that the family $\Psi|_{\overline{X_g}}$ admits some continuous choice of inside and outside regions for $\Psi(y)$, for every $y\in\overline{X_g}$. 

For each $y$, let $$L:H_1(\ins(\Psi(y)))\x H_1(\out(\Psi(y)))\to \Z_2$$ be the bilinear form given by the linking number for 1-cycles. Moreover we consider the vector space $\Z_2^g\oplus \Z_2^g$ (which is isomorphic to $H_1(S^3\backslash \Sigma_g)$ where   $\Sigma_g$ denotes any standard Heegaard surface of genus $g$ in $S^3$), and let $I_\id:\Z^g_2\oplus\Z^g_2\to\Z_2$ denote the  standard bilinear form on it, given by the $g\x g$ identity matrix. Let us  equip  $\Z_2^g $ with the discrete topology. Then the following proposition relates the homology classes of $S^3\backslash\Psi(y)$ for $y\in\overline{X_g}$.  
\begin{prop}\label{prop:embedHomology}
There exist  continuous maps\footnote{$\frak i$ stands for identification.} $$\mathfrak i_{\ins}:\fH _{\ins}(\Psi|_{\overline{X_g}})\to \Z_2^g.\quad \mathfrak i_{\out}:\fH _{\out}(\Psi|_{\overline{X_g}})\to \Z_2^g$$ 
such that for each $y\in \overline{X_g}$:
\begin{itemize}
    \item  $\frak i_{\ins}$ gives an isomorphism of $H_1(\ins(\Psi(y)))$ onto its image.
    \item $\frak i_{\out}$ gives an isomorphism of $H_1(\out(\Psi(y)))$ onto its image.
    \item Let $\frak i :=(\frak i_{\ins},\frak i_{\out})$, and consider the isomorphism from $H_1(\ins(\Psi(y)))\oplus H_1(\out(\Psi(y)))$ onto its image given by  $\frak i$. Then under this isomorphism, the linking number bilinear form
$$L:H_1(\ins(\Psi(y)))\oplus H_1(\out(\Psi(y)))\to\Z_2$$
is sent to the standard bilinear form $I_\id $ restricted onto the image  $$\frak i_{\ins}(H_1(\ins(\Psi(y))))\oplus \frak i_{\out}(H_1(\out(\Psi(y)))).$$
\end{itemize}
\end{prop}
We postpone the proof to \S  \ref{sect:ProofEmbedHomology}.
Now, by passing to the Grassmannians,  the maps $\frak i_{\ins},\frak i_{\out}$  naturally induce, respectively, two maps 
$$\frak i_{\Gr,\ins}: \fGr_\ins(\Psi|_{\overline{X_g}})\to \Gr(\Z^{g}_2), \quad \frak i_{\Gr,\out}: \fGr_\out(\Psi|_{\overline{X_g}})\to \Gr(\Z^{g}_2).$$
Then,  we can define a map $${\frak f} :\overline{X_g}\to \Gr(\Z^g_2)\x \Gr(\Z^g_2)\; $$
  that sends each $y$ to the pair of subspaces
$$\left( \frak i_{\Gr,\ins}(y, H_1( \ins(\Psi   (y)))),\frak i_{\Gr,\out}(y, H_1( \out(\Psi (y))))\right). $$
Note, for example, here $H_1( \ins(\Psi (y)))$ is viewed as an element of its own Grassmannian. Since $\frak i_\ins,\frak i_\out$ are continuous, we can easily check the map $\frak f$ is continuous: In essence, it all boils down to the simple fact that (by the closedness property of Simon-Smith family) if $\gamma_0$ is a loop in $S^3\backslash\Psi(y_0)$, then it is also a subset of $S^3\backslash\Phi(y)$ for any $y$ close enough to $y_0.$

For any chain $\sigma$ in $\overline{X_g}$, which can be viewed as a map from $\dmn(\sigma)$ into  $\overline{X_g}$, we   define the map  $\frak f_\sigma:=\frak f\circ\sigma$. Now we take $\sigma:=\sigma_3$. Since each member of   $\Psi\circ \partial\sigma_3$ has genus  in the range $[1,g-1]$, we know by Lemma \ref{lem:alexander}  that the linking number bilinear form 
$$L:H_1(\ins(\Psi\circ\partial\sigma_3(y)))\x H_1(\out(\Psi\circ\partial\sigma_3(y)))\to\Z_2$$
has rank within the range  $[1,g-1]$. Thus, by the properties of $\frak i$ described in Proposition \ref{prop:embedHomology}, we in fact know that the image of ${\frak f}_{\partial \sigma_3}$ is a subset of
 $\Gr^{g}[1,g-1]$. Now,  ${\frak f}_{\partial \sigma_3}$ can be treated as  a $(2g-3)$-cycle. Now, we state an important topological fact about this cycle,   for which our whole proof relies on.  

\begin{thm}\label{thm:partialFSigma2}
     $[ {\frak f}_{\partial\sigma_3}]\ne 0$ in $H_{2g-3}(\Gr^g[1,g-1])$.
 \end{thm}
Note this theorem is not trivially false because $\frak f_{\sigma_3}$ is {\it not} a map into $\Gr^g[1,g-1]$. Indeed, some members of $\Psi\circ{\sigma_3}$ may have genus $g$.
 We will prove this theorem in \S  \ref{sect:ProofPartialFSigma2}.

\subsection{Homology descent for $\Psi|_{\overline{X_g}}$}

Recall that we have a    deformation by pinch-off process, $H:[0,1]\x Y\to\cS_{\leq g}$,  for which $H(0,\cdot)=\Psi$, and  $\Psi':= H(1,\cdot)$ maps into $\cS_{\leq g-1}$. We will still focus on the region $\overline{X_g}\subset Y$, on which $\Psi$ has a consistent choice of inside and outside regions. Hence, to simplify notation, let us in this section abuse notation and assume $H$ has domain $[0,1]\x \overline{X_g}$,

Now, we apply the terminology established in \S \ref{sect:descent}, and consider   the maps 
$${\frak b} ^{H }_\ins:[0,1]\x \overline{X_g}\to \fGr_\ins(\Psi|_{\overline{X_g}}), \quad{\frak b}_\out^{H}:[0,1]\x \overline{X_g}\to \fGr_\out (\Psi|_{\overline{X_g}}).$$
For example, remember that ${\frak b} ^{H }_\ins$ sends each $(t,y)$ to the set of all elements $(y,c)\in \{y\}\x H_1(\ins(\Psi(y)))$  such that $c$ would not have  terminated by time $t$  under the pinch-off process $H(\cdot,y)$ (recall that,  ``$c $  has not terminated by time $t$"  means $c$ descends to some homology class in $H_1(\ins(H(t,y)))$).

Hence, we can define a map
$${\frak F}:[0,1]\x \dmn(\sigma_3)\to \Gr(\Z^g_2)\x \Gr(\Z^g_2)$$
by $$\fF(t,y):= \left(\frak i_{\Gr,\ins}(\fb^{H}_\ins(t,\sigma_3(y))),\frak i_{\Gr,\out}(\fb^{H}_\out(t,\sigma_3(y)))\right).$$
Note, by definition, $\fF(0,\cdot)={\frak f}_{\sigma_3}$. Again, it is easy to check that this map is continuous. 

Finally, we are ready to derive a contradiction. We recall that the genus of $H(1,  y)$ is in the range $[1,g-1]$ for each $y\in\dmn( \sigma_3) $, and the genus of $H(t, y)$ is also in the range $[1,g-1]$ for each $y\in\dmn( \partial\sigma_3)$ and $t\in[0,1]$ (this uses the fact that pinch-off process is genus-non-increasing). As a result, using Lemma \ref{lem:alexander} and the properties of $\frak i$ in Proposition \ref{prop:embedHomology}, we know  the following $(2g-2)$-chain actually lies in $\Gr^g[1,g-1]:$
$$\fF|_{\{1\}\x\dmn(\sigma_3)}+\fF|_{[0,1]\x\dmn(\partial\sigma_3)}.$$
Now, the boundary of this  $(2g-2)$-chain is equal to $\fF|_{\{0\}\x\dmn(\partial\sigma_3)}$, which is the same as the  $(2g-3)$-cycle  ${\frak f}_{\partial \sigma_3}$. This shows $[{\frak f}_{\partial \sigma_3}]$ is  actually trivial in $H_{2g-3}(\Gr^g[1,g-1])$, contradicting Theorem \ref{thm:partialFSigma2}. This finishes the proof of Theorem \ref{thm:mainTopo}.

 \section{Proof of Theorem \ref{thm:mainMinimal}}\label{sect:mainProofMinimal}
 Let $(S^3,\bg_0)$ be the given 3-sphere of positive Ricci curvature. 
 First, recall by Remark \ref{rmk:firstWidth} that the first Simon-Smith width $\sigma_1(S^3)$ is equal to  the area of some  embedded minimal sphere $\Sigma$  in $(S^3,\bg_0)$.   Then by Theorem \ref{thm:TopoMinMax}, to prove Theorem \ref{thm:mainMinimal}, it suffices to construct, for  any constant $\bar\delta>0$ and any genus $g$, a Simon-Smith family $\Xi:X\to\cS_{\leq g}$ such that:
\begin{itemize}
    \item $\Xi$ cannot be deformed via pinch-off processes to a map into $\cS_{\leq g-1}.$ 
    \item For any $x\in X$ with $\fg(\Psi(x))=g$, we have $\area(\Psi(x))<2\area(\Sigma)+\bar\delta.$
\end{itemize} 

\subsection{A  family $\Psi^\delta$ in the unit 3-sphere}
Recall that we defined $$A_{\sing}:=\{[a_1a_2:a_1:a_2:0:0:1]:a^2_1+a^2_2<1\},$$
and for each $a=[a_1a_2:a_1:a_2:0:0:1]\in\RP^5$, $\Phi_5(a)$ is defined to be the union of two spheres, given by $(x_1+a_2)(x_2+a_1)=0$.
 Fix any $0\leq \delta<1$. It is possible to find a continuous $A_{\sing}$-family of diffeomorphisms on $S^3$, 
 $\rho_\delta:A_{\sing}\to \textrm{Diff}(S^3)$, such that for each $a$ the diffeomorphism $\rho_\delta(a)$  maps $\Phi_5(a)$ to the zero set given by \begin{equation}\label{eq:pushed}
     \left((x_1+a_2)+\delta(x_2+a_1)\right)\left((x_2+a_1)+\delta(x_1+a_2)\right)=0.
 \end{equation}
 Geometrically, as $\delta$ varies from 0 to 1, $\rho_\delta(a)$ pushes both spheres in $\Phi_5(a)$ towards the sphere given by $(x_1+a_2)+(x_2+a_1)=0$: See Figure \ref{fig:rho}. Note the singular circle in $\Phi_5(a)$ and in $\rho_\delta(a)(\Phi_5(a))$ are the same.

    \begin{figure}
        \centering
        \makebox[\textwidth][c]{\includegraphics[width=3in]{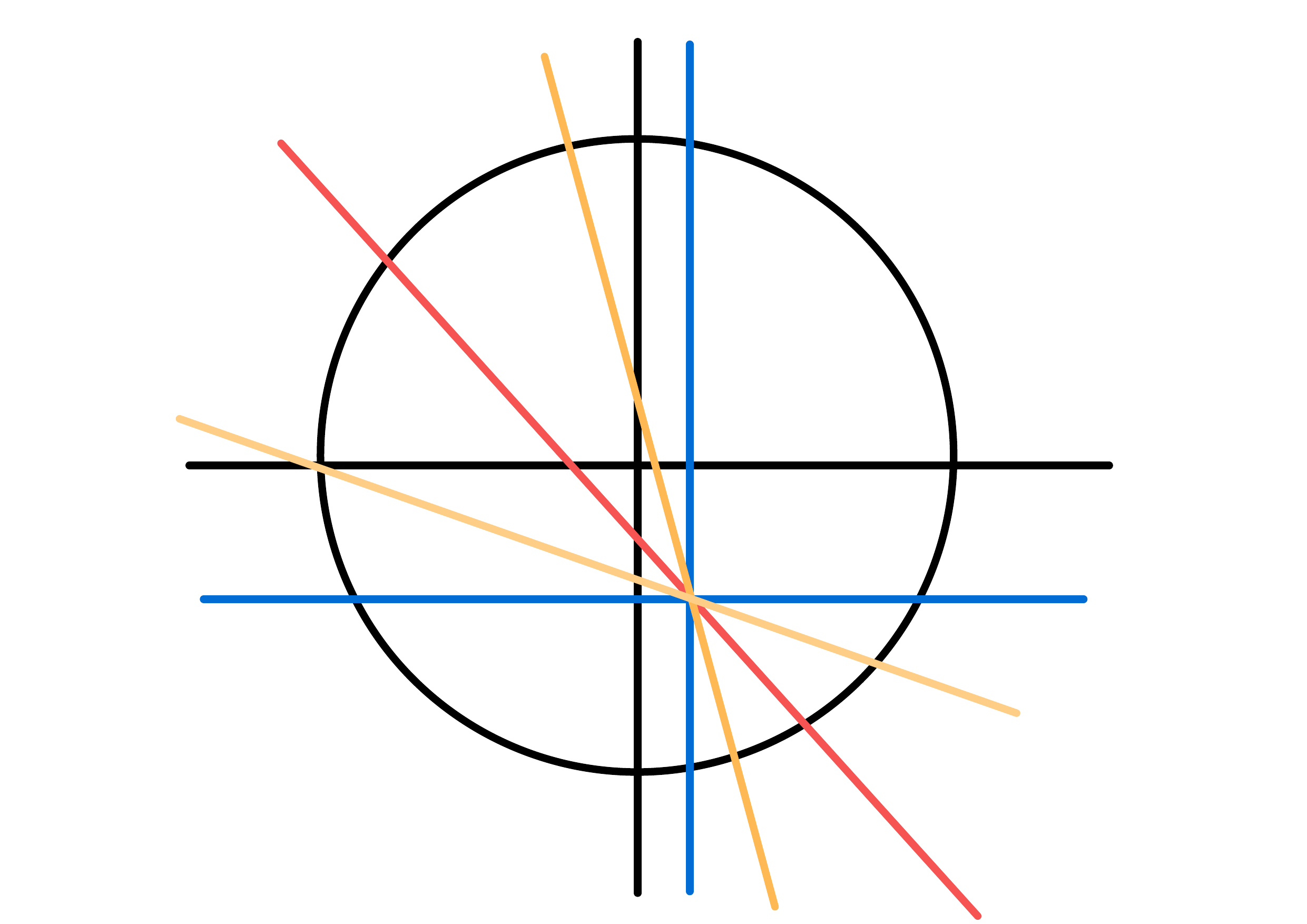}}
        \caption{We just focus on the $x_1,x_2$ coordinates.
        The black lines are the $x_1,x_2$-axes, and the circle is the unit circle. The two blue lines are given by $(x_1+a_2)(x_2+a_1)=0$. Given a fixed $\delta$, they are pushed by $\rho_\delta(a)$ to the orange lines given by (\ref{eq:pushed}). The red line is given by $(x_1+a_2)+(x_2+a_1)=0$.}
        \label{fig:rho}
    \end{figure}

 Now, let us extend $\rho_\delta$ to become a map $\rho_\delta:\RP^5\to\textrm{Diff}(S^3)$: Fix some small constant $\epsilon_3>0$ and pick an  $\epsilon_3$-neighborhood $A_3$ containing $A_{\sing}$ (the exact metric we put on $\RP^5$ would not matter). Let $\rho_\delta:=\id$ outside $A_3$, and let $\rho_\delta$ interpolate between $\rho_\delta|_{A_{\sing}}$ and $\id$ on the intermediate region. Then, we define a map $\Phi_5^\delta:\RP^5\to \cZ_2(S^3;\Z_2)$ by   $$\Phi^\delta_5(a):=\rho_\delta(a)(\Phi_5(a)).$$
 Since $\rho_\delta$ is a continuous family of diffeomorphisms,  the two families $\Phi^\delta_5$ and $\Phi_5$ are ``topologically the same".

 Now, we use $\rho_\delta$ to modify $\Psi$ too. For any constant $0\leq \delta<1$, we define $$\Psi^{ \delta}:\RP^5\x B^{2g-2}\to \cS_{\leq g}$$
 by $\Psi^{\delta}(a,b):=\rho_\delta(a)(\Psi(a,b)).$ 
 \begin{rmk}\label{rmk:closeToMulti2Sphere}
Using Theorem \ref{prop:Psi1} (\ref{Item:Family_Genus0}) and (\ref{item:areaBound}), by choosing the functions $\epsilon_1,\epsilon_2$ (in the definition of $\Psi$) to be sufficiently small, we can assume that any  member of $\Psi$ with non-zero genus is uniformly arbitrarily $\bF$-close to at least one member of $\Phi_5|_{A_{\sing}}$. In fact, by further choosing $\epsilon_1,\epsilon_2$ to be sufficiently small (depending on the definition of $\rho_\delta$ on $A_3$), we can assume any  member of $\Psi^\delta$ of non-zero genus is uniformly arbitrarily $\bF$-close to at least one member of $\Phi^\delta_5|_{A_{\sing}}$, which is defined using (\ref{eq:pushed}).

Note, as $\delta\to 1$, the equation (\ref{eq:pushed}) ``approaches" the equation
\begin{equation}\label{eq:multi2Sphere}
    ((x_1+a_2)+(x_2+a_1))^2=0.
\end{equation}
In other words, all members of  $\Phi^\delta_5|_{A_{\sing}}$, and thus all members of $\Psi^\delta$ with  non-zero genus,  will be close as varifolds to ``the sphere (\ref{eq:multi2Sphere}) with multiplicity two".
 \end{rmk}

\subsection{A family $\Xi$ in $(S^3,\bg_0)$}
Now, by Theorem 1.8 in Haslhofer-Ketover's work \cite{HK19},  there exists a  foliation $\{\Sigma_t\}_{t\in [-1,1]}$ such that:
\begin{itemize}
\item $\Sigma_{-1}$ and $\Sigma_1$ both consist of just a point, while $\Sigma_t$ are smooth embedded spheres for $t\in (-1,1)$.
\item $\Sigma_{0}$ is the given minimal sphere $\Sigma$.
\item For any $t\in (0,1)$ (resp. $t\in(-1,0)$), $\Sigma_t$ is a smooth mean convex (resp. mean concave) sphere with area strictly less than $\area(\Sigma)$.
\end{itemize}
The construction of this foliation uses mean curvature flow with surgery: See \cite{haslhoferKleiner2017MCFsurgery,brendleHuisken2018MCFsurgery,buzanoHaslhoferHershkovits2021moduli}.

Now, let us consider two foliations. First, we have the above foliation $\cF_1$ in $(S^3,\bg_0)$. Second, in the {\it unit $3$-sphere}, we consider  the following foliation $\cF_2$ by parallel round spheres:
$$\{x_1+x_2+s=0\}_{s\in [-\sqrt2,\sqrt 2]}.$$
Note, $s=\pm\sqrt 2$ correspond to antipodal points in the unit 3-sphere.

\begin{rmk}\label{rmk:foliation}
By choosing   $\delta$ to be sufficiently close to $1$, we can ensure that for any $(a_1,a_2)\in \bD$, the solution set of (\ref{eq:multi2Sphere}) is arbitrarily close as varifolds to some leaf of  $\cF_2$ with multiplicity two (uniformly in $a_1,a_2$).
\end{rmk}

Now, we choose a diffeomorphism $\varphi:S^3\to S^3$  that sends $\cF_2$ to  $\cF_1$, which means $\varphi$ injectively sends  each leaf of $\cF_2$ to some leaf of $\cF_1$. 

With the above preparation, we can now define the desired family $\Xi$ for proving  Theorem \ref{thm:mainMinimal}. Recall that we are given some metric $\bg_0$,  embedded minimal sphere $\Sigma$, positive integer $g$, and $\bar\delta>0$. Now, we define a Simon-Smith family $\Xi:\RP^5\x B^{2g-2}\to\cS_{\leq g}$ by
$$\Xi(a,b):=\varphi(\Psi^\delta(a,b))=\varphi\circ\rho_\delta(a)\circ\Psi(a,b).$$
Using Remark \ref{rmk:closeToMulti2Sphere} and \ref{rmk:foliation}, it is easy to check that, by choosing $\delta$ to be sufficiently close to 1, and also $\epsilon_1,\epsilon_2$ to be sufficiently small (all depending on $\bg_0,\Sigma,g,\bar \delta$), we can ensure that all members of $\Xi$ with  non-zero genus to be arbitrarily close as varifolds to some leaf of $\cF_1$ with multiplicity two (uniformly in $(a,b)$). Hence, we can guarantee that, for any $y$ such that $\fg(\Xi(y))=g$, we have $\area(\Xi(y))<2\area(\Sigma)+\bar\delta$, as desired. 

Hence, by the discussion at the beginning of  \S \ref{sect:mainProofMinimal}, Theorem \ref{thm:mainMinimal} follows.

\section{Proof of Theorem \ref{prop:Psi1}}\label{sect:ProofPsi1}

To prove the theorem, we  first need to understand the geometry  of the members of  $\Psi$.
\subsection{Description of members of $\Psi$.}\label{sect:descrption}
We will describe the surfaces $\Psi(a,b)$ for $(a,b)\in\RP^5\x B^{2g-2}$  in four parts:
\begin{enumerate}[label=(\alph*)]
    \item $a \in\overline{\RP^5\backslash (A_1\cup A_2)}$.
    \item The part $\Psi(a,b)\cap \overline{S^3\backslash N_1(a_1,a_2)}$ for $ a\in A_1\cup A_2$. 
    \item The part $\Psi(a,b)\cap N_1(a_1,a_2)$ for $ a\in A_1$.
    \item The part $\Psi(a,b)\cap N_1(a_1,a_2)$ for $ a\in A_2$.
\end{enumerate} 
Part (c) is the most interesting, since it would be the only part for we see  non-zero genus.

\subsubsection{Part (a)}
In \cite[Lemma 7.1 (iii)] {ChuLiWang2025GenusTwoI}, we have shown that $\Phi_5(a)$ must map into $\cS_0$ whenever $a\notin A_{\sing}$. Hence,  $\Psi(a,b)\in\cS_0$ for any $(a,b)\in \overline{\RP^5\backslash (A_1\cup A_2)}\x B^{2g-2}$, as $\Psi(a,b)=\Phi_5(a)$.

\subsubsection{Part (b)}
 For both part (b) and (d), the surfaces concerned are defined as the zero sets given by 
\begin{align}\label{eq:partB}
    0=\;&(x_1+a_2)(x_2+a_1)+(a_0-a_1a_2)+(1-\psi(a,x))\sqrt{1-x_1^2-x_2^2}(a_3\cos\alpha+a_4\sin\alpha)\\
   \nonumber  &+\psi(a,x)\sqrt{1-a_1^2-a_2^2}(a_3\cos\alpha+a_4\sin\alpha).
\end{align}
Having first chosen the function $\eta$, we can choose $\epsilon_1$ sufficiently small, which forces $a_0-a_1a_2,a_3,a_4$ to be small (depending on $\eta$), such that within the annular neighborhood $N_2$  of the  circle of intersection  $C(a_1,a_2)$, the zero set above is homeomorphic to four copies of $S^1\x [0,1]$ (see Figure \ref{fig:fourPlane}). The reason is that, the original surface $\Phi_5(a)$ is smooth within $N_2$, so within $N_2$ all zero sets given  by (\ref{eq:partB}), being smooth perturbations of $\Phi_5(a)\cap N_2$, must be geometrically the same as $\Phi_5(a)\cap N_2.$ As a result, we can assume that $\Psi(a,b)\cap \overline{S^3\backslash N_1(a,b)}$ is homeomorphic to the union of two intersecting spheres with a  neighborhood  of the intersection circle  removed.

   \begin{figure}[h]
        \centering
        \makebox[\textwidth][c]{\includegraphics[width=3in]{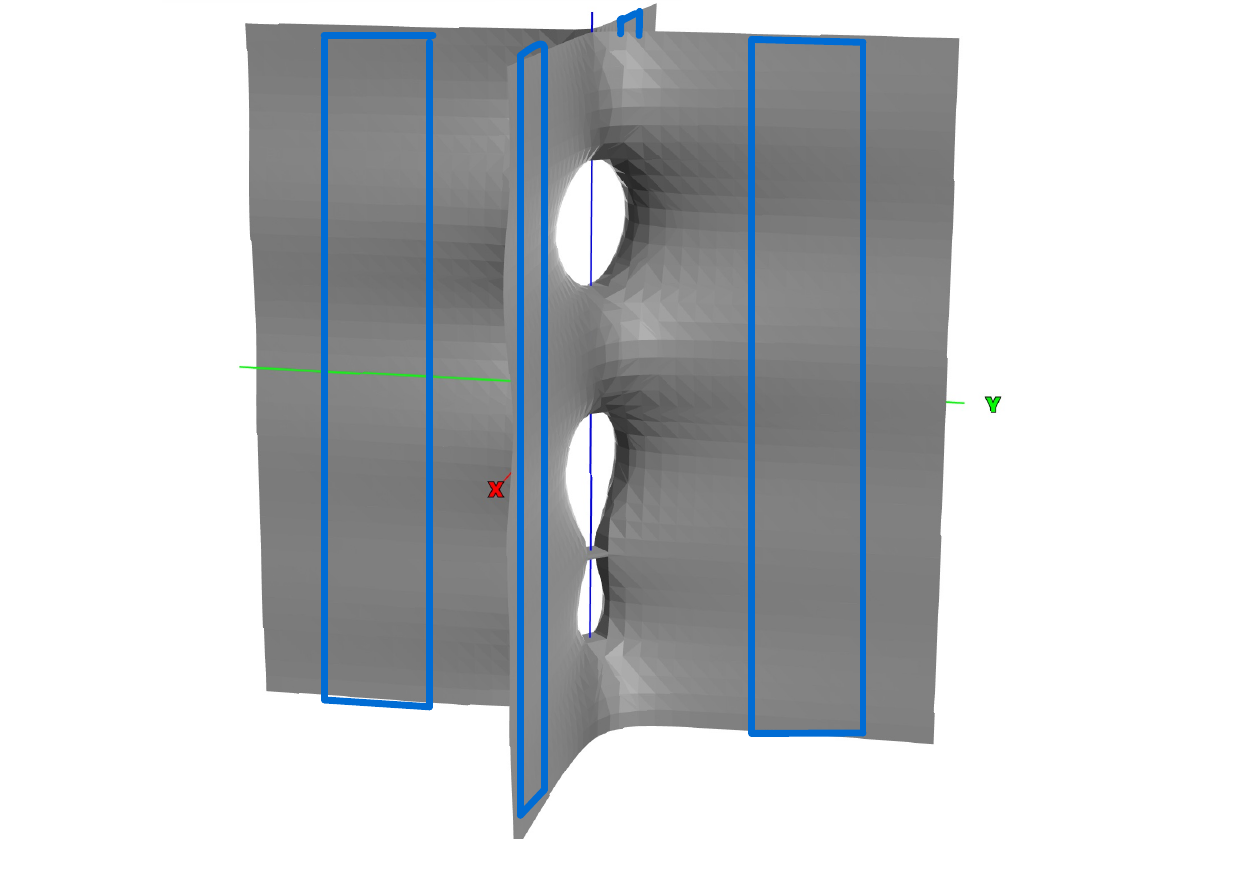}}
        \caption{This shows an example of $\Psi(a,b)$ for $a\in A_1$. The vertical axis denotes the circle $C(a_1    ,a_2)$. The four blue planes denote $\Psi(a,b)\cap N_2$, which is homeomorphic to  four copies of $S^1\x [0,1]$.}
        \label{fig:fourPlane}
    \end{figure}

\subsubsection{Part (c)}   For $a\in A_1$, $\Psi(a,b)\cap N_1$ is given by the zero set of 
\begin{align*}       &\;\;(x_1+a_2)(x_2+a_1)+(a_0-a_1a_2)+ \sqrt{1-a_1^2-a_2^2}(a_3\cos\alpha+a_4\sin\alpha)\\
        +&\;\epsilon_2(a_1^2+a_2^2)\cdot (\epsilon_1(a_1^2+a_2^2)-\|(a_0-a_1a_2,a_3,a_4)\|)\;\cdot\\
&\left[b_2\cos2\alpha+b'_2\sin2\alpha+...+b_g\cos g\alpha+b'_g\sin g\alpha+(1-\norm{b})\cos(g+1)\alpha\right]
\end{align*}
For simplicity we denote this expression as  $(x_1+a_2)(x_2+a_1)+F_{(a,b)}(\alpha)=0$, where $F_{(a,b)}(\alpha)$ is a trigonometric polynomial of degree $\leq g+1$. Importantly, $F_{(a,b)}(\alpha)$ has at most $2g+2$ roots in $[0,2\pi)$. Using this fact, we can describe the geometry of this zero set within $N_1(a_1,a_2)$. Namely, let us intersect it with the ``plane" $\alpha=\alpha_0$ in $N_1$, with $\alpha_0$ varying from $0 $ to $2\pi$, and examine the intersection curve. Here it would be convenient to view $N_1$ not as a subset of $S^3$, but as the original form given by (\ref{eq:N1}). Also, as before, we will assume  $\epsilon_1$ to be sufficiently small, depending on the already chosen  function $\eta$.

Whenever $F_{(a,b)}(\alpha_0)$ changes sign, say from positive  to negative, the intersection curve $\Psi(a,b)\cap N_1\cap \{\alpha=\alpha_0\}$ will undergo a topological change: It changes from the hyperpola $x_1x_2=1$, to the cross $x_1x_2=0$, and then the hyperpola $x_1x_2=-1$. But since $F_{(a,b)}$ has at most $2g+2$ roots, such topological change of the intersection curve can only happen at most $2g+2$ times.  Geometrically, if we choose the function $\epsilon_2>0$ to be small enough, depending on $\epsilon_1,\eta$, the surfaces $\Psi(a,b)\cap N_1$ ```opens up" the set 
$$(x_1+a_2)(x_2+a_1)=0$$
along the singular curve $C(a_1,a_2)$ by inserting at most $g+1$ handles, with $\epsilon_2$ controlling how much to open up. Of course, the sizes and distances of the handles can vary, and they might even pinch or merge with each other. From this, we can  deduce that $\Psi(a,b)\cap N_1$ has at most genus $g$.

More precisely, let $N_\odd(F_{(a,b)})$ denote the  number of roots of $F_{(a,b)}(\alpha)$ (in $\R/2\pi\Z$) of  {\it odd multiplicity}. Note, if $\alpha_1$ is  a root of odd (resp. even)  multiplicity, then as $\alpha_0$ increases and passes through $\alpha_1$,  $F_{(a,b)}(\alpha_0)$ would (resp. would not)  change sign, and thus the intersection curve $\Psi(a,b)\cap N_1\cap \{\alpha=\alpha_0\}$ would (resp. would not) change its topological type (from type $x_1x_2=+1$ to $x_1x_2=-1$ or vice versa).   Based on this observation, it follows immediately that we have the relation
$$\fg(\Psi(a,b)\cap N_1)=\frac 12 N_\odd(F_{(a,b)})-1.$$

\begin{rmk}\label{rmk:descriptionSurf}
    Let us summarize part (b) and (c), and describe the shape of surface $\Psi(a,b)$ for any  $(a,b)\in A_1\x B^{2g-2}$.  Outside the solid torus $N_1(a,b)$, the piece $\Psi(a,b)\cap \overline{S^3\backslash N_1(a_1,a_2)}$ is homeomorphic to a disjoint  union of four discs.
    
    As for inside the solid torus $N_1(a_1,a_2)$, let us cut it along $\alpha=0$ and pretend for a moment that it is the solid cylinder $$\{(x_1,x_2)\in\R^2:\|(x_1,x_2)+(a_2,a_1)\|\leq\eta(a_1,a_2) \}\x [0,2\pi).$$Then inside this solid cylinder, $\Psi(a,b)$ looks like two sheets (i.e. topologically $\R^2$) with $\leq g+1$ handles bridging them: See Figure (\ref{fig:fourPlane}). These handles are all near the axis $\{(-a_2,-a_1)\}\x [0,2\pi)$ (corresponding to the circle $C(a_1,a_2)$), and they intersect the axis at  $\leq 2g+2$  points, all of the form  $(-a_2,-a_1,\alpha)$, where $\alpha\in [0,2\pi)$ is any of the real roots of the trigonometric polynomial $F_{(a,b)}(\alpha)$. The handles can pinch or merge with each other, creating isolated singularities in $\Psi(a,b)$. In fact, these singularities take the form $(-a_2,-a_1,\alpha)$, where $\alpha$ is any root of multiplicity $\geq 2$ for  $F_{(a,b)}$.

Now, let $N_\odd(F_{(a,b)})$ denote the  number of roots of $F_{(a,b)}(\alpha)$ (in $\R/2\pi\Z$) that have  odd multiplicity.  Then $$\fg(\Psi(a,b))=\frac 12 N_\odd(F_{(a,b)})-1.$$
\end{rmk} 

\subsubsection{Part (d)}\label{sect:PartD}
  As for the part $\Psi(a,b)\cap N_1(a_1,a_2)$ for $a\in A_2$, we first examine the special case of $a\in A_1\cap A_2$. In this case, the zero set defining $\Psi(a,b)$  is given by
\begin{align*}
    (x_1+a_2)(x_2+a_1)+(a_0-a_1a_2)+  \sqrt{1-a_1^2-a_2^2}(a_3\cos\alpha+a_4\sin\alpha)=0.
\end{align*}
Hence, if $\epsilon_1$ is small enough (depending on $\eta$), using the discussion for part (c) above we immediately see that $\Psi(a,b)\cap N_1(a_1,a_2)$ is homeomorphic of one of three types: (1) A cylinder, (2) two discs, or (3) two discs touching at an interior point: See  Figure \ref{fig:3types}.

   \begin{figure}[h]
        \centering
        \makebox[\textwidth][c]{\includegraphics[width=\textwidth]{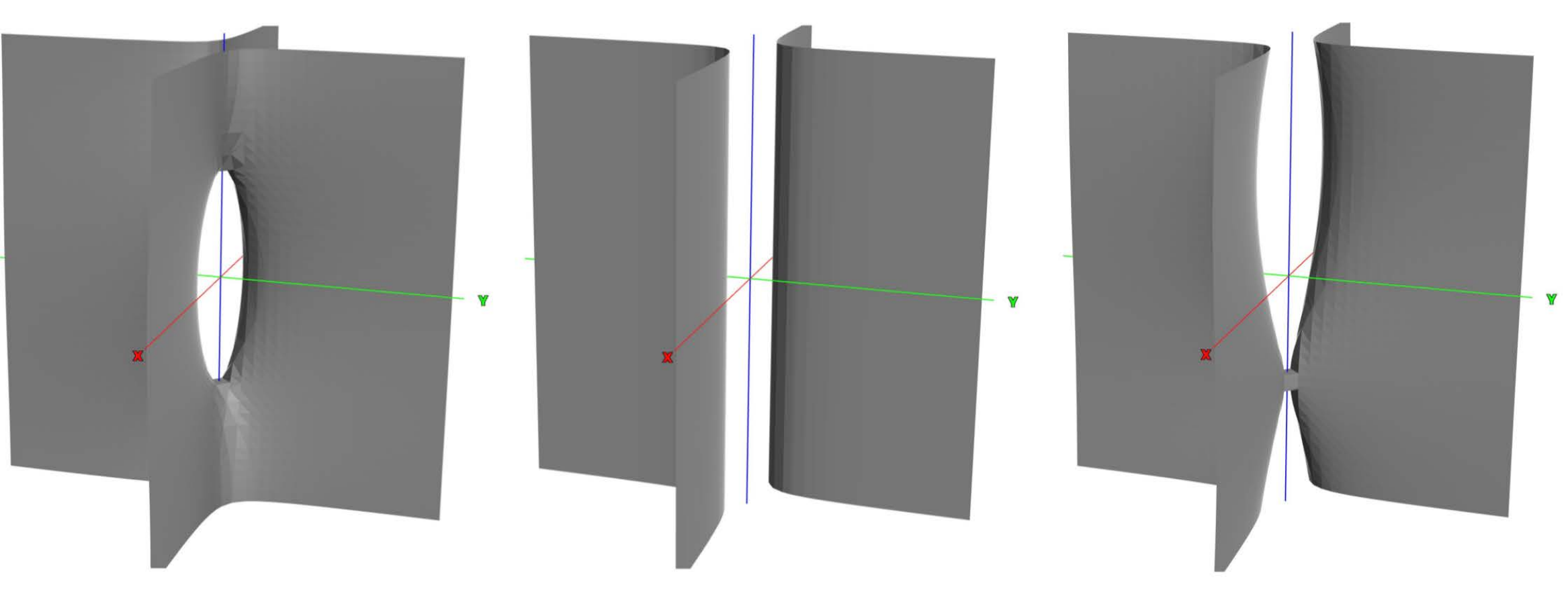}}
        \caption{From left to right, we type type (1), (2), and (3).}
        \label{fig:3types}
    \end{figure}

For the general case  $a\in A_2$, the surface concerned is given by the zero set of 
\begin{equation}\label{eq:tildeFa}
    (x_1+a_2)(x_2+a_1)+(a_0-a_1a_2)+\varsigma(a,x)(a_3\cos\alpha+a_4\sin\alpha),
\end{equation}
where $$\varsigma(a,x):=(1-\psi(a,x))\sqrt{1-x_1^2-x_2^2}+\psi(a,x)\sqrt{1-a^2_1-a^2_2},$$
which is close to $\sqrt{1-a^2_1-a^2_2}$ as $x\in N_1(a_1,a_2)$.
Let $\tilde F_a(x_1,x_2,\alpha)$ denote the expression (\ref{eq:tildeFa}).  Now, we can reuse the  proof of \cite[Lemma 7.7 (a), (b)]{ChuLiWang2025GenusTwoI} (applied to the function $\tilde F_a$ instead of $\tilde F_a$ therein) to show that: {\it Assume the functions $\epsilon_1,\eta>0$ are chosen small enough. Then for any $a\in A_2$, there are at most $2$ choices of  $(x_1,x_2,\alpha)\in N_1$   such that $\nabla\tilde F_a(x_1,x_2,\alpha)=0$, in which   $\nabla$ is only taken with respect to $(x_1,x_2)$, and at the same time $\tilde F_a(x_1,x_2,\alpha)=0$.}
Since the calculation is essentially the same as in \cite[Lemma 7.7 (a), (b)]{ChuLiWang2025GenusTwoI}, we move the proof to Appendix  \ref{appendix:countRoots}.

As a result, when we slice the surface with the horizontal planes $\{\alpha=\alpha_0\}$ for $\alpha_0\in [0,2\pi)$, we can only see topological change for the intersection curve at those $\leq 2$ choices of  $(x_1,x_2,\alpha)$.  Then  it follows that  $\Psi(a,b)\cap N_1(a_1,a_2)$  can only take one of the three forms described at the end of last  paragraph.

\subsection{Simon-Smith family}
From the above description, it is easy to check that whenever $a\in A_1\cup A_2$,  $\Psi$ maps into $\cS_{\leq g}$.  As for part (a), $a\in \overline{\RP^5\backslash (A_1\cup A_2)}$, we already pointed out above that $\Psi(a,b)\in \cS_0$. Thus, $\Psi$ indeed maps into $\cS_{\leq g}$. 

We now explain why  $\Psi$ is a Simon-Smith family. On the part $$\inte(\RP^5\backslash (A_1\cup A_2))\x B^{2g-2},$$ we have $\Psi(a,b):=\Phi_5(a)$, so this claim is already checked in   \cite[\S 7] {ChuLiWang2025GenusTwoI}: It fact, it is a consequence of Theorem 3.4 therein. As for the part $\inte(A_1\cup A_2)\x B^{2g-2}$, and the part that is near $\partial(A_1\cup A_2)\x B^{2g-2}$ but excluding $\partial A_{\sing}\x B^{2g-2}$, the claim is easy to using the analysis in \S \ref{sect:descrption}. Finally, as for the part near $\partial A_{\sing}\x B^{2g-2}$, we  fix any $(a,b)\in \partial A_{\sing}\x B^{2g-2}$. Then $a_1^2+a_2^2=1$, so that $a_3=a_4=0$ and  $\Psi(a,b)$ is given by $$(x_1+a_2)(x_2+a_1)=0.$$
If   $a_1,a_2$ are both non-zero, this is a union of two round spheres touching each other at the point $(-a_2,-a_1,0,0)\in S^3$. If one of $a_1,a_2$ is zero, then this zero set is just one equatorial sphere. 

Now, we observe that, for $(\hat a,\hat b)\in (A_1\cup A_2)\x B^{2g-2}$, as $(\hat a,\hat b)\to (a,b)$ in $\RP^5\x B^{2g-2}$, the  circle $C(\hat a_1,\hat a_2)$  shrinks, and tends to the point $(-a_2,-a_1,0,0)$. So the handles of $\Psi(\hat a,\hat b)$ described in  \S \ref{sect:descrption} must also shrink accordingly. Based on this, it can be checked that as $(\hat a,\hat b)\to (a,b)$ in $Y$, $\Psi(\hat a,\hat b)\to \Psi(a,b)$ smoothly away from the point $(-a_2,-a_1,0,0)$. So $\Psi$ is indeed also a Simon-Smith family near $\partial A_{\sing}\x B^{2g-2}$.

\subsection{Proof of items of Theorem \ref{prop:Psi1}}\label{sect:proofItemTheoremPsi}
For item (\ref{Item:Family_Genus0}), the claim that $\Psi|_{\overline{\RP^5\backslash A_1}\x B^{2g-2}}$ is of genus $0$ follows from the  discussion in \S \ref{sect:descrption} for surfaces in part (a), (b), and (d). 

For item (\ref{Item:Family_5-sweepout}), since the parameter space of $\Psi|_{\RP^5\x \{0\}}$ is an $\RP^5$, it suffices to show that it has a 1-sweepout. We consider the curve given by  $a_2=a_3=a_4=a_5=0$. This gives the zero sets $\{a_0+a_1x_1=0\}$, which is a 1-sweepout.

For item (\ref{Item:Family_Genus1}), in order to show that $\Psi|_{\RP^5\x \partial B^{2g-2}}$ is of genus $\leq g-1$,  it suffices to show that $\Psi(a,b)$ has  genus $\leq g-1$ for any $a\in A_1$, because of item (\ref{Item:Family_Genus0}). This claim  follows easily from the  description in  Remark \ref{rmk:descriptionSurf}, and the fact that the trigonometric polynomial $F_{(a,b)}$ has degree $\leq g$ for $b\in\partial B^{2g-2}$.

For item (\ref{item:topoSame}), we recall that for any $(a,b)\in A_1\x B^{2g-2}$, the surface $\Psi(a,b)$  is the zero set of $(x_1+a_2)(x_2+a_1)+F_{(a,b)}(\alpha)=0$ where the trigonometric polynomial $F_{(a,b)}(\alpha)$ takes the form
\begin{align*}       &\; (a_0-a_1a_2)+ \sqrt{1-a_1^2-a_2^2}(a_3\cos\alpha+a_4\sin\alpha)\\
        +&\;\epsilon_2(a_1^2+a_2^2)\cdot (\epsilon_1(a_1^2+a_2^2)-\|(a_0-a_1a_2,a_3,a_4)\|)\;\cdot\\
&\left[b_2\cos2\alpha+b'_2\sin2\alpha+...+b_g\cos g\alpha+b'_g\sin g\alpha+(1-\norm{b})\cos(g+1)\alpha\right]
\end{align*}
Now, the geometry of the surfaces $\Psi(a,b)$ was discussed in  Remark \ref{rmk:descriptionSurf}. Essentially, the locations and sizes of the handles of the surfaces are determined by the roots of the  trigonometric polynomial $F_{(a,b)}$. 
\begin{lem} There exists a continuous map  $$P:\inte(A_1)\x B^{2g-2}\to \inte(A^{(0,0)})\x B^{2g-2}$$ such that:
\begin{itemize}
\item  For any $c\in\inte(\bD)$, the restriction
    $$P|_{\inte(A^{c}_1)\x B^{2g-2}}$$ gives a homeomorphism from $\inte(A^{c}_1)\x B^{2g-2}$ onto $\inte(A^{(0,0)}_1)\x B^{2g-2}$. 
    \item For each $(a,b)$,   the trigonometric polynomials $F_{(a,b)}$ and $F_{P(a,b)}$ are the same up to some positive  factor, so that they have the same roots.
\end{itemize}
\end{lem}
Once we have this lemma, the construction of the desired family $\varphi$ of diffeomorphisms and the second bullet point of Theorem \ref{prop:Psi1} (\ref{item:topoSame}) follow in an elementary way, based on  Remark \ref{rmk:descriptionSurf}. So below we will only give the proof of this lemma.

\begin{proof}
For each fixed $c\in\inte(\bD)$, consider the map  $T^{c}:A^{c}_1\x B^{2g-2}\to \R^{2g+1}$ that sends $(a,b)$ to the point
\begin{align*}
    (&\textrm{the constant term of } F_{(a,b)},\\
    &\textrm{the coefficient of } \cos\alpha \textrm{ in } F_{(a,b)},\\
    &\textrm{the coefficient of } \sin\alpha \textrm{ in } F_{(a,b)},\\
    &...\\
    &\textrm{the coefficient of } \cos g\alpha \textrm{ in } F_{(a,b)},\\
    &\textrm{the coefficient of } \sin g\alpha \textrm{ in } F_{(a,b)}).
\end{align*}
From the definition of $F_{(a,b)}$, it is easy to see that this map sends $\inte(A^{c}_1)\x B^{2g-2}$ homeomorphically onto its image
and that $T^c(a,b)=T^c(a,b')$ whenever $a\in \partial A^c_1$.
Then consider the map $\tilde T^{c}:A^{c}_1\x B^{2g-2}\to\R^{2g+2}$ that sends each $(a,b)$ to $$(T^{c}(a,b),\textrm{the coefficient of } \cos (g+1)\alpha \textrm{ in } F_{(a,b)}).$$
Let $\hat \proj:\R^{2g+2}\to S^{2g+1}$ be the projection map onto the unit sphere, and let $S^{2g+1}_+$ be the closed upper hemisphere, where the $(2g+2)$-th coordinate $x_{2g+2}$ is non-negative.

Let $Q\subset \R^{2g+2}$ be the subspace given by $x_4=x_5=...=x_{2g+2}=0$.
For each $c$, we claim that the map $\hat \proj\circ \tilde T^{c}$ sends $\inte(A^{c}_1)\x B^{2g-2}$ homeomorphically onto 
$S^{2g+1}_+\backslash Q$. Once we have this claim, we can just define $P$ by $$P(a,b):=(\hat \proj\circ \tilde T^{(0,0)})^{-1}\circ (\hat \proj\circ \tilde T^{(a_1,a_2)})(a,b),$$and then the lemma follows immediately.

So let us now prove the claim. To prove that $\hat \proj\circ \tilde T^{c}$ is injective, we first observe the following: Fixing any $(a,b)\in \inte(A^{c}_1)\x B^{2g-2}$, as $t$ increases from $0$ to $1$, the $(2g+2)$-th coordinate of $$\tilde T^{c}([t(a_0-a_1a_2):a_1:a_2:ta_3:ta_4:1],tb),$$ which is
$$(\epsilon_1(a_1^2+a_2^2)-t\|(a_0-a_1a_2,a_3,a_4)\|)(1-t\|b\|),$$
is strictly decreasing in $t$. Then, together with the previous fact that $T^c$ sends $\inte(A^{c}_1)\x B^{2g-2}$ homeomorphically onto its image, the injectivity of $\hat \proj\circ \tilde T^{c}$ follows.

As for the surjectivity of $(\hat \proj\circ \tilde T^{c})|_{\inte(A^c_1)\x B^{2g-2}}$ onto $S^{2g-2}_+\backslash Q$, we first note the following facts:
\begin{itemize}
    \item $\tilde T^{c}$ maps into the upper half-space $\{x_{2g+2}\geq 0\}\subset \R^{2g+2}$.
    \item $\tilde T^{c}$ maps $\partial(A^c_1\x B^{2g-2})$ into $\{x_{2g+2}=0\}\subset\R^{2g+2}$.
    \item $T^{c}$ maps $\inte(A^c_1\x B^{2g-2})$ homeomorphically onto its image, which contains the origin.
\end{itemize}
Based on these facts the desired  surjectivity  follows directly.  

Hence, the claim that $\hat \proj\circ \tilde T^{c}$  gives a homeomorphism $\inte(A^{c}_1)\x B^{2g-2}\to S^{2g+1}_+\backslash Q$ follows easily. This finishes the proof of the lemma.
\end{proof}

We return to proving the items in Theorem \ref{prop:Psi1}. For item (\ref{item:areaBound}),
the claim for $(a,b)\in \overline{\RP^5\backslash (A_1\cup A_2)}\x B^{2g-2}$ is clear, as  $\Psi(a,b):=\Phi_5(a)$ for each $(a,b)$. As for $a$ in $A_1\cup A_2$, the claim can also be easily verified: Intuitively, the surface $\Psi(a,b)$ only opens up $\Phi_5(a)$ along the intersection circle $C(a_1,a_2)$ by a small amount, if $\epsilon_1,\epsilon_2$ are small. 

This finishes the proof of Theorem \ref{prop:Psi1}.

\section{Proof of Lemma \ref{lem:capProd}}\label{sect:ProofCapProd}

Recall that we defined $Z_0\subset Y$ as the set of all $y\in Y$ such that $\Psi'(y)=0$. Note $Z_0$ is compact. So by Proposition \ref{prop:pinchOffg0g1}, there exists a subcomplex $Z\subset Y$ (after refinement), whose interior contains $Z_0$, such that $\Psi'|_Z$ can be deformed via pinch-off processes to become some map into $\cS_0$. As a result, using the fact that there does not exist a Simon-Smith family of genus 0 that is a 5-sweepout (Proposition \ref{prop:no5sweepoutGenus0}), we know that $\lambda^5|_{Z}=0$, and in particular $\lambda^5|_{\inte(Z)}=0$.

We now state a purely algebraic topological fact, which is a special case of Lemma 3.11 of \cite{ChuLiWang2025GenusTwoI} (by putting $A:=\emptyset$ and $C:=X$ therein).
\begin{lem}
    In this lemma,  we fix  an arbitrary commutative ring for the coefficients of homology and cohomology. 
    Let $X$ be a (not necessarily finite) simplicial complex, and $B\subset X$ be an open subset. Let $p\geq q\geq 0$. 
    Let $C$ be a finite union of $p$-subcomplexes of $X$ with $\spt(\partial C)\subset B$ (so that 
     $[C]\in H_{p}(X,B)$), and   $\omega\in H^q(X)$. 
    Suppose that: 
    \begin{itemize}
        \item $C$ is a union of two open subsets $W_1,W_2\subset C$.
        \item The pullback of $\omega$ under the inclusion $W_1\hookrightarrow X$ is zero.
    \end{itemize}
    Then there exists some  $\theta\in H_{p-q}(W_2,W_2\cap B)$  such that the pushforward of $\theta$ under the inclusion $(W_2,W_2\cap B)\hookrightarrow (X,B)$ is equal to
    $$[C]\frown\omega\in H_{p-q}(X,B).$$ 
\end{lem}
We apply this lemma to our situation with the following.
\begin{itemize}
    \item $X:=Y$,
    \item $B:=\partial Y$,
    \item $C:= X$,
    \item   $W_1:=\inte(Z),W_2:=Z_{\geq 1}$ (both are open sets), 
    \item $p:=2g+3, q:=5$,
    \item $\omega:=\lambda^5$.
\end{itemize}
Then there exists some $\theta\in H_{2g-2}(Z_{\geq 1},Z_{\geq 1}\cap\partial Y)$ such that the pushforward of $\theta$ under the inclusion $(Z_{\geq 1},Z_{\geq 1}\cap\partial Y)\hookrightarrow (Y,\partial Y)$ is equal to $$[Y]\frown\lambda^5\in H_{2g-2}(Y,\partial Y).$$ Letting $D$ be a $(2g-2)$-subcomplex of $Y$ such that $\theta=[D]$, we finish the proof.

\section{Proof of Proposition \ref{prop:DeformRetractF2}} \label{sect:ProofDeformRetractF2}

In this section, we prove Proposition \ref{prop:DeformRetractF2}.
\subsection{About $\partial X_g$} 
We recall that by Remark \ref{rmk:descriptionSurf} for any $a\in A_1$, $\Psi(a,b)$ has genus $g$ if and only if the trigonometric polynomial $F_{(a,b)}(\alpha)$ has $2g+2$ distinct real roots.  Let us try to describe the set $\partial X_g$ using this fact. Note,  $X_g$ is clearly an open $(2g+1)$-dimensional region in $X$.

For any $n\geq 1$,
consider the set $\Omega\subset\R^{2n-1}$ of all $s=(s_0,s_1,s'_1,...,s_{n-1},s'_{n-1})\in \R^{2n-1}$ such that all $2n$ roots of the trigonometric polynomial
$$f_s(\alpha):=s_0+s_1\cos\alpha+s'_1\sin\alpha+...+s_{n-1}\cos (n-1)\alpha+s'_{n-1}\sin (n-1)\alpha+{\cos n\alpha}$$
are real: Here the roots need not be distinct. Recalling the definition of the trigonometric polynomial  $F_{(a,b)}(\alpha)$, to prove that $X_g$ has piecewise smooth algebraic boundary, it suffices to show that $\Omega\subset \R^{2n-1}$ has piecewise smooth, algebraic boundary.

We introduce new variables $u=\cos\alpha$ and $v=\sin\alpha$. It is elementary to see that we can rewrite  $f_s(\alpha)$  as some polynomial  $h(u,v)$ in $u,v$ with coefficients being polynomial expressions in $s_i,s'_i$. Then the condition  $s\in\inte(\Omega)$   translates to the system
$$h(u,v)=0, \quad u^2+v^2=1$$
having exactly $2n$ distinct real points. Now, this condition can be phrased as an algebraic condition on $s_i,s_i'$. So  $\partial \Omega$ must be a piecewise smooth, algebraic set, as desired.

As for the claim that $\overline{X_g}$ is topologically a closed $(2g+1)$-ball, and the construction of the deformation retraction $F_2$, we first need some elementary observations regarding trigonometric polynomials. 

\subsection{General facts about trigonometric polynomials} 
A function $f:\C\to\C$ is called a {\it trigonometric polynomial of degree $n$} if it has   the form
\begin{equation}\label{eq:fAlpha}
f(\alpha)=s_0+s_1\cos\alpha+s'_1\sin\alpha+...+s_n\cos n\alpha+s'_n\sin n\alpha,
\end{equation}
where each $s_i,s'_i\in \C$,  and either $s_n$ or $s'_n$ is non-zero.   
Note that all trigonometry polynomials $f$ are $2\pi$-periodic, and thus can also be viewed as   functions  from $(\R/2\pi \Z)\x i\R$ into $\C$.

We will need the following  elementary facts.

\begin{lem}\label{lem:sumOfRoots}
 Let $f$  be a trigonometric polynomial with complex coefficients. Then the followings are equivalent:
 \begin{enumerate}
     \item All roots of $f$ come in conjugate pairs (i.e. if $w\in\C$ is a root then so is $\bar w$).
     \item  $f$ is, up to a constant factor, equal to some trigonometric polynomial with real coefficients.  
 \end{enumerate}
\end{lem}
\begin{proof}
Let $f$ be given by (\ref{eq:fAlpha}).
    Rewrite  $\cos k\alpha$ and $\sin k\alpha$ into powers of $e^{ik\alpha}$, so that $f(\alpha)=s_0+\sum^n_{k=1}c_k e^{ik\alpha}+c_{-k}e^{-ik\alpha}$, where $c_{{\pm}k}:=\frac 12(s_k\pm is'_k)$. ``(2) implies (1)" can be checked directly. For the converse, we consider the function
    $\tilde f(\alpha):=\overline{f(\bar\alpha)}$. If (1) holds, then $\tilde f$ has the same roots as $f$, so they are the same up to a multiplicative factor. Hence, one can show for some $\theta$,  $\overline{c_k}=e^{i\theta}c_{-k}$ for every $k$. It follows easily that    $e^{i\theta/2}s_k$ is real.  
\end{proof}

\begin{lem}
Let $f$ be a degree $n$ trigonometric polynomial with complex coefficients. Then the followings are equivalent: 
\begin{enumerate}
    \item The sum of all  $2n$ complex roots (not necessarily distinct) of $f$ in $[0,2\pi)\x i\R\subset\C$ is  in $2\pi\Z$. 
    \item Up to some constant factor, $f$ has the form
\begin{equation}\label{eq:trigPoly} s_0+s_1\cos\alpha+s'_1\sin\alpha+...+s_{n-1}\cos (n-1)\alpha+s'_{n-1}\sin (n-1)\alpha+\textcolor{red}{\cos n\alpha}.
    \end{equation}
\end{enumerate}
\end{lem}
\begin{proof}
    As in the  proof of the previous lemma, we write  $f(\alpha)=s_0+\sum^n_{k=1}c_k e^{ik\alpha}+c_{-k}e^{-ik\alpha}$. Let $z=e^{i\alpha}$, and $g(z)=s_0+\sum^n_{k=1}c_k z^k+c_{-k}z^{-k}$. Then (1) is equivalent to the statement that the {\it product} of all $2n$ roots of $g(z)$ is $1$, which means $c_{-k}/c_k=1$. Now, the equivalence of (1) and (2) can be checked easily.
\end{proof}

We recall that $\Sym^n(\cdot)$ denotes the $n$-th symmetric product of a space. Define a map $\bz:\R^{2n-1}\to \Sym^{2n}((\R/2\pi \Z)\x i\R)$ as follows. Recall that given any ${s}=(s_0,s_1,...,s_{n-1},s'_1,...,s'_{n-1})\in \R^{2n+1}$, we denote by  $f_{s}$ the trigonometric polynomial given by the expression  (\ref{eq:trigPoly}).  Then collecting its $2n$ roots, we  obtain an element $\bz({s})$ in $\Sym^{2n}((\R/2\pi \Z)\x i\R)$. Now, we  consider the subset $\cW\subset \Sym^{2n}((\R/2\pi \Z)\x i\R)$ that consists of all configurations ${\bf w}\in \Sym^{2n}((\R/2\pi \Z)\x i\R)$ such that:
\begin{itemize}
    \item The sum of the $2n$ members of ${\bf w}$ is in $2\pi\Z$.
    \item If $w$ is a member of ${\bf w}$, then so is $\bar w$.
\end{itemize}
Then it follows from the two lemmas mentioned above that $\bz$ actually maps into $\cW$, and $\bz$ gives a homeomorphism between $\R^{2n-1}$ and $\cW$.

Recall that $S^1:=\R/2\pi\Z$, which can naturally be identified with the  subset  $(\R/2\pi\Z)\x \{0\}\subset (\R/2\pi\Z)\x i\R$.
Let $\Sym^{2n}_0(S^1)\subset\Sym^{2n}(S^1)$ be the subset of all $\balpha\in \Sym^{2n}(S^1)$ such that the sum of all  $2n$ members of $\balpha$ is in $2\pi\Z$. Note $\Sym^{2n}_0(S^1)\subset \cW$.

Recall that we let $\Omega\subset\R^{2n+1}$ be the set of all points $s\in \R^{2n+1}$ such that all $2n$ roots of the trigonometry polynomial $f_{s}$ defined by (\ref{eq:trigPoly}) are {\it real}: Here the roots need not be distinct. Hence, $\bz$  maps $\Omega$ onto $\Sym^{2n}_0(S^1)$  homeomorphically. Note it is well-known that $\Sym^{2n}_0(S^1)$ is homeomorphic to a closed $(2n-1)$-simplex (see \S \ref{sect:symProd}).

Now, we are going to define a deformation retraction of $\R^{2n-1}$ onto $\Omega$, $$F_\Omega:[0,1]\x \R^{2n-1}\to \R^{2n-1}.$$
 First, we can deformation retract $\cW$ onto $\Sym^{2n}_0(S^1)$ as follows. Take any ${\bf w}\in\cW$, we pushes all its members towards the real line while changing only their imaginary parts: For each member $w_j$ of $\bf w$, as time $t$ varies from $0$ to $1$, we push the point $w_j$   to the point
 $$w_j(t):=\Re(w_j)+i\;\textrm{sign}(\textrm{Im} w_j) \min\{|\textrm{Im} w_j|,\tan^{-1}((1-t)\pi/2)\}.$$
 This deformation retraction of $\cW$ onto $\Sym^{2n}_0(S^1)$ would induce, under the map $\bf z$, a deformation retraction $F_\Omega$ of $\R^{2n-1}$ onto $\Omega$: We define $F_\Omega$ to be this deformation retraction. Note that the subset $\Omega$ is  fixed  throughout the retraction.

\subsection{A deformation retraction on $\mathrm{int}(X)$}
We now apply the above to the setting we are interested in. Namely, we  take $n=g+1$, as the defining equations for the family $\Psi$ have degree $\leq g+1$. In particular, the subset $\Omega$ is in $\R^{2g+1}$, and  we have a deformation retraction of $\R^{2g+1}$ onto $\Omega$, $$F_\Omega:[0,1]\x \R^{2g+1}\to \R^{2g+1},$$
defined as above.
 
 Let us now define the desired strong deformation retraction   of $\mathrm{int}(X)$ (recall $X:=A^0_1\x B^{2g-2}$) onto $\overline{X_g}$,
    $$F_2:[0,1]\x \mathrm{int}(X)\to \mathrm{int}(X).$$
Recall that the   surface $\Psi([a_0:0:0:a_3:a_4:1],(b_2,b_2',...,b_g,b'_g))$ is given by the zero set of an equation of the form $x_1x_2+F_{(a,b)}(\alpha)=0$, where $F_{(a,b)}(\alpha)$ is the the trigonometric  polynomial
\begin{align*}       &\; a_0+ a_3\cos\alpha+a_4\sin\alpha\\
        +&\;\epsilon_2(0)\cdot (\epsilon_1(0)-\|(a_0,a_3,a_4)\|)\;\cdot\\
&\left[b_2\cos2\alpha+b'_2\sin2\alpha+...+b_g\cos g\alpha+b'_g\sin g\alpha+(1-\norm{b})\cos(g+1)\alpha.\right]
\end{align*} It follows easily that, there exists a  unique, well-defined homeomorphism $T:\textrm{int}(X)\to\R^{2g+1}$
such that $s=T(a,b)$ if and only if   the trigonometric polynomial $F_{(a,b)}(\alpha) $ is equal to
$$f_{s}(\alpha):=s_0+s_1\cos\alpha+s'_1\sin g\alpha+...+s_g\cos\alpha+s'_g\sin g\alpha+\cos(g+1)\alpha$$
up to some positive real factor.
Indeed, recall in \S \ref{sect:proofItemTheoremPsi} we showed $\hat\proj\circ \tilde T^{(0,0)}$ maps $\inte(A^{0}_1)\x B^{2g-2}$ onto $ S^{2g+1}_+\backslash Q$ homeomorphically, which means it maps $\inte(X)\to \inte(S^{2g+1}_+)$ homeomorphically. Then the desired homeomorphism $T$ can be obtained immediately.

Thus, using $T$, we can transport the deformation retraction $F_\Omega$ in $\R^{2n+1}$ over to $\mathrm{int}(X)$ and obtain a map
$$F_2:[0,1]\x \mathrm{int}(X)\to \mathrm{int}(X).$$

We now explain why $F_2$ is a deformation retraction of $\inte(X)$ onto $\overline{X_g}$. We first recall a fact from  Remark \ref{rmk:descriptionSurf}: Let $N_\odd(F_{(a,b)})$ be the number of {\it real} roots of the trigonometric polynomial $F_{(a,b)}$ in $S^1$ that have odd multiplicity. Then   the genus of $\Psi(a,b)$  is given by $\frac 12 N_\odd(F_{(a,b)})-1$ for every $a\in A_1$.  Now, by construction, for every $(a,b)$ in the image of $F_2(1,\cdot)$, all $2g+2$ roots (not necessarily distinct) of the trigonometric polynomial $F_{(a,b)}$ are real. If we slightly perturb these roots to turn them into $2g+2$ distinct points in $S^1$, and let $(a',b')\in \inte(X)$ be such that $F_{(a',b')}$ has these $2g+2$ points as its roots, then $\Psi(a',b')$ has genus $g$ by Remark \ref{rmk:descriptionSurf}. As a result, we know that $ F_2(1,\cdot)$  indeed maps into the closure of $X_g$.

\begin{rmk}\label{rmk:chainHomeo}
    By the above analysis, we can obtain a chain of homeomorphisms
$$\overline{X_g}\xrightarrow[]{T}\Omega\xrightarrow[]{\bz} \Sym^{2g+2}_0(S^1).$$ 
In particular, $\overline{X_g}$ is a closed $(2g-1)$-ball
\end{rmk}

To prove the first bullet point in Proposition \ref{prop:DeformRetractF2}, which says the genus of $\Psi(F_2(t,(a,b)))$ stays unchanged with respect to $t$,  we first denote by $G:[0,1]\x \cW\to\cW$ the deformation retraction of $\cW$ onto $\Sym^{2g+2}_0(S^1)$ that we defined in the previous section. As a point $\bf w\in \cW $ get pushed under this deformation, the number of real points {\it with odd multiplicity} in $G(t,{\bf w})$ stays fixed, {\it even at the final time $t=1$ when some points get pushed onto the real line} (because points in $\bf w$ come in conjugate pairs by definition). Then using Remark \ref{rmk:descriptionSurf}, the first bullet point in Proposition \ref{prop:DeformRetractF2} follows immediately. 

To prove the second bullet point in Proposition \ref{prop:DeformRetractF2}, which says $F_2(t,\cdot)$ fixes $\overline{X_g}$, we just need to note that $G(t,\cdot)$ fixes the subset  $\Sym^{2g+2}_0(S^1)\subset \cW$ for all $t$.

Finally, the third bullet point, which says for every $t$ and $y\in \mathrm{int}(X)\backslash{X_g}$ we have $F_2(t,y)\in \mathrm{int}(X)\backslash{X_g}$,  follows directly from the first bullet point. This finishes the proof of Proposition \ref{prop:DeformRetractF2}.

\section{Proof of Proposition \ref{prop:embedHomology}\label{sect:ProofEmbedHomology}}
From Remark \ref{rmk:chainHomeo}, we have a chain of homeomorphisms
$$\overline{X_g}\xrightarrow[]{T}\Omega\xrightarrow[]{\bz} \Sym^{2g+2}_0(S^1).$$
In \S \ref{sect:symProd}, we explained that $\Sym^{2g+2}_0(S^1)$ is homeomorphic to a closed $(2g+1)$-simplex, and thereby inherits a  simplicial complex structure from it. In particular, we know that  each point in the interior region $\inte(\Sym^{2g+2}_0(S^1))$ corresponds to a configuration of $2g+2$ distinct points in $S^1$. 
Hence, based on the description on the geometry of $\Psi(y)$ for $y\in A^0_1\x B^{2g-2}$ in  Remark \ref{rmk:descriptionSurf}, we know the $g+1$  handles of $\Psi(y)$ within the solid torus $N_1(0,0)$ must have positive widths and  positive distances from each other. Thus, all such $\Psi(y)$ are  isotopic to each other: They are all isotopic to a standard Heegaard surface of genus $g$. Therefore, using the fact that $X_g$ is contractible, there are  {\it canonical} isomorphisms 
$$H_1(\ins(\Psi(y)))\cong H_1(\ins(\Psi(y'))) \quad\textrm{ and }\quad H_1(\out(\Psi(y)))\cong H_1(\out(\Psi(y')))$$
for any $y,y'\in  X_g$. 

Now, fix some $y_0\in X_g$. We can choose isomorphisms
$$H_1(\ins(\Psi(y_0)))\cong\Z^g_2,\quad  H_1(\out(\Psi(y_0)))\cong  \Z_2^g$$ of the following property:  The direct sum of these two  isomorphisms  sends the linking number bilinear form  $$L:H_1(\ins(\Psi(y_0)))\oplus H_1(\out(\Psi(y_0)))\to\Z_2$$ to the standard bilinear form $I_\id$. This is possible because one can find loops $\alpha_1,...,\alpha_g\subset \ins(\Psi(y_0))$ and $\beta_1,...,\beta_g\subset \out(\Psi(y_0))$ for which $\alpha_i$ and $\beta_j$ are linked (in the  mod 2 sense) if and only if $i=j$. Now, using the canonical isomorphisms we obtained from the last paragraph, we can define  isomorphisms 
$$H_1(\ins(\Psi(y)))\cong\Z^g_2,\quad  H_1(\out(\Psi(y)))\cong  \Z_2^g$$
for all $y\in X_g$, such that the direct sum of these two isomorphisms also sends the linking number bilinear form $L$ (for $y$) to $I_\id$.

In other words, we have  obtained Proposition \ref{prop:embedHomology} ``for the subfamily $ \Psi|_{X_g}$" (instead of $ \Psi|_{\overline{X_g}}$). More precisely, There exist  continuous maps $$\mathfrak i_{\ins}:\fH _{\ins}(\Psi|_{X_g})\to \Z_2^g.\quad \mathfrak i_{\out}:\fH _{\out}(\Psi|_{X_g})\to \Z_2^g$$ 
such that for each $y\in  X_g$ :
\begin{itemize}
    \item  $\frak i_{\ins}$ gives an isomorphism of $H_1(\ins(\Psi(y)))$ onto its image.
    \item $\frak i_{\out}$ gives an isomorphism of $H_1(\out(\Psi(y)))$ onto its image.
    \item Let $\frak i :=(\frak i_{\ins},\frak i_{\out})$. Then $\frak i$ sends the linking number bilinear form
$$L:H_1(\ins(\Psi(y)))\oplus H_1(\out(\Psi(y)))\to\Z_2$$
  to the standard bilinear form $I_\id $ restricted onto the image  $$\frak i_{\ins}(H_1(\ins(\Psi(y))))\oplus \frak i_{\out}(H_1(\out(\Psi(y)))).$$
\end{itemize}
To finish the proof of Proposition \ref{prop:embedHomology}, it suffices to extend the above maps $\frak i_{\ins},\frak i_{\out}$ to take care of $y\in\partial X_g$ too. 

Let us first define  $\frak i_{\ins}$. 
Let $y_0\in\partial X_g$ and $c_0\in H_1(\ins( \Psi(y_0)))$. Pick an {\it arbitrary} $\gamma_0$ that represents $c_0$. From the definition of Simon-Smith family, we can find an open neighborhood $U_{y_0,c_0}\subset \overline{X_g}$ of $y_0$  such that for every $y\in U_{y_0,c_0}\cap \overline{X_g}$, $\gamma_0$ also lies in $\ins(\Psi(y))$. In particular, for every $y\in U_{y_0,c_0}\cap X_g$, the element $\frak i_\ins (y,[\gamma_0])\in\Z^g_2$ as defined using the previous paragraphs must be the same for all such $y$. Hence, we may  just define $\frak i_\ins(y_0,c_0)$ as $\frak i_{\ins}(y,[\gamma_0])$. 

For   $\frak i_{\out}$, the definition is analogous.  The continuity of $\frak i_\ins,\frak i_\out$ across $\frak H(\Psi|_{\overline{X_g}})$ follow straightforwardly, and so do the three bullet points of Proposition \ref{prop:embedHomology}.

\section{Proof of Theorem \ref{thm:partialFSigma2}}\label{sect:ProofPartialFSigma2}

Recall that we have the map $${\frak f} :\overline{X_g}\to \Gr(\Z^g_2)\x \Gr(\Z^g_2)$$
that sends each $y$
to the pair of subspaces$$\left( \frak i_{\Gr,\ins}(y, H_1( \ins(\Psi  (y)))),\frak i_{\Gr,\out}(y, H_1( \out(\Psi (y))))\right). $$
Let us focus on the subset $\partial X_g\cap X_{\geq 1}$, where $X_{\geq 1}$ denotes the set of all $y\in X$ such that $\fg(\Psi(y))\geq 1$. Note that $\frak f$   maps $\partial X_g\cap X_{\geq 1}$ into $\Gr^g[1,g-1]$ by Lemma \ref{lem:alexander}.

\begin{rmk}\label{rmk:longestChain}
Recall from \S \ref{sect:subspaceOfGr} that $\Gr^g[1,g-1]$ has a natural abstract simplicial complex structure, induced by the partial order given by inclusion of subspaces. It consists of simplexes of dimension $0,1,...,2g-3$, because the longest chains have  $2g-2$ terms: For example
\begin{align*}
   & (\langle e_1\rangle,\langle e_1\rangle)\leq(\langle e_1\rangle,\langle e_1,e_2\rangle)\leq...\leq (\langle e_1\rangle,\langle e_1,...,e_{g}\rangle)\\
   &\leq (\langle e_1,e_2\rangle,\langle e_1,...,e_{g}\rangle)\leq ...\leq (\langle e_1,...,e_{g-1}\rangle,\langle e_1,...,e_{g}\rangle).
\end{align*}    
\end{rmk}

We claim that Theorem \ref{thm:partialFSigma2} follows readily from the following:
\begin{prop}\label{prop:WeakHomotopyEqui}
    $\frak f$ gives a weak homotopy equivalence from $\partial X_g\cap X_{\geq 1}$ onto its image.
\end{prop}
Indeed, recall the crucial fact (\ref{eqLsigma2NonTri}) that $[\partial \sigma_3]\ne 0$ in $H_{2g-3}(\partial X_g\cap X_{\geq 1})$. Thus, by the above proposition, in  $H_{2g-3}(\frak f(\partial X_g\cap X_{\geq 1}))$ we have $$[\frak f_{\partial \sigma_3}]=[\frak f \circ \partial\sigma_3]=\frak f_*[\partial\sigma_3]\ne 0.$$
Hence, if we pass to the geometric realizations of $\frak f(\partial X_g\cap X_{\geq 1})$ and $\Gr^g[1,g-1]$ (note any subset of $\Gr^g[1,g-1]$ is a subcomplex under the abstract simplicial structure), and consider the (geometric) simplicial complexes $$|\frak f(\partial X_g\cap X_{\geq 1})|\subset|\Gr^g[1,g-1]|,$$ we know that $[\frak f_{\partial \sigma_3}]$ can be represented by a sum $\Sigma$ of {\it multiplicity one} $(2g-3)$-simplexes in $|\frak f(\partial X_g\cap X_{\geq 1})|$. {\it However,  the  space $|\Gr^g[1,g-1]|$   only has cells of dimension  $0,1,...,2g-3$, by Remark \ref{rmk:longestChain}. Thus, $\Sigma$ must also be homologically non-trivial  non-trivial in  the larger space $|\Gr^g[1,g-1]|$.} This  shows $[\frak f_{\partial \sigma_3}]\ne 0$ in $H_{2g-3}(\Gr^g[1,g-1])$, as desired. 

So it remains to prove the above proposition.  

\begin{proof}[Proof of Proposition \ref{prop:WeakHomotopyEqui}]
For any $A=(A_1,A_2)\in \frak f(\partial X_g\cap X_{\geq 1})$, let us  denote
$$V_{A}:=\{(B_1,B_2)\in \Gr^g[1,g-1]:(A_1,A_2)\leq (B_1,B_2)\}\cap  \frak f(\partial X_g\cap X_{\geq 1}).$$
It is easy to check from definition that the collection of such subsets  form a basis of open set for $\frak f(\partial X_g\cap X_{\geq 1})$, under the  subspace topology from $\Gr(\Z^g_2\x\Z^g_2)$. Now, each  such basic set must be contractible, by Lemma \ref{lem:contrac}. Hence, by Theorem \ref{lem:homoEquiv}, to prove the above proposition, it suffices to prove that each such set $\frak f^{-1} ( V_A)$ is contractible.

Recall  from Remark \ref{rmk:chainHomeo} that we have a homeomorphism $(\bz\circ T)|_{\overline{X_g}} $ of the form $\overline{X_g}\to\Omega\to\Sym^{2g+2}_0(S^1)$, where $\Omega\subset\R^{2g+1}$. Below, we will frequently relate $\overline{X_g}$ with $\Sym^{2g+2}_0(S^1)$ using the map $\bz\circ T$.

Using \S \ref{sect:symProd}, we can  naturally identify $\Sym^{2g+2}_0(S^1)$ with a $(2g+1)$-simplex, so that $\Sym^{2g+2}_0(S^1)$ carries a simplicial complex structure. Note, for any element $\balpha\in \Sym^{2g+2}_0(S^1)$, the unique {\it open} cell that contains it has dimension  equal to ``the number of distinct points in $\balpha$,  minus $1$". (Note, by an {\it open cell} we mean the interior of a closed cell.)
Now, let us  define a partial order on $\Sym^{2g+2}_0(S^1)$ as follows. For any $\balpha,\balpha'\in\Sym^{2g+2}_0(S^1)$, we declare that $\balpha \geq\balpha'$ if the {\it open} cell that contains $\balpha'$ is a face of the  {\it open} cell that contains $\balpha$. 

\begin{rmk}\label{rmk:SymOrder}
    We note that this partial order also has the following explicit geometric description: For any $\balpha\in\Sym^{2g+2}_0(S^1)$, and for any $k$ with $1\leq k\leq$ the number of distinct points in $\balpha$, we take any set of $k$ {\it consecutive distinct} points in $\balpha$ (each distinct point may have  multiplicity). We bring these $k$ points  closer and closer together, and at the end let them coincide, while ensuring this path of configurations still lies in $\Sym^{2g+2}_0(S^1)$. Call the final configuration $\balpha'$. Then, for any   configuration  $\balpha'$ obtained this way, we have $\balpha\geq \balpha'$.

    Above, we explained what ``the  open cell that contains $\balpha'$ is a face of the   open cell that contains $\balpha$" means geometrically. Lastly, we recall that (as mentioned in  \S \ref{sect:symProd}) if $\balpha,\balpha'$ belong to the same open cell, then the configuration $\balpha'$ can be obtained from applying some diffeomorphism of $S^1$ to $\balpha$.
\end{rmk}

Let us first fix some $y_0\in  \partial X_g\ \cap X_{\geq 1}$. 
We can use the map $\bz\circ T$ to transport the simplicial complex structure and the partial order $\leq$ from $\Sym^{2g+2}_0(S^1)$ onto $\overline{X_g}$. Define  
$$V_{y_0}:=\{y\in \partial X_g\cap X_{\geq 1}:y\geq y_0\}.$$
In other words, $V_{y_0}$ is the union of all open cells that have ``the open cell that contains $y_0$" as a face.
Since $\partial X_g\cap X_{\geq 1}$ is a subcomplex of $\overline{X_g}$ (under the simplicial structure we borrowed from $\Sym^{2g+2}_0(S^1)$), it is easy to see that $V_{y_0}$ is contractible. Thus, in order to show that $\frak f^{-1}(V_{\frak f(y_0)})$ is contractible, it suffices to show that:
\begin{lem}\label{lem:preimage}
     $V_{y_0}=\frak f^{-1}(V_{\frak f(y_0)})$.
\end{lem}  
\begin{proof}[Proof of Lemma \ref{lem:preimage}]

First, using  Remark \ref{rmk:descriptionSurf} and \ref{rmk:SymOrder}, we observe that  fixing any {\it open} cell $\delta$ of $\overline{X_g}$, for any  $y\in\delta$, the surfaces $\Psi(y)$ are isotopic to each other, and so the subspaces $\frak f(y)$ are all the same.

Let $\delta_0$ be the open cell in $\partial X_g\cap X_{\geq 1}$ that contains $y_0$. We first recall that, that for any loop $\gamma_0\subset M\backslash\Psi(y_0)$ and for any sufficiently small neighborhood $W\subset\overline{X_g}$ of $y_0$, for any $y\in W$, we have $\gamma_0\subset M\backslash \Psi(y)$ too: This just uses the closedness property of  Simon-Smith family. It follows immediately that for any such $y$, we have $\frak f(y)\geq \frak f(y_0)$. As a result, by the definition of the partial order $\leq$ in $\Sym^{2g+2}_0(S^1)$ (and thus on $\overline{X_g}$), and the observation  we made in the last paragraph,  we know that for any $y\in\overline{X_g}$,  $y\geq y_0$ implies $\frak f(y)\geq \frak f(y_0)$. In particular, $V_{y_0}\subset \frak f^{-1}(V_{\frak f (y_0)})$.

Finally, to prove $\frak f^{-1}(V_{\frak f (y_0)})\subset V_{y_0}$, we fix some $y_1\in \frak f^{-1}(V_{\frak f (y_0)})$. Let us use the coordinate system $(r_1,...,r_{2g+1})$  on $\Sym^{2g+2}_0(S^1)$ that we defined in \S \ref{sect:symProd}. Geometrically, these $2g+1$ numbers are the distances between consecutive points in a given configuration:  If some $r_i=0$ then it means  the two corresponding consecutive points actually coincide with each other. Now,  let $(r_1^0,...,r^0_{2g+1})$ be the coordinate for $(\bz\circ T)(y_0)$ and $(r_1^1,...,r^1_{2g+1})$  the coordinate for $(\bz\circ T)(y_1)$.  Then by the definition of the partial order $\leq $ on $\Sym^{2g+2}_0(S^1)$, we see that the inequality $y_0\leq y_1$ holds if and only if the following holds: For each $i=1,...,2g-1$, $r^1_i=0$ implies $r^0_i=0$.

To prove $\frak f^{-1}(V_{\frak  f (y_0)})\subset V_{y_0}$, let us suppose by contradiction that $y_1\notin U_{y_0}$, i.e.  $y_1\ngeq y_0$. Hence,  there is some $i$ such that $r^1_i=0$ but $r^0_i\ne 0$.  Geometrically, this means there exist two consecutive distinct points $\alpha_k,\alpha_{k+1}$ in the configuration $(\bz\circ T)(y_0)$ such that their respective corresponding points in  the configuration $(\bz\circ T)(y_1)$ actually coincide with each other. This implies that, based on the  geometric description of family $\Psi|_{\overline{X_g}}$ in Remark \ref{rmk:descriptionSurf}, {\it  there is some handle in $\Psi(y_0)$ such that the ``corresponding handle" in $\Psi(y_1)$ is actually pinched}. To be more precise,  this means there exists some loop $\gamma_0\subset M\backslash \Psi(y_0)$ (passing through  the aforementioned ``handle") such that we cannot find a  path  in $\fH(\Psi|_{\overline{X_g}})$  that starts at $(y_0,[\gamma_0])$ and ends within the subset $\{y_1\}\x H_1(M\backslash \Psi(y))$. This implies the subset $\frak f(y_1)$ does not include $\frak f(y_0)$, which means $\frak f(y_1)\ngeq \frak f (y_0)$, contradicting our choice that  $y_1\in \frak f^{-1}(V_{\frak f (y_0)})$. This finishes the proof that $\frak f^{-1}(V_{\frak  (y_0)})\subset V_{y_0}$.\end{proof} 

This finishes the proof of Proposition \ref{prop:WeakHomotopyEqui}.\end{proof}


\appendix
\section{Proof of Lemma \ref{lem:alexanderSmooth}}\label{appendix:Linking}
We claim that $H_1(\ins(S))\cong\Z^g_2$. Indeed, first note that, by Alexander duality, $H^2(\ins(S))\cong \tilde H_0(\out(S))=0$, and so by Poincar\'e duality, $H_1(\ins(S),S)=0$. Hence, 
in the long exact sequence for homology applied to the pair $(\ins(S),S)$, the map $H_1(S)\to H_1(\ins(S))$ is surjective, and thus  
$$H_1(\ins(S))\cong H_1(S)/\ker(H_1(S)\to H_1(\ins(S))).$$ By the ``half lives, half dies" theorem, the kernel has rank $\frac 12\rank(H_1(S))=g$, so $H_1(\ins(S))=\Z^g_2.$ Similarly, we can show $H_1(\out(S))=\Z^g_2.$ Then the lemma just follows from the fact that the Alexander duality $H_1(\ins(S))\cong H^1(\out(S))$ is induced by the linking number bilinear form $L.$ Namely, for any $c_1\in H_1(\ins(S))$, its Alexander dual $AD(c_1)\in H^1(\out(S))$ is such that for any $c_2\in H_1(\out(S))$, the pairing $\langle AD(c_1),c_2\rangle$ is equal to $L(c_1,c_2).$ 

\section{Details for \S \ref{sect:PartD} }\label{appendix:countRoots}
Recall that for any $a=[a_0:a_1:...:a_4:1]\in\RP^5$ we defined \begin{equation*} 
\tilde F_a(x_1,x_2,\alpha):=(x_1+a_2)(x_2+a_1)+(a_0-a_1a_2)+\varsigma(a,x)(a_3\cos\alpha+a_4\sin\alpha),
\end{equation*}
where $x=(x_1,x_2,\alpha)$ and $$\varsigma(a,x):=(1-\psi(a,x))\sqrt{1-x_1^2-x_2^2}+\psi(a,x)\sqrt{1-a^2_1-a^2_2}.$$
For \S \ref{sect:PartD}, we just need to  focus on $a\in A_2$ and $(x_1,x_2,\alpha)\in N_1(a_1,a_2)$, where
\begin{align*}
A_2:=\{[a_0:a_1:a_2:a_3:a_4:1]:&\;\epsilon_1(a_1^2+a_2^2)\leq\|(a_0-a_1a_2,a_3,a_4)\|\leq 2\epsilon_1(a_1^2+a_2^2),\\&\;(a_1,a_2)\in\inte(\bD)\}.
    \end{align*}
and
\begin{equation*}
    N_1(a_1,a_2):=\{(x_1,x_2)\in\R^2:\|(x_1,x_2)+(a_2,a_1)\|\leq \eta(a_1^2+a_2^2)\}\x (\R/2\pi\Z).
\end{equation*}
We note that  the functions $\epsilon_1,\eta,\psi$ were introduced in \S \ref{sect:2g+3}.
{\it Without loss of generality, we can assume that fixing $a\in A_2$, $\psi(a,x)$ is constant for $x\in N_1(a_1,a_2)$.}
Our goal in this section is to show that:
\begin{lem}\label{lem:countRootMain}
By choosing the functions $\epsilon_1,\eta>0$  small enough we have: For every $a\in A_2$, there are at most $2$ choices of  $(x_1,x_2,\alpha)\in N_1(a_1,a_2)$  that  solve 
$$\nabla\tilde F_a(x_1,x_2,\alpha)=0\quad \textrm{ and }\quad \tilde F_a(x_1,x_2,\alpha)=0$$
simultaneously, in which   $\nabla$ is only taken with respect to $(x_1,x_2)$.   
\end{lem}

We begin with a lemma, which is from Corollary 7.5 of \cite{ChuLiWang2025GenusTwoI}. To be consistent with the notation there, in this section we use $T$ to denote   trigonometric polynomials. In addition, we adopt the following notation. As usual, we denote $S^1:= \R/2\pi \Z$. Given a real function $f\in C^\infty(S^1)$, let $\ord(f)$ be the number of its real roots in $S^1$, counted with multiplicity. Moreover, given a trigonometric polynomial 
$$T(\alpha)=s_0+s_1\cos\alpha+s'_1\sin\alpha+...+s_k\cos k \alpha+s'_k\sin k\alpha,$$ we define its norm by
$$\|T\|:=\sqrt{s_0^2+s_1^2+s'^2_1+...+s^2_k+s'^2_k}.$$
\begin{lem}[Corollary 7.5 of \cite{ChuLiWang2025GenusTwoI}] \label{Cor_alm trig polyn has zero bd}
For every integer $k>0$, there exists some $\delta_1(k)>0$ with the following property: For any  trigonometric polynomial $T$ of degree $ k$, if a real function $f\in C^\infty(S^1)$ is
such that \[
\|f-T\|_{C^{2k}(S^1)}\leq \delta_1\|T\|\,,
     \]
then $\ord(f)\leq 2k$.
   \end{lem}

Below, we will denote  $r(a):=(1-\sqrt{a_1^2+a_2^2})/2$, and we let $\bD_r(a_1,a_2)$ be the closed  disc in $\R^2$ with center $(a_1,a_2)$ and radius $r.$ {\it Without loss of generality, we can always assume $\eta(a^2_1+a^2_2)<r(a)$.}  Now,  Lemma \ref{lem:countRootMain} immediately follows from the lemma below.
\begin{lem}
    If the functions $\epsilon_1,\eta>0$ are small enough, then $\tilde F_a$ satisfies the following for all $a\in A_2$.
\begin{enumerate}
\item\label{Item_Smalldelta_bx error} There exists $\bx_a \in C^\infty(S^1; \bD_{r(a)}(-a_2, -a_1))$ such that \[
         |\bx_a(\alpha) - (-a_2, -a_1)| < C_1\cdot (|a_3| + |a_4|)\,
       \] 
       for some constant $C_1=C_1(a^2_1+a^2_2)>0$,
       and that for  $(x_1,x_2, \alpha) \in \bD_{r(a)}(-a_2, -a_1)\times S^1$, 
       $$
         \nabla \tilde F_{a}(x_1,x_2, \alpha) = 0\quad\Longleftrightarrow	\quad 
         (x_1,x_2) = \bx_a(\alpha)\,.
       $$
    \item\label{Item_Smalldelta_ZerosBd} Let $\tilde f_{a}\in C^\infty(S^1)$ be given by $\tilde f_{a}(\alpha):= \tilde F_{a}(\bx(\alpha), \alpha)$. Then $\ord(\tilde f_a)\leq 2$, so in particular $\tilde f_a$ has at most $2$ distinct roots.
\end{enumerate}
\end{lem}
\begin{proof}
In the following the constant $C=C(a^2_1+a^2_2)>0$  changes from line to line. 

To prove (\ref{Item_Smalldelta_bx error}), we first recall that, by $a\in A_2$,
\begin{align}\label{eq:a0a3a4Bound}
       a_3^2 + a_4^2 + (a_0-a_1a_2)^2 \leq 4\epsilon_1(a_1^2+a_2^2)^2. 
     \end{align}
     It is easy to check that in $\bD_{r(a)}(-a_2, -a_1)$, 
     \begin{equation}\label{eq:boundVarSigma}
    |\varsigma| + |\nabla \varsigma| + |\nabla^2\varsigma| \leq \Lambda r(a)^{-4}      
     \end{equation}
     for some large absolute constant $\Lambda>0$.
Note $\nabla$ is only taken with respect to $x_1,x_2.$
We also compute the following:
     \begin{align}
       \label{eq:nablaXF} \nabla_x \tilde F_a(x, \alpha) & = (x_2+a_1, x_1+a_2) + (a_3\cos\alpha + a_4\sin\alpha)\nabla\varsigma(x), \\
       \nonumber \nabla_x \tilde F_a(-a_2, -a_1, \alpha) & = O(|a_3|+|a_4|)r(a)^{-4} = O(\epsilon_1(a_1^2+a_2^2)^{1/2})r(a)^{-4},
     \end{align}
     \begin{align}
       \nabla^2_x \tilde F_a(x, \alpha) = \begin{bmatrix}
         0 & 1 \\ 1 & 0 
       \end{bmatrix} + O(|a_3|+|a_4|)r(a)^{-4} \,. \label{Equ_Smalldelta_nabla^2_x F bd}
     \end{align}
     
     To find solutions  for $\nabla \tilde F_{a}(x, \alpha) = 0$ in $(x_1,x_2)$, given $a$ and $\alpha$, let us  consider the map $\varpi:\bD_{r(a)}(0,0)\to\R^2$ given by
\begin{align}
\nonumber \varpi(y_1,y_2)&:= -\begin{bmatrix}
    0 & 1 \\ 1 & 0
       \end{bmatrix} \nabla_x\tilde F_a(-a_2+y_1, -a_1+y_2, \alpha) + (y_1, y_2)\\
\label{eq:varPiMap}    &=\begin{bmatrix}
    0 & 1 \\ 1 & 0
       \end{bmatrix}(a_3\cos\alpha+a_4\sin\alpha)\nabla\varsigma(x).
\end{align}
Note,  $(y_1,y_2)$ is a fixed point of this map if and only if $(y_1-a_2,y_2-a_1)$ solves $\nabla \tilde F_{a}(x_1,x_2, \alpha) = 0$.

Due to the  bounds (\ref{eq:a0a3a4Bound}) and (\ref{eq:boundVarSigma}), we can choose $\epsilon_1$  small enough, depending on $\eta$,  such  that $\varpi$ maps $\bD_{\eta(a_1^2+a_2^3)}(0,0)$ onto itself, and that $\varpi$ is a contraction map. Then by the contraction mapping  theorem, $\varpi$  has some  unique fixed point  $y\in\bD_{\eta(a^2_1+a^2_2)}(0,0)$, which by (\ref{eq:varPiMap}) satisfies
$$|y|< C_1(a_1^2+a_2^2)\cdot (|a_3|+|a_4|)$$
for some constant $C_1(a_1^2+a_2^2)>0$.
so  item (\ref{Item_Smalldelta_bx error}) follows.

     To prove item (\ref{Item_Smalldelta_ZerosBd}),   we first observe that  by applying chain rule to $\nabla_x\tilde F_a(\bx_a(\alpha), \alpha) = 0$ and using (\ref{eq:nablaXF}),  we can derive
     \begin{align}
        \|\bx_a - (-a_2, -a_1)\|_{C^2(S^1)} \leq C(a_1^2+a_2^2)(|a_3|+|a_4|) \,. \label{Equ_Smalldelta_|bx_a|_C^6 bd}
     \end{align}
Now, for each $a\in A_2$, we consider the following trigonometric polynomial of degree $\leq 1$: 
     \begin{align*}
       T_a(\alpha) & := \tilde F_a(-a_2, -a_1, \alpha) \\
       & \ = (a_0-a_1a_2) + \varsigma(-a_2, -a_1)(a_3\cos\alpha +a_4\sin\alpha).
     \end{align*}
     Note, 
     \begin{align*}
       \tilde f_a(\alpha) - T_a(\alpha) & = (\bx_a(\alpha)_1+a_2)(\bx_a(\alpha)_2+a_1) \\
       & + (\varsigma(\bx_a(\alpha))-\varsigma(-a_2, -a_1))(a_3\cos\alpha +a_4\sin\alpha)\,.
     \end{align*}
     Then from this, and using (\ref{Equ_Smalldelta_|bx_a|_C^6 bd}),
\begin{align*}
       \|\tilde f_a - T_a\|_{C^2(S^1)} & \leq C(a_1^2+a_2^2) \Big(\|\bx_a-(-a_2,-a_1)\|_{C^2(S^1)}^2 \\
       & \;\;\qquad\qquad + (|a_3|+|a_4|)\|\bx_a-(-a_2,-a_1)\|_{C^2(S^1)} \Big) \\
       & \leq C(a_1^2+a_2^2)(|a_3|+|a_4|)^2 \\
       & \leq \tilde C(a_1^2+a_2^2)(|a_3|+|a_4|)\|T_a\|,
     \end{align*}
     for some constant $\tilde C(a^2_1+a^2_2)>0$.
     Hence let us take $\epsilon_1(a_1^2+a_2^2)$ small enough  such that $\tilde C(a_1^2+a_2^2)(|a_3|+|a_4|) \leq \delta_1(1)$ (recall by (\ref{eq:a0a3a4Bound}), $a_3,a_4$ are controlled by $\epsilon_1$), where $\delta_1(1)$ is from Lemma \ref{Cor_alm trig polyn has zero bd}. Then item (\ref{Item_Smalldelta_ZerosBd})  follows immediately based on Lemma \ref{Cor_alm trig polyn has zero bd}. \end{proof}

\printbibliography
\end{document}